\documentclass[twoside,11pt]{article}

\usepackage[utf8]{inputenc}

\usepackage[top=1in,bottom=1in,left=1in,right=1in]{geometry}
\usepackage[numbers,sort&compress]{natbib} 
\usepackage{amsfonts, amsmath, amssymb, amsthm, bbm,bm}
\usepackage{comment}

\usepackage{float}

\usepackage{hyperref}
\usepackage{listings}
\usepackage{algorithm2e}
\usepackage[noend]{algpseudocode}

\theoremstyle{definition}

\theoremstyle{plain}
\newtheorem{theorem}{Theorem}
\newtheorem{corollary}{Corollary}
\newtheorem{proposition}{Proposition}

\newtheorem{lemma}{Lemma}

\theoremstyle{remark}

\usepackage[usenames,dvipsnames]{xcolor}
\usepackage[textwidth=1.5cm, textsize=scriptsize]{todonotes}
\usepackage{pdfpages}

\DeclareMathOperator{\Esp}{\mathbb{E}}
\DeclareMathOperator{\Prob}{\mathbb{P}}

\DeclareMathOperator{\Tr}{Tr}

\def\beq{\begin{equation}}
\def\eeq{\end{equation}}
\def\beqn{\begin{eqnarray*}}
\def\eeqn{\end{eqnarray*}}
\def\bitem{\begin{itemize}}
\def\eitem{\end{itemize}}
\def\benum{\begin{enumerate}}
\def\eenum{\end{enumerate}}
\def\bmult{\begin{multline*}}
\def\emult{\end{multline*}}
\def\bcenter{\begin{center}}
\def\ecenter{\end{center}}

\newcommand{\argmin}{\mathop{\mathrm{arg\,min}}}

\def\cA{\mathcal{A}}
\def\cB{\mathcal{B}}

\def\cD{\mathcal{D}}

\def\cG{\mathcal{G}}
\def\cH{\mathcal{H}}
\def\cI{\mathcal{I}}

\def\cK{\mathcal{K}}

\def\cN{\mathcal{N}}

\def\cP{\mathcal{P}}

\def\cR{\mathcal{R}}

\def\cT{\mathcal{T}}

\def\cV{\mathcal{V}}

\def\cY{\mathcal{Y}}
\def\cZ{\mathcal{Z}}

\def\nnu{u}

\def\nx{x}

\def\bA{\mathbf{A}}

\def\bC{\mathbf{C}}
\def\bD{\mathbf{D}}

\def\bI{\mathbf{I}}

\def\bN{\mathbf{N}}

\def\bS{\mathbf{S}}

\def\1{{\mathbf 1}}

\def\bbE{\mathbb{E}}

\def\bbN{\mathbb{N}}

\def\bbP{\mathbb{P}}

\def\bbR{\mathbb{R}}

\renewcommand{\P}{\operatorname{\mathbb{P}}}

\newcommand{\Esper}[2][]{\mathbb{E}_{#1} \left[ #2 \right] }
\newcommand{\Proba}[2][]{\mathbb{P}_{#1} \left( #2 \right) }

\newcommand{\Log}[1]{\log \left( #1 \right) }

\newcommand{\norm}[1]{\left\lVert#1\right\rVert}
\newcommand{\floor}[1]{\left\lfloor#1\right\rfloor}
\newcommand{\ceil}[1]{\left\lceil#1\right\rceil}
\newcommand{\abs}[1]{\left\lvert#1\right\rvert}
\newcommand{\poubelle}[1]{}

\DeclareMathOperator{\Noise}{\varepsilon}

\newcommand{\CP}{\mathcal T} 
\newcommand{\CPd}{\mathcal T^*} 
\newcommand{\SCP}{\mathcal S} 
\newcommand{\SCPd}{\mathcal S^*} 

\newcommand{\sm}{\bar s} 
\newcommand{\Cg}{\mathcal C}
\newcommand{\Cgc}{\mathrm C}

\newcommand{\Thetarv}{\tilde{\Theta}} 
\newcommand{\Thetarvs}{\tilde \Theta_{(1)}} 
\newcommand{\Thetarvd}{\tilde{\Theta}_{(2)}} 
\newcommand{\Thetarvm}{\tilde{\Theta}_{(3)}} 
\newcommand{\pis}{\pi_1} 
\newcommand{\pid}{\pi_2} 
\newcommand{\pim}{\pi_3} 

\newcommand{\Lrv}{\nu } 
\newcommand{\Crvs}{a } 
\newcommand{\Crvd}{b } 
\newcommand{\Crvm}{c} 

\newcommand{\tg}{\bar{\tau}}  
\newcommand{\rg}{\bar{r}} 

\newcommand{\tgD}{\bar{\tau}^{(\mathrm{d})}}  
\newcommand{\tgSp}{\bar{\tau}^{(\mathrm{s})}}  

\newcommand{\rgD}{\bar{r}^{(\mathrm{d})}} 
\newcommand{\rgSp}{\bar{r}^{(\mathrm{s})}} 

\newcommand{\xDense}{ \xi^{(\mathrm{d})}} 
\newcommand{\xSparse}{ \xi^{(\mathrm{s})}} 

\newcommand{\xPN}{ \xi^{(\mathrm{p})}} 
\newcommand{\xBJ}{ \xi^{(\mathrm{BJ})}} 

\newcommand{\Tg}{T}
\newcommand{\TD}{T^{(\mathrm{d})}} 
\newcommand{\psiD}{\Psi^{(\mathrm{d})}} 
\newcommand{\tD}{x^{(\mathrm{d})}} 

\newcommand{\Tbj}{T^{(\mathrm{BJ})}}
\newcommand{\tbj}{x^{(\mathrm{BJ})}}
\newcommand{\Tsparse}{T^{(\mathrm{s})}}

\newcommand{\tsup}{x^{(\mathrm{p})}}
\newcommand{\Ssup}{\Psi^{(\mathrm{p})}} 
\newcommand{\Tsup}{T^{(\mathrm{p})}}

\newcommand{\Pginf}{\bar {\mathcal P}( u)} 
\newcommand{\Pginfrs}{\bar {\mathcal P}(u,r,s)}

\newcommand{\Csto}{c_0}
\newcommand{\alphabj}{\delta}

\newcommand{\SetS}{S}
\newcommand{\Choose}{\bar C} 

\newcommand{\lo}{\gamma_r}  

\newcommand{\xsgD}{{\xi}^{(\mathrm{d})}} 
\newcommand{\xsgSp}{{\xi}^{(\mathrm{p})}} 

\newcommand{\TDsg}{{T}^{(\mathrm{d})}}  
\newcommand{\psiDsg}{{\Psi}^{(\mathrm{d})}}  
\newcommand{\tDsg}{{x}^{(\mathrm{d})}} 

\newcommand{\TDmsg}{{T}^{(\mathrm{p})}}
\newcommand{\psiDmsg}{{\Psi}^{(\mathrm{p})}}
\newcommand{\tDmsg}{{x}^{(\mathrm{p})}}

\newcommand{\Cxi}{{\bar c}_{\mathrm{conc}}}
\newcommand{\thresh}{{\bar c}_{\mathrm{thresh}}}
\newcommand{\Diffcarremoy}[2]{\norm{\overline{\Noise}_{#1,+#2} - \overline{\Noise}_{#1,-#2} }^2}
\newcommand{\Diffcarre}[2]{\norm{\Noise_{#1,+#2} - \Noise_{#1,-#2} }^2}

\newcommand{\yS}{\bar{y}^{(S)}}
\newcommand{\NS}{\bar{\Noise}^{(S)}}
\newcommand{\tS}{\bar{\theta}^{(S)}}

\newcommand{\NSk}{\bar{\Noise}^{(S_k)}}
\newcommand{\tSk}{\bar{\theta}^{(S_k)}}

\def\NoiseBoundsg<#1,#2,#3>{\left( \sqrt{#1 \Log{\frac{#2}{#3}}} + \Log{ \frac{#2}{#3}}  \right)}

\def\proscal<#1,#2>{\langle #1~,~#2\rangle}
\title{Optimal multiple change-point detection for high-dimensional data}

\def\NoiseBound<#1,#2,#3,#4>{\left( \sqrt{#1 \Log{\frac{#2}{#3#4} }} + \Log{\frac{#2}{#3#4} }  \right)}

\def\NoiseBoundsparse<#1,#2,#3,#4, #5>{ \left(#5\Log{
\frac{#1}{#5^2} \Log{\frac{#2}{#3 #4} }}  + \Log{\frac{#2}{#3 #4}}\right)}

\def\NoiseBoundd<#1,#2>{\left( \sqrt{#1 \Log{#2}} + \Log{ #2} \right)}

\usepackage{authblk}

\author[1]{Emmanuel Pilliat}
\author[2]{Alexandra Carpentier}
\author[3]{Nicolas Verzelen}

\affil[1]{Université de Montpellier, Montpellier, France}
\affil[2]{Institut für Mathematik, Universität Potsdam, Germany}
\affil[3]{INRAE Montpellier, Montpellier, France}

\usepackage{hyperref} 
\hypersetup{
    colorlinks=true,
    linkcolor=blue,
    filecolor=magenta,
    urlcolor=cyan,
}
\usepackage[capitalise,nameinlink]{cleveref} 
\crefformat{equation}{(#2#1#3)}
\crefformat{proposition}{Proposition #2#1#3}
\crefformat{theorem}{Theorem #2#1#3}
\crefformat{lemma}{Lemma #2#1#3}
\crefformat{corollary}{Corollary #2#1#3}
\crefformat{section}{Section #2#1#3}
\crefformat{algorithm}{Algorithm #2#1#3}
\crefformat{appendix}{Appendix #2#1#3}
\crefformat{figure}{Figure #2#1#3}

\begin{document}
\sloppy

\lstset{language=Python}
\maketitle

\begin{abstract}
  This manuscript makes two contributions to the field of change-point detection. In a general change-point setting, we provide a generic algorithm for aggregating local homogeneity tests into an estimator of change-points in a time series. Interestingly, we establish that the error rates of the collection of tests directly translate into detection properties of the change-point estimator. This generic scheme is then applied to various problems including covariance change-point detection, nonparametric change-point detection and sparse  multivariate mean change-point detection. For the latter, we derive minimax optimal rates that are adaptive to the unknown sparsity and to the distance between change-points when the noise is Gaussian. For sub-Gaussian noise, we introduce a variant that is optimal in almost all sparsity regimes.
\end{abstract}

\section{Introduction}
\label{sec:introduction}

Change-point detection has a long history since the seminal work of Wald \cite{Wald1945} that lead to flourishing lines  (see \cite{Niu2016,truong2020selective} for recent surveys). Earlier contributions focused on the problems of detecting and localizing change-points in a univariate time series. Spurred by applications in genomics~\cite{olshen2004circular} and finance, there has been a recent trend in the literature towards the analysis of more complex time series for instance in a high-dimensional linear space~\cite{jirak2015uniform} or even belonging to a non-Euclidean space~\cite{chu2019asymptotic}.

\medskip

In this work, we study high-dimensional time series whose mean may change possibly on a few number of coordinates. See the introduction of~\cite{Wang2018} for an account of possible applications and practical motivations. In particular, we build a procedure which is able to detect and localize change-points under minimal assumptions on the height of these change-points. Along the way towards this optimal procedure, we define and analyze a scheme for general change-point problems that aggregates a collection of local tests into an estimator change-points. This generic scheme is of independent interest and easily allows to derive optimal change-point procedure in other complex settings such as covariance change-points problems or nonparametric change-point problems. 
In this introduction, we first describe this generic scheme before turning to our results in high-dimensional sparse change-point detection and finally discussing other applications.

\subsection{General change-point setting}
\label{sec:setting_general}
In the most general form of a change-point problem, we consider a random sequence $Y= (y_1,y_2,\ldots, y_n)$ in some measured space $\cY^n$ and, for $t=1,\ldots, n$, we write $\bbP_t$ for the marginal distribution of $y_t$. We are also given  a functional $\Gamma$ mapping the probability distribution $\bbP_t$ to some space $\cV$. Then, the purpose of change-point detection is to detect changes in the sequence $(\Gamma(\bbP_1),\Gamma(\bbP_2),\ldots, \Gamma(\bbP_n))$ in $\cV^n$ and to estimate the positions of these changes. This setting is really general and does not require that the random variables
$(y_t)$ are independent.
\medskip

Let us shortly explain how this general framework encompasses most offline change-point detection problems. In the Gaussian mean univariate change-point setting, we have $\cY=\bbR$, the distribution $\bbP_t$ corresponds to the normal distribution with mean $\theta_t\in \bbR$ and variance $\sigma^2$ and $\Gamma(\bbP_t)=\theta_t$.
In the (heteroscedastic) mean univariate change-point problem, the distribution $\bbP_t$ is not necessarily Gaussian and, in particular, the variance of $y_t$ is allowed to vary with $t$. Still, one is only interested in detecting variations of $\Gamma(\bbP_t)= \int x d\bbP_t= \bbE[y_t]$. By contrast, in the {\it variance} univariate change-point problems, one focuses on changes in the  variance of $y_t$. This can be done by taking  $\Gamma(\bbP_t)= \int x^2d\bbP_t - [\int xd\bbP_t]^2 = \mathrm{Var}(y_t)$. If  one is interested in possibly nonparametric changes in the distributions, then the functional $\Gamma$ is simply taken to be the identity map. In semi-parametric quantile change-point detection~\cite{jula2021multiscale}, the univariate distributions $\mathbb{P}_t$ can be arbitrary whereas $\Gamma(\mathbb{P}_t)$ is a quantile of $\mathbb{P}_t$.

\medskip

To further formalize the change-point detection problem in the sequence  $(\Gamma(\bbP_1),\Gamma(\bbP_2),\ldots, \Gamma(\bbP_n))$, we define an integer $0 \leq K\leq n - 1$ and a vector of integers $\tau = (\tau_1, \dots, \tau_{K})$ satisfying $1 = \tau_0 < \tau_1 < \dots < \tau_K < \tau_{K+1} = n+1$ such that $\Gamma(\bbP_t)$ is constant over each interval $[\tau_k,\tau_{k+1}-1]$ and $\Gamma(\bbP_{\tau_k-1})\neq \Gamma(\bbP_{\tau_k})$. Hence, $\tau_k$ corresponds to the \emph{position} of the $k^{th}$ change-point. We shall often refer to $\tau_k$ as a \emph{change-point}. Equipped with this notation, we are interested in building an estimator $\hat\tau = (\hat \tau_1, \dots, \hat \tau_{\hat K})$ of $\tau$ from the time series $Y$. Here, $\hat \tau_{1}, \dots, \hat \tau_{\hat K}$ correspond to the \emph{estimated change-points} of $\tau$ and $\hat K$ to the number of the estimated change-points.

\subsubsection{Desirable  Guarantees of an estimator.} Before describing the generic scheme for estimating $\tau$, let us first formalize the desired properties of a good change-point procedure. Informally, the primary objectives are to detect most if not all change-points while estimating no (or at least very few) spurious change-points.

Regarding the latter objective, it is usually required that the number of change-points $K$ is not overestimated by $\hat{\tau}$. Here, we require a slightly stronger local property introduced in~\cite{verzelenoptimal_change_point}. An estimator $\hat{\tau}$ of size $\hat{K}$ is said to detect no spurious change-points  {\bf (NoSp)}  if
\begin{equation}\label{eq:def:nosp}
\left\{
	\begin{array}{cc}
\abs{\{\hat \tau_{k'},\, 1\leq  k' \leq \hat K\} \cap \left[ \tau_{k} - \frac{\tau_{k} - \tau_{{k}-1}}{2},\tau_{k} + \frac{\tau_{k+1} - \tau_{{k}}}{2} \right]} \leq 1\ , &\quad \text{ for all }1\leq k\leq K \ ; \\
		 \{\hat \tau_{k'}, 1\leq k' \leq \hat K\} \subset \left[ \tau_{1} - \frac{\tau_{1} - 1}{2},\tau_{K} + \frac{n+1 - \tau_{{K}}}{2} \right]\ .&
	\end{array}
\right.
\end{equation}
The second condition simply ensures that no change-point is estimated near the boundaries of the time series. The first condition entails that, for each change-point $\tau_k$ there is at most one estimated change-point $\hat \tau_k$ in the interval $\left[ \tau_{k} - (\tau_{k} - \tau_{{k}-1})/2,\tau_{k} + (\tau_{k+1} - \tau_{{k}})/2 \right]$. In other words, ({\bf NoSp}) requires that, on each sub-interval, the number of change-points is not overestimated.

Let us now formalize the objective of detecting the change-points. In this work, we consider as in \cite{verzelenoptimal_change_point} realistic settings where some change-points are so close or their  heights are so small that they are impossible to detect. As a consequence,  we can only hope to detect the subset of \emph{significant} change-points.
In what follows, we define a subset $\cK^* \subset [K]$ of change-point indices that correspond to \emph{significant} change-points. Obviously, the significance of a particular change-point is relative to the problem under consideration - data distribution, nature of change-points - and the definition is  problem dependent. As an example, we define in the next subsection the suitable notion of energy and significance of a change-point in  the mean multivariate change-point setting. In Section~\ref{sec:othermodels}, we formalize this notion for covariance and univariate nonparametric change-point problems. In light of this discussion, the second guarantee we aim for is to {\bf detect} all significant change-points. A change-point $\tau_k$ is said to be detected if there is at least one estimated change-point $\hat{\tau}_l$  in the interval $\left[ \tau_{k} - (\tau_{k} - \tau_{{k}-1})/{2},\tau_{k} + (\tau_{k+1} - \tau_{{k}})/2 \right]$. Equivalently, this means that at least one of the estimated change-points is closer to $\tau_k$ than to any other true change-point.

Aside from  ({\bf NoSp}) and ({\bf detect}) properties, one may additionally aim at localizing the change-points as well as possible -- see the discussions in~\cite{wang2018optimal}. Given a specific change-point $\tau_k$ detected by an estimator $\widehat{\tau}$, its localization error $d_{H,1}(\widehat{\tau},\tau_k)$ is defined by
\[
d_{H,1}(\widehat{\tau},\tau_k)= \min_{l=1,\ldots, |\widehat{\tau}|}|\widehat{\tau}_l-\tau_k| \ ,
\]
which is the smallest distance between $\tau_k$ and one of the estimated change-points. While this work mainly focused on the detection problem, we shall also provide localization bounds along the way.

\subsubsection{A generic roadmap for change-point detection.}
In this manuscript, our first contribution is a generic procedure for aggregating a collection of tests into an estimator $\widehat{\tau}$ of $\tau$. For two positive integers $(l,r)$, we consider the time interval $[l-r,l+r)$.  Suppose we are given a collection $\cG$ of such $(l,r)$. For each $(l,r)\in \cG$, we are also given a homogeneity test $\Tg_{l,r}$  of the null hypothesis $\cH_0$: \{($\Gamma(\bbP_t))$ is constant over the segment $[l-r,l+r)$\}. This hypothesis is equivalent to the absence of any change-point on the interval $(l-r,l+r)$. Given such a collections of homogeneity tests $(\Tg_{l,r})$, $(l,r)\in \cG$, we build in this manuscript an  estimator $\widehat{\tau}$ that satisfies the following properties.
  If the multiple testing procedure does not reject any true null hypothesis (no false positives), then $\widehat{\tau}$ does not estimate any spurious change-point, that is, it satisfies ({\bf No Sp}). Furthermore,  any change-point $\tau_k$ that is detected by some test $\Tg_{\bar \tau_k,\overline{r}_k}$, where $\bar \tau_k$ is close enough to $\tau_k$ and  $\overline{r}_k$ is small enough is {\bf detected} by the estimator $\widehat{\tau}$. In other words, we establish a completely generic result that translates properties of the multiple testing procedure into {\bf detection} properties.
Thus, the construction of a change-point procedure boils down to building a suitable multiple testing procedure $(\Tg_{l,r})$, $(l,r)\in \cG$ whose family-wise error rate (FWER) is controlled, while being able to detect all the significant change-points. In turn, this allows us to reduce the problem of change-point detection under minimal distance between the change-points to the well-established field of minimax testing.

\subsubsection{Related Work and possible applications.}

In the last years, there has been a growing interest into the extension of univariate mean change-point procedures such as wild binary segmentation (WBS)~\cite{fryzlewicz2014wild} to other problems such as covariance change-point~\cite{wang2017optimal}, network change-point~\cite{wang2018optimal}, or nonparametric change-point~\cite{padilla2019optimal1}. 
For each of these problems (and for others), it turns out that the general ideas of WBS can be instantiated. However, for each setting, the proofs need to be fully adapted in a case by case manner. Besides, the resulting procedures are only optimal up to logarithmic terms.

Recently, Chan and Chen~\cite{chan2017multi} and Kov{\'a}cs et al.~\cite{kovacs2020seeded} have introduced bottom-up aggregation procedures for mean change-point segmentation (see also~\cite{kovacs2020optimistic} for localization improvements). Moreover, Kov{\'a}cs et al.~\cite{kovacs2020seeded,kovacs2020optimistic} illustrate the numerical performances to other change-point models, such as graphical models or multivariate mean-change point models. In fact, one may extend their procedures to generic problems, but the theoretical guarantees are only provided for univariate models and it remains unclear whether one can extend them beyond very specific cases.

In contrast, it is quite straightforward to adapt our generic procedure to any new setting once suitable homogeneity multiple tests have been crafted.  As the most prominent example, we consider the  sparse high-dimensional mean change-point detection and establish the optimality of our procedure -- see the next subsection for details. In Section~\ref{sec:othermodels}, we also handle the covariance change-point detection and the univariate nonparametric change-point detection problems. In each case, we pinpoint the first tight minimal conditions for detection.

Besides, we could apply our strategy to other problems such changes in auto-regressive models~\cite{wang2019localizing}, changes in the inverse covariance matrix of $y_i$~\cite{gibberd2017multiple,kovacs2020seeded} or changes in a high-dimensional regression model~\cite{rinaldo2021localizing}. All such change-point problems can be addressed through the construction and careful analysis of two-sample tests for auto-regressive models, inverse covariance matrices, and linear regression models respectively. Similarly, we can build  Kernel change-point  procedures~\cite{Arlot2019,Garreau2018} from kernel two-sample tests~\cite{gretton2012kernel}.



\subsection{Sparse Multivariate Change-point Setting}
\label{sec:setting}
As explained above, our primary application of our generic scheme is the multivariate mean change-point detection problem with sparse variations where one observes a time series  $Y= (y_1, \dots, y_n) \in \bbR^{p \times n}$ with unknown means  $\Theta = (\theta_1, \dots, \theta_n) \in \bbR^{p \times n}$ so that we have the decomposition
\beq\label{eq:model_multivariate}
	y_t=\theta_t+\Noise_t\quad \quad t=1,\ldots,n\enspace,
\eeq
where the noise matrix $\Noise= (\Noise_1,\ldots, \Noise_n)$ is made of independent and mean zero random vectors of size $p$. In this manuscript, we make two distributional assumptions on the noise. Either we suppose that all random vectors $\Noise_i$ follow independent normal distribution with variance $\sigma^2 \bI_p$ (see \cref{sec:multi_scale_change_point_estimation_in_gaussian_noise}) or that the components of $\Noise_i$ follow independent sub-Gaussian distributions with variance $\sigma^2$ (see \cref{sec:multi_scale_change_point_estimation_in_sub_gaussian_noise}). In either case, we assume that $\sigma^2$ is known.

\medskip

Here, we are interested in the variations of the {\it mean} vector $\theta_t$ so that, relying on the formalism of the previous subsection, we have $\Gamma(\bbP_t)= \theta_t$. Considering the vector of change-points  $\tau = (\tau_1, \dots, \tau_{K})$, we can define $K+1$ vectors $\mu_0, \dots, \mu_{K}$ in $\bbR^p$ satisfying $ \mu_k \neq \mu_{k+1}$ for all $k = 0, \dots, K-1$ such that
$$ \theta_t = \sum_{k = 0}^{K} \mu_k \1_{\tau_{k} \leq t < \tau_{k+1}}\enspace .$$

Equivalently, $\mu_k$ is the constant mean of $y$ over the interval $[\tau_k, \tau_{k+1}-1]$. The difference $\mu_k - \mu_{k-1}$ in $\bbR^p$ measures the variation of $\Theta$ at the change-point $\tau_k$ and can possibly have many null coordinates. In this possibly sparse multi-dimensional setting, the significance of a change-point is measured through three quantities $\Delta_k$, $r_k$, and $s_k$. First, the \emph{height} $\Delta_k$ of the change-point $\tau_k$ is defined as the Euclidean norm of the signal difference. The \emph{length} $r_k$ of the change-point $\tau_k$ is the minimal distance from $\tau_k$ to another change-point, $\tau_{k-1}$ or $\tau_{k+1}$. More precisely,
\beq\label{eq:definition_Delta_k_r_k}
\Delta_k = \norm{\mu_k - \mu_{k-1}}\ ; \quad \quad r_k = \min(\tau_{k+1} - \tau_{k}, \tau_k - \tau_{k-1} ) \enspace .
\eeq
As a simple example,  \Cref{fig:setting} depicts a one dimensional piece-wise constant sequence $\Theta$ with $3$ change-points illustrating the setting presented above. In the univariate change-point literature (e.g. \cite{fryzlewicz2014wild,fryzlewicz2018tail,cho2021data})  the height and the length of a change-point characterize the significance of a change-point. In the multivariate setting, where the change-points can be sparse, meaning the number of non null coordinates of the vector $\mu_k - \mu_{k-1}$ is possibly small, one also considers the \emph{sparsity}  $s_k$ of change-point $\tau_k$, defined by
\beq\label{eq:definition_s_k}
s_k = \norm{\mu_k - \mu_{k-1}}_0 \enspace ,
\eeq
where, for any $v \in \bbR^p$, $\norm{v}_0 = \sum_{1\leq i \leq p}\1\{v_i\neq 0\}$.

\begin{figure}[ht]
	\centering
	\includegraphics[scale = 0.15]{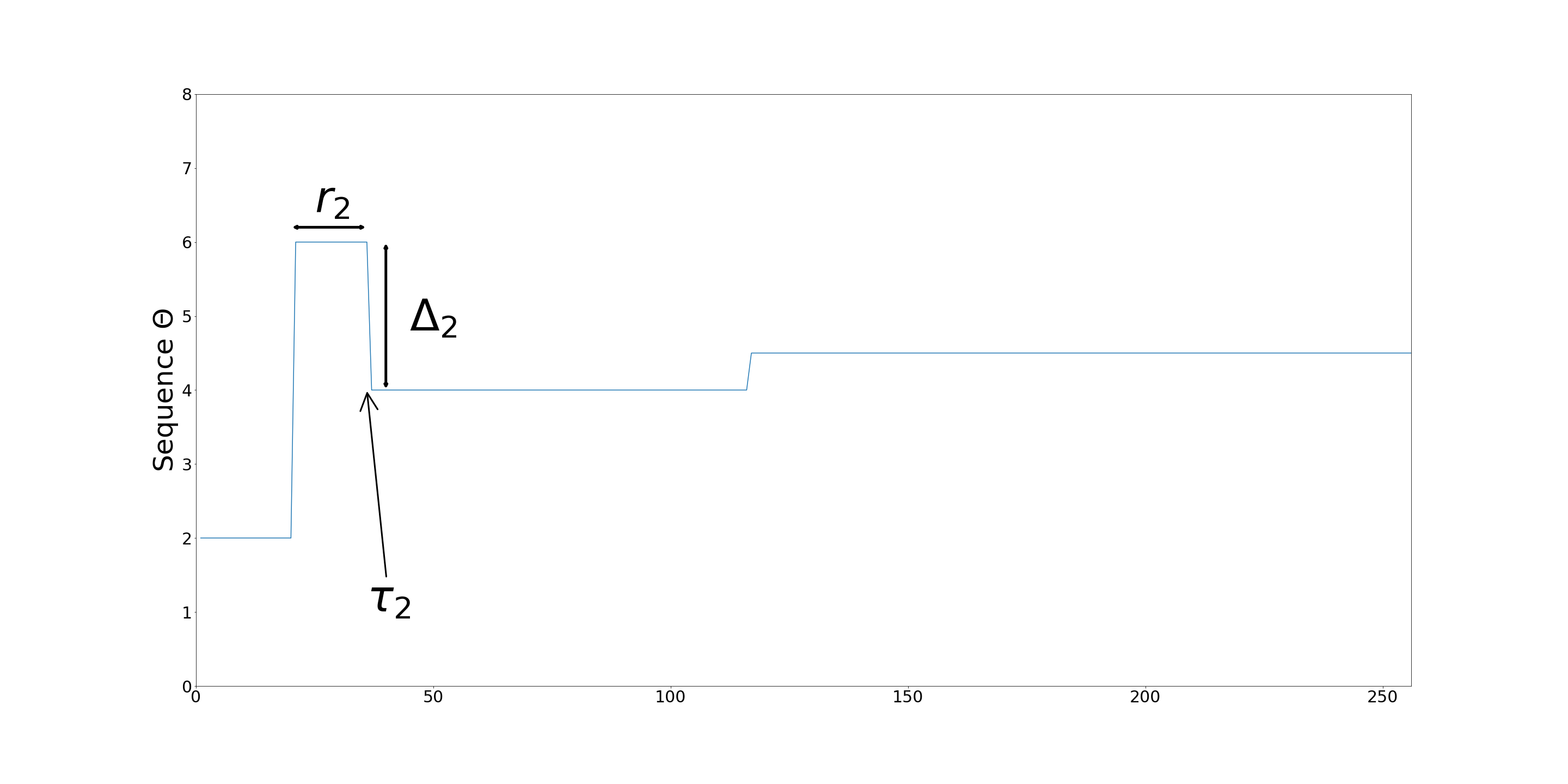}
	\caption{\label{fig:setting}An example of a piece-wise constant sequence $\Theta$ with $3$ change-points and $p = 1$.}
\end{figure}

\subsubsection{Two-sample tests and CUSUM statistics} Our objective is to detect and recover positions $(\tau_k)_{k \leq K}$ under minimal conditions on the change-point height $\Delta_k$, change-point length $r_k$ and sparsity $s_k$. In view of the generic change-point procedure discussed in the previous subsection, this mainly boils down to building suitable tests of the assumptions \{$\Theta$ is constant over $[l-r,l+r)$\} versus \{$\Theta$ is not constant on this segment\}. Following the literature on binary and wild binary segmentation, we consider the CUSUM statistic
\[
\bC_{l,r}(Y) = \sqrt{\frac{r}{2 \sigma^2}}\left(\frac{1}{r}\sum_{i = l}^{l+r-1}y_i - \frac{1}{r}\sum_{i = l-r}^{l-1}y_i\right)\enspace .
\]
This statistic computes the normalized difference of empirical mean of $y_i$ on $[l-r,l)$ and $[l,l+r)$. If the noise is Gaussian and if $\Theta$ is constant on $[l-r,l+r)$, then $\bC_{l,r}(Y)$ simply follows a standard $p$-dimensional normal distribution. To simplify, consider a specific instance of our testing problem where we want to test whether \{$\Theta$ is constant over $[l-r,l+r)$\} versus \{$\Theta$ contains exactly one change-point at $l$ on the segment $[l-r,l+r)$\}. This corresponds to a  two-sample mean testing problem, for which the CUSUM statistic $\bC_{l,r}(Y)$ is a sufficient statistic if the noise is Gaussian. Then, given $\bC_{l,r}(Y)$, one wants to test whether its expectation is $0$ (no change-point on $[l-r,l+r)$) versus its expectation is non-zero
but is $s$-sparse for some unknown $s$. This classical detection problem is well understood \cite{jin2004} and it is well known that a combination of a $\chi^2$-type test with a higher-criticism-type test is optimal. Here, the challenge stems from the fact that we do not want to perform a single such test, but a large collection of tests over a collection of $(l,r)\in \mathcal{G}$.

\subsubsection{Our contribution} As usual in the mean change-point literature, we consider  the energy $r_k\Delta_k^2$ of the change-point $\tau_k$. Up to a possible factor in $[1/2,1]$, $r_k\Delta_k^2$ is the square distance between $\Theta$ and its projection on the space of vectors $\Theta'$  with change-point at $(\tau_1,\ldots, \tau_{k-1},\tau_{k+1},\ldots,\tau_K)$ --see e.g. \cite{verzelenoptimal_change_point} for a discussion in the univariate setting. In other words, the energy $r_k\Delta_k^2$ characterizes the significance of the change-point $\tau_k$. In Section~\ref{sec:multi_scale_change_point_estimation_in_gaussian_noise}, we introduce a multi-scale change-point detection procedure detecting any change-point $\tau_k$ whose energy is higher, up to a numerical constant, than $ \sigma^2s_k\log(1+\tfrac{\sqrt{p}}{s_k}\sqrt{\log(n/r_k)}) + \sigma^2\log(n/r_k)$. This result is valid for arbitrary length $r_k$ and sparsity $s_k$, and does not require the knowledge of these two quantities. In summary, our procedure does not estimate any spurious change-point ({\bf NoSp}) and {\bf detects} all the
change-points whose energy are higher than the latter threshold.
In Section~\ref{sec:lower}, we establish that, as soon as the unknown number $K$ of the change-points is larger than $1$, the  condition $\sigma^2s_k\log(1+\tfrac{\sqrt{p}}{s_k}\sqrt{\log(n/r_k)}) + \sigma^2 \log(n/r_k)$ on the energy is tight with respect to $n$, $p$, $r_k$ and $s_k$, in the sense that no procedure achieving ({\bf NoSp}) is able to detect with high probability a change-point whose energy is smaller (up to some constant) than the latter threshold. In Section~\ref{sec:multi_scale_change_point_estimation_in_sub_gaussian_noise}, we consider the more general setting where the  noise is $L$-sub-gaussian with known variance, and we establish a similar result to the Gaussian case up to a logarithmic loss in some regimes. Finally,  we illustrate in Section~\ref{sec:numerical} the behavior of our procedure on numerical experiments.

\subsubsection{Related work}
For dense change-points ($s_k=p$) but with unknown covariance for the noise, Wang et al.~\cite{wang2019inference} (see also \cite{wang2020dating}) study the behavior of a procedure based on $U$-statistics of the CUSUM. Jirak~\cite{jirak2015uniform} and Yu and Chen~\cite{yu2017finite} introduce binary segmentation procedures based on the $l_{\infty}$ norm of the CUSUMs. Although those work explicitly characterize the asymptotic distribution of the test statistics and, for some of them, allow temporal dependencies in the data, the corresponding energy requirements for change-point detection are either not studied or turn out to be suboptimal.

Closest to our work, Chan and Chen \cite{chan2017multi} study a bottom-up approach to detect change-points of a Gaussian multivariate time series in an asymptotic setting. More specifically, the authors consider an asymptotic regime where the size of the time series is exponential in the dimension: $n = e^{p^\zeta}$ with $\zeta \in  (0,1)$. The authors also assume that the number $K$ of change-points remains finite when $n,p \to \infty$ and that the minimal sparsity $s$ of these change-points is polynomial is $p$.  In this specific regime, their procedures provably recover change-points under a near-minimal (up to logarithmic factors with respect to $n$) condition on the energy. In contrast, our results provide non-asymptotic and tight results for all scaling with respect to $n$ and $p$, allow for arbitrarily large number $K$ of change-points and allow for the presence of non-significant change-points.
In the same specific asymptotic setting,~\cite{hu2021sparsity} introduce a so called score test statistic used in a change-point detection procedure which is shown to achieve the same performance as \cite{chan2017multi} in the gaussian model but also handle Poisson observations.

Recently, Liu et al.~\cite{liu2019minimax} have characterized the optimal detection rate of a possibly sparse change-point in the specific case where there is at most one change-point, but the optimal rates are significantly slower in the multiple change-point setting. See also \cite{EnikeevaHarchaoui2019} and~\cite{dette2018relevant} for earlier results. Wang and Samworth~\cite{Wang2018} have proposed the INSPECT method based on sparse projection to handle sparse change-points, but INSPECT provably detects the change-points under strong assumption on the energy; see Section~\ref{sec:multi_scale_change_point_estimation_in_gaussian_noise} for a precise comparison.

In the univariate setting ($p=1$), minimal energy requirements for change-point detection are well understood~\cite{frick2014multiscale,fryzlewicz2018tail,wang2020univariate, verzelenoptimal_change_point} and are nearly achieved by a wide range of procedures including penalized least-square and  multi-scale tests methods.



\section{A Generic algorithm for multiscale change-point detection on a grid}
\label{sec:meta_algorithm_for_multi_scale_change_point_detection_on_a_grid}
In this section, we study the problem of change-point detection in the general setting defined in \Cref{sec:setting_general}. We introduce a bottom-up algorithm that aggregates a collection of homogeneity tests, performed at many positions, and for many scales, of our data. Then, we establish that, under some conditions on these tests, the procedure detects significant change-points.

\subsection{Grid and multiscale statistics}

Since our purpose is to translate a collection of local tests $\Tg = (\Tg_{l,r})_{(l,r) \in \cG}$  indexed by a grid $\cG$ into a change-point detection procedure, we first need to formalize what we mean by a grid.
Henceforth, we call a grid $\cG$ of $[n]$ a collection of locations and scales where a scale $r$ is a positive integer smaller or equal to $\floor{n /2}$ and a location  $l$  is an integer between $r+1$ and $n-r$. This couple $(l,r)$ refers to the segment $[l-r,l+r)$ centered at $l$ and with radius $r$.  Formally, $\cG$ is therefore a subset of $J_n = \left\{(l,r):~ r = 1,\dots, \floor{\frac n 2} \text{ and } l = r+1, \dots, n - r + 1\right\}$. Given a grid $\cG$, we call $\cR$ its collection of scales, that is
$\cR = \{r:~\exists l \text{ s.t. } (l,r) \in \cG\}$. Finally, for a scale $r\in \cR$,  $\cD_r$ stands for the  corresponding collection of locations, that is $\cD_r = \{l:~ (l,r) \in \cG\}$.
Although we do not make any assumption on the grid $\cG$ for the time being, we will mainly consider two specific grids in this section: the {\bf complete} grid $\cG_F=J_n$ and the {\bf dyadic} grid $\cG_D$ defined by $\cR=\{1,2, 4,\dots, 2^{\lfloor \log_2(n)\rfloor-1}\}$, $\cD_1 = [2,n]$, and
\beq\label{eq:def_dyadic}
\cD_r = \left\{r+1, 3\lfloor r/2\rfloor +1, 4\lfloor r/2\rfloor+1, \dots, \Big(\frac{n}{\lfloor r/2\rfloor}-2\Big)\lfloor \frac{r}{2}\rfloor  + 1,n-r+1\right\} \quad \mbox{ for } r\in \cR \setminus \{1\}\ .
\eeq
See \Cref{figure_diadic_grid} for a visual representation of the dyadic grid. At some points, we shall also mention $a$-adic grids $\cG_a$. For any $a\in (0,1)$, $\cG_a$ is defined by $\cR=\{1,\lfloor a^{-1}\rfloor, \lfloor a^{-2}\rfloor,\ldots,  \lfloor a^{1- \lfloor \log(n)/\log(a)\rfloor}\rfloor \}$
and $\cD_r$ as in~\eqref{eq:def_dyadic}. Interestingly, the cardinality of the dyadic grid or more generally of the $a$-adic grid is order $O(n)$, whereas the complete grid $\cG_D$ is quadratic.
\begin{figure}[h]
	\centering
	\includegraphics[scale = 0.15]{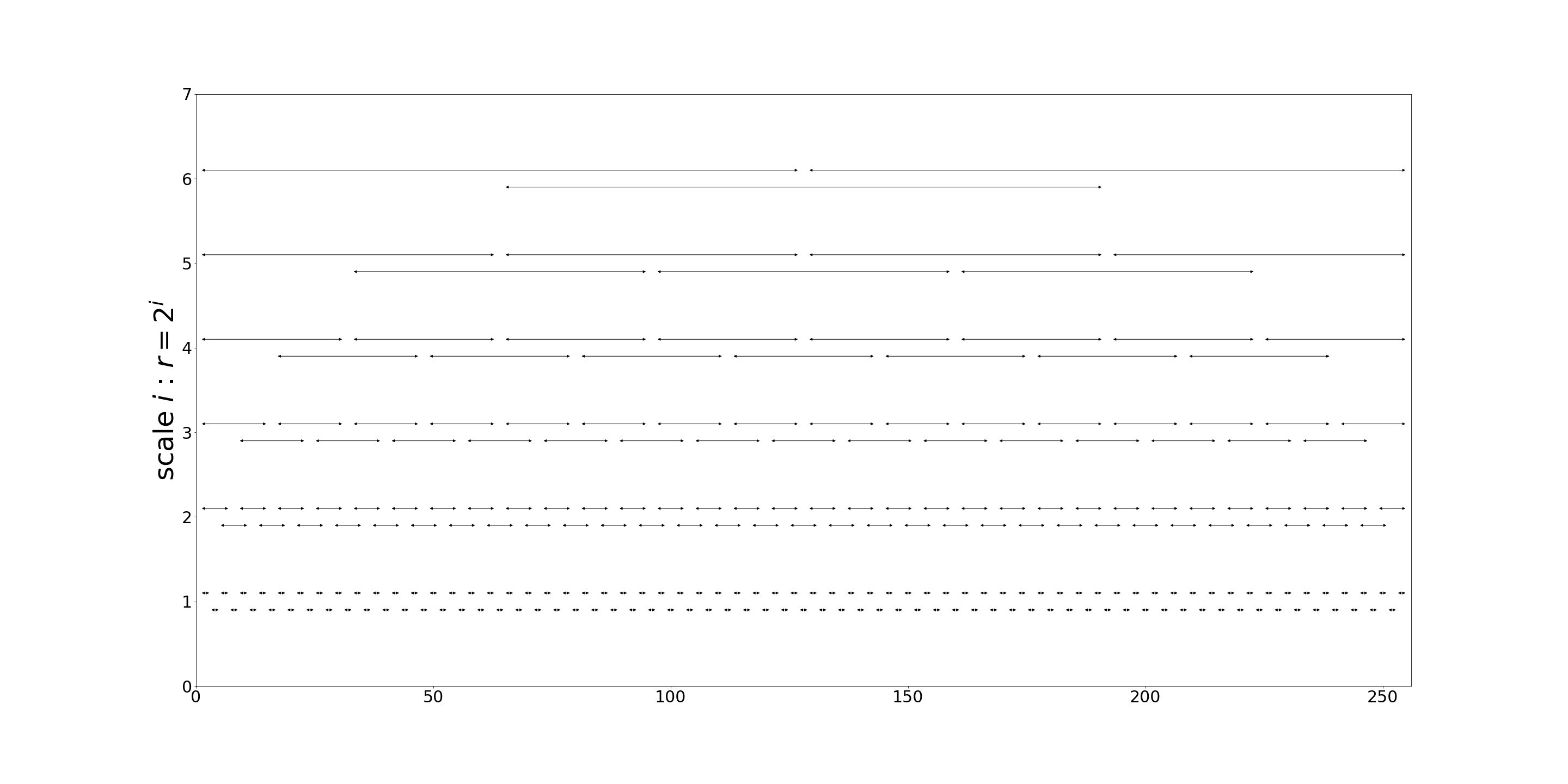}
	\caption{\label{figure_diadic_grid} The dyadic grid is represented as follows : for each $r = 2^i$ and $l \in \cD_r$, we draw the interval $[l-r+1,l+r-1]$ at position $(l,\log_2(r))$.}
\end{figure}

Grids are reminiscent of the $c$-normal systems of intervals introduced by Nemirovsky~\cite{nemirovskiy1985nonparametric} (see also \cite{li2019multiscale} for a definition) although our definition allows for non-necessarily normal intervals.

Given a fixed grid $\cG$, a multiscale test is simply a collection of test $\Tg = (\Tg_{l,r})_{(l,r) \in \cG}$ indexed by the elements of $\cG$, which  amounts to testing at all scales
$r\in \cR$ and all locations $l\in \cD_r$ whether the functional $\Gamma(\mathbb{P}_t)$ is constant over the segment $[l-r,l+r)$. Equivalently, $T_{l,r}$ tests whether there exists a change-point in $[l-r+1,l+r-1]$.




\subsection{From a multiscale test to a change-point detection procedure}

Our purpose is to introduce a generic procedure to translate a multiscale procedure into a vector of change-points. Intuitively,  if, for some $(l,r)\in \cG$, we have $T_{l,r}=1$, then the functional $\Gamma(\mathbb{P}_t)$ is certainly not constant over $[l-r,l+r)$ which entails that there is possibly at least one change-point in $[l-r+1,l+r-1]$. As a consequence, the multiscale test gives a collection  $\cI(T)=\{[l-r+1,l+r-1] \text{ s.t. } T_{l,r}=1\}$ of intervals that tentatively contain at least one change-point.

\medskip

If all these intervals were disjoint, then one simply would take $\widehat{\tau}$ as the sequence of centers of these intervals.  Unfortunately, when two intervals $[l_1-r_1+1,l_1+r_1-1]$ and $[l_2-r_2+1,l_2+r_2-1]$ in $\cI(T)$ have a non-empty intersection, one cannot necessarily decipher  whether there is only one change-point in the intersection of both intervals or if each interval contains a specific change-point. Hence, our general objective is to transform the collection $\cI(T)$ into a collection of non-intersecting intervals by either discarding or merging some of them.

We propose the following bottom-up iterative procedure for building a collection of non-intersecting intervals. Start with $\cT_0= \SCP_0=\emptyset$. For any scale $r\in \cR$, we compute the  collections $\SCP_r$ of intervals of scale $r$ and the collection $\cT_r$ of locations based on the following
\begin{align*}
\cT_r &= \left\{
    l \in \cD_r, \quad
    \Tg_{l,r} = 1 \quad \text{and }\quad [l-r+1,l+r-1] \bigcap \big(\underset{r' < r,\  r'\in \cR}{\bigcup} \SCP_{r'}\big) = \emptyset\enspace ;
 \right\}\\
\SCP_r & = \underset{l \in \cT_r}{\bigcup} [l-r+1,l+r-1]\enspace .
\end{align*}
The sets $\cT_1$ and $\SCP_1$ are made of all positions $l$ such that $T_{l,1}=1$. More generally, $\cT_r$ contains all locations $l$ such that
$T_{l,r}=1$ and the corresponding interval  $[l-r+1,l+r-1]$ does not intersect with any of the detected intervals at a smaller scale $r'<r$. The set $\SCP_r$ contains all intervals associated to $\cT_r$.

One can easily check that $\SCP=\bigcup_r \SCP_r$ is a union of closed non-intersecting intervals. Denote  $\Cg = \{ \Cgc_1, \dots , \Cgc_{\hat K}  \}$   the partition of $\SCP$ into connected components such that, for all
$1\leq i<j \leq {\hat{K}}$, $\max \Cgc_i < \min \Cgc_j$. Finally, we estimate the vector of change-points $\widehat{\tau}$ by taking the center of each segment $\Cgc_k$. In other words, we take  $\hat \tau_k := \tfrac{1}{2}(\min \Cgc_k + \max \Cgc_k)$ for any $1\leq k\leq \widehat{K}$. This bottom-up  aggregation procedure is summarized in \Cref{algo_1}  and illustrated in \Cref{fig:3} below.

\medskip

\noindent
{\bf Remark}: If, for some $r\in \cR$ and some $l_1<l_2\in \cD_r$, we have $T_{l_1,r}=1$, $T_{l_2,r}=1$, and $l_1+r-1\geq l_2-r+1$, then $\SCP_r$ contains the segment $[l_1-r+1,l_2+r-1]$. In other words, our aggregation procedure merges two intervals if and only if they correspond to the same scales.  In Section~\ref{sec:algo2}, we also introduce a variant of the algorithm where, instead of merging these two intersecting with identical scale, we discard one of them.

\begin{algorithm}
	\caption{\label{algo_1} Bottom-up aggregation procedure of multiscale tests}
	\KwData{$y_t, t = 1 \dots n$ and local test statistics $(\Tg_{l,r})_{(l,r) \in \cG}$}
	\KwResult{$(\hat \tau_k)_{k\leq \hat K}$}
	$\cT_r,\SCP_r  = \emptyset$ for all $r \in \cR$ and $\SCP=\emptyset$\;
	\For{increasing $r \in \cR$}{
		\For{$l \in \cD_r$ s.t. $\Tg_{l,r} = 1$}{
		\If{$[l-r+1,l+r-1]\bigcap \SCP = \emptyset$}{
				$\cT_r \gets \cT_r \cup \{l\}$\;
				$\SCP_r \gets \SCP_r \cup [l-r+1,l+r-1]$\;
		 }
	}
$\SCP = \SCP \bigcup \SCP_r$\;
  }
	Let $(\Cgc_k)_{k= 1,\dots,\hat K}$ be the connected components of $\SCP$ sorted in increasing order\;
	\Return $\left(\hat \tau_k = \frac{1}{2}(\min \Cgc_k + \max \Cgc_k )\right)_{k = 1,\dots,\hat K}$
\end{algorithm}
\begin{figure}[ht]
	\centering
	\includegraphics[scale = 0.15]{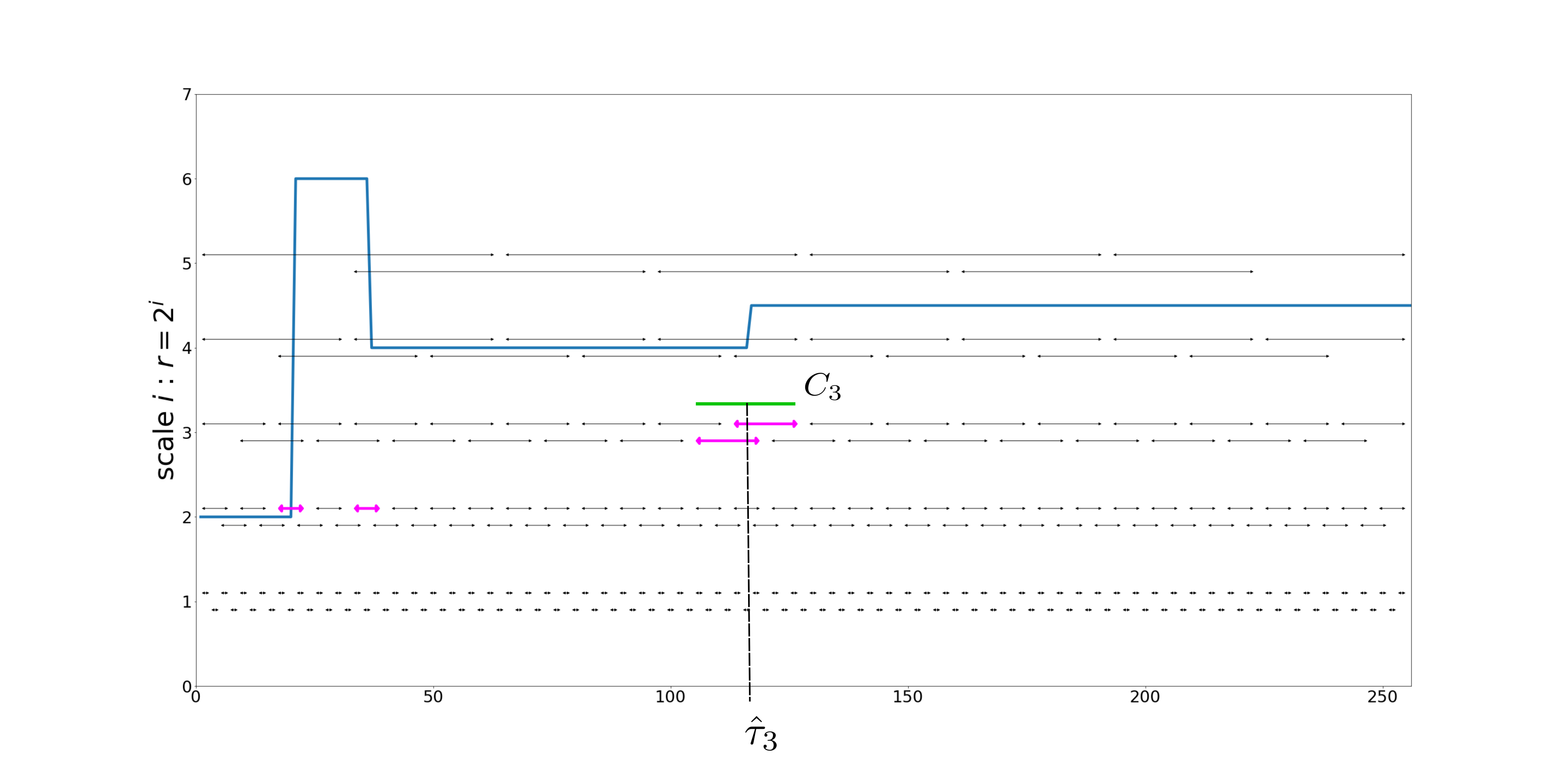}
	\caption{\label{fig:3} Example of our change-point detection procedure with three change-points. The first two change-points have large heights and are detected at a small scale $r$ (in magenta) while the third one is detected at a larger scale $r$.}
\end{figure}

\noindent
{\bf Computational Cost}. A naive implementation of \Cref{algo_1} - and also of ~\Cref{algo_2} defined in Appendix - requires to compute all tests $T_{l,r}$ on the grid, whereas the aggregation procedure  only needs to compute a number of tests $T_{l,r}$ proportional to the size of the grid. More precisely, if the computational cost of $T_{l,r}$ is $\Lambda_{l,r}$ for each $(l,r)$ in the grid $\cG$, then the aggregation procedure requires $O(\sum_{(l,r)\in \cG}\Lambda_{l,r})$ computations. If for all $(l,r)$, the cost $\Lambda_{l,r}$ is proportional to $r$, that is $\Lambda_{l,r}=O( r\Lambda)$, then the overall computational cost is 
$O(\Lambda\sum_{(l,r)\in \cG}r)$ which is $O(\Lambda n^3)$ for the complete grid and $O(\Lambda n\log(n))$ for the dyadic grid. One can speed up the full procedure by computing the statistics $T_{l,r}$ and aggregating on the fly by checking whether $[l-r+1,l+r-1]$ intersects $\SCP$ before evaluating $T_{l,r}=1$. Indeed, the connected components $C_k$ can be computed at each increasing scale $r$. Hence, at scale $r$, one only needs to compute the tests $T_{l,r}$ at locations $l$ such that  $[l-r+1, l+r-1]$ does not intersect the connected components detected at scales $r' < r$.

\subsection{General analysis}
\label{sub:general_analysis}

In this subsection, we provide an abstract theorem translating error controls of the multiple test procedure $T$ in terms of properties of $\widehat{\tau}$.
As explained in the introduction, the time series $(y_t)$ may contain change-points that are too small to be detected.  Having this in mind, we define a subset $\cK^* \subset [K]$ of indices corresponding to so-called significant change-points. As our purpose is to provide deterministic condition so that the change-points in $\cK^*$, we  need to introduce, for each $k\in \cK^*$, an element of the grid $(\tg_k,\rg_k) \in \cG$   at which the statistic $T$ is expected to detect $\tau_k$. One could think of $\tg_k$ as some position close to $\tau_k$ and to $\rg_k$ as some radius which is large enough to convey information on the change-point. 
Recall that the length $r_k$ of the change-point $\tau_k$ is defined by $r_k=\min(\tau_{k+1}-\tau_k,\tau_k-\tau_{k-1})$.
We assume that the scales $\rg_k$ and the location $\tg_k$  of detection satisfy the two following conditions:
\begin{equation}\label{eq:hypothese_scale}
4(\rg_k - 1) < r_k \quad \mbox{and}\quad |\tg_k - \tau_k| \leq \rg_k - 1. \enspace
\end{equation}
The first condition ensures that the scale $\rg_k< r_k/4+1$ is small enough compared to the length $r_k$.
The second condition is always satisfied if  $\tg_k$ is the best approximation of $\tau_k$ in $\cD_{\rg_k}$ and if the grid $\cG$ satisfies the following approximation property

\smallskip

\noindent
{\bf (App):}  For all  $r\in \cR$ and all $l\in [r+1,n-r+1]$, there exists  $l'\in \cD_r$ such that
$|l' - l| \leq r - 1$.

\smallskip

This property  entails that any point $l$ can be approximated at distance $r-1$ by some location in $\cD_r$. This also implies that each point $l\in [r+1,n-r]$ belongs to at least one segment $(l'-r,l'+r)$ where $l_1$ lies in $\cD_r$. In practice, the $a$-adic grids $\cG_a$ and the complete grid satisfy {\bf (App)}.



\medskip

Next, we introduce an event on the tests $(T_{l,r})$ under which  the change-point estimator $\widehat{\tau}$ of \Cref{algo_1} performs well.
In the following, we write $\cH_0$, the collection of all possible $(l,r)\in J_n$ such that there is no change in $[l-r+1,l+r-1]$, i.e. $\Gamma(\bbP_t)$ is constant on $[l-r, l+r)$. Equivalently, we have
\begin{equation}\label{eq:Ho}
(l,r) \in \cH_0 \quad \text{iff} \quad (l-r,l+r) \cap \{\tau_k, k=1,\dots,K\} = \emptyset \enspace .
\end{equation}

For a collection $\cK^*$ and some elements of the grid $(\bar \tau_k, \bar r_k )$ satisfying \Cref{eq:hypothese_scale},  the Event {\bf $\cA \left( \Tg, \cK^*,(\tg_k,\rg_k)_{k\in \cK^*}\right)$} is defined as the conjunction of the two following properties:
	(i) {\bf (No false positive)} 
    $\Tg_{l,r} = 0$ for all $(l,r)\in \cH_0 \cap \cG$
  	(ii) {\bf (Detection of significant change-points)} for every $k \in  \cK^*$, we have $\Tg_{\tg_k,\rg_k} = 1$.

 The first property states that $T$ performs no type I errors on the event $\cA \left( \Tg, \cK^*,(\tg_k, \rg_k)_{k\in \cK^*}\right)$, whereas the second property enforces that all the significant change-points are detected by the specific tests $\Tg_{\tg_k,\rg_k}$.

\begin{theorem}\label{th:g}
 The following holds for any grid $\cG$, any local test statistic $\Tg$, any non-negative integer $K$, any distribution with $K$ change-points, any $\cK^*\subset [K]$  and scales and locations $(\tg_k, \rg_k)_{k \in \cK^*}$ in $\cG$ satisfying Assumption~\eqref{eq:hypothese_scale}. Under the event $\cA( \Tg, \cK^*, (\tg_k, \rg_k)_{k \in \cK^*})$, the estimated change-point vector $\hat{\tau}$ returned by \Cref{algo_1} satisfies the two following properties
	\begin{itemize}
		\item {\bf Significant change-points are detected}:
		for all $k \in \cK^*$, there exists $k' \leq \hat K$ such that
		$|\hat\tau_{k'} - \tau_{k}| \leq \rg_{k} - 1 < \frac{r_{k}}{4}$.		\item ({\bf NoSp}): No Spurious change-point is detected~\eqref{eq:def:nosp}.
	\end{itemize}
\end{theorem}
The first property states that so-called significant change-points $(\tau_k)_{k \in \cK^*}$ are detected by the generic algorithm at the right scale. The no-spurious property~\eqref{eq:def:nosp} guarantees that, around any true change-point $\tau_k$, the procedure estimates at most one single change-point $\widehat{\tau}_l$.
Importantly, the theorem does not make any assumption on  the non-significant change-points. In fact, change-points $\tau_k$ with $k\in [K]\setminus \cK^*$ may or may not be detected. In general, we can only conclude from \Cref{th:g} that $|\cK^*|\leq \widehat{K}\leq K$ on the event $\cA \left( \Tg, \cK^*,(\tg_k,\rg_k)_{k\in \cK^*}\right)$ .

\medskip

\Cref{th:g} is abstract but its main virtue is to translate multiple testing properties into change-point detection properties. For a specific problem such as multivariate mean change-point detection considered in the next section, the construction of a near optimal procedure boils down to introducing a collection of local test statistics, such that (a) change-points $\tau_k$ belong to $\cK^*$ under minimal conditions,  (b) the scale $\rg_k$
is the smallest possible, and (c) the event $\cA( \Tg, \cK^*, (\tg_k,\rg_k)_{k \in \cK^*})$ holds with high probability.

\medskip

In the case where all the change-points are significant, the result of \Cref{th:g} can be reformulated as follows:
\begin{corollary}\label{cor:thg}
 The following holds for any grid $\cG$, any local test statistic $\Tg$, any non-negative integer $K$, any distribution with $K$ change-points, any $(\tg_k, \rg_k)_{k =1,\ldots, K}$ in $\cG$ satisfying Assumption~\eqref{eq:hypothese_scale}. Under the event $\cA( \Tg, [K], (\tg_k, \rg_k)_{k=1,\ldots, K}$, the estimated change-point vector $\hat{\tau}$ returned by \Cref{algo_1} satisfies $\widehat{K}=K$ and,
 \[|\hat\tau_{k} - \tau_{k}| < \rg_k - 1\leq \frac{r_{k}}{4}\quad\quad  \text{  for all }k=1,\ldots, K\enspace .
 \]
\end{corollary}
Let us respectively  define the Hausdorff distance and the Wasserstein distance of two vectors $(u_1,\dots,u_K)$ and $(v_1,\dots,v_K)$ in $\bbR^K$ by
 $d_H(u,v) = \max_{k = 1,\dots,K}|u_k - v_k|$ and $d_W(u,v) = \sum_{k = 1,\dots,K}|u_k - v_k|$. Then,
\Cref{cor:thg} straightforwardly implies that, if $\cK^* = [K]$, then these two losses are bounded as follows
$$d_H(\hat \tau, \tau) \leq \max_{k = 1,\dots,K}(\rg_k - 1) \quad \mbox{and } \quad d_W(\hat \tau, \tau) \leq \sum_{k = 1,\dots, K} (\rg_k - 1)\enspace .$$

\medskip

As an alternative of~\Cref{algo_1}, one could use other bottom-up aggregating procedures. For instance, \Cref{algo_2} defined in \Cref{sec:algo2}  also satisfies \Cref{th:g}. Although these two algorithms are closely related, \Cref{algo_1} is slightly more conservative than \Cref{algo_2} since it merges all detection intervals at a given resolution while \Cref{algo_2} only keeps one interval at a given resolution when multiple intervals intersect - the one with smallest index $t$. While the minimax properties of both methods are comparable - at least up to a multiple constant - the choice of aggregation method will have an influence in practice on the outcome: \Cref{algo_1} will be slightly more stable, detect less change-points, and provide wider confidence interval around them, while \Cref{algo_2} will be slightly more sensitive to smaller changes, i.e.~detect smaller change-points, will be more precise, and somewhat less stable.

\medskip

\Cref{th:g} ensures that, if $T_{\overline{\tau}_k, \overline{r}_k}=1$ with $(\overline{\tau}_k, \overline{r}_k)$ satisfying Assumption~\eqref{eq:hypothese_scale}, then the change-point $\tau_{k}$ is detected. Inspecting the proof of \Cref{th:g}, one easily checks that Assumption~\eqref{eq:hypothese_scale} is minimal for \Cref{algo_1} (and also for \Cref{algo_2}). Still, one may wonder whether any generic algorithm has to require that  $4(\overline{r}_k-1)<r_k$ to detect the change-points or if there exists a generic algorithm where the constant $4$ in the above condition can be improved.

\medskip

\noindent
{\bf Comparison with narrowest over threshold methods.} As mentioned in the introduction, other aggregation procedures have been proposed in the literature. In particular, the narrowest over threshold scheme proposed by \cite{baranowski2019narrowest} and later used in~\cite{kovacs2020seeded} is also closely related  to the local segmentation algorithm of Chan and Chen~\cite{chan2017multi}. A simple extension of these procedures for generic change-point problems and for a general collection of tests $(T_{l,r})$ would amount to modifying~\cref{algo_1} by selecting locations $l$ in $\mathcal{D}_r$ such that $T_{l,r}=1$ and $[l-r+1,l+r-1]$ does not intersect previously detected change-points, whereas we require in Algorithms \ref{algo_1} and \ref{algo_2}, that $[l-r+1,l+r-1]$ does not intersect previously detected confidence intervals. In some way, the narrowest-over threshold scheme is therefore less conservative. Unfortunately, there is no generic result in the form of Theorem~\ref{th:g} for such procedures and, from informal arguments,  we doubt that the corresponding procedure provably achieves ({\bf NoSp}) under a control of the FWER of the tests. Inspecting the proof of Theorem 1 in~\cite{baranowski2019narrowest} and Theorem 3 in~\cite{kovacs2020seeded} for univariate mean change-point problems, one observes that the chosen threshold is much larger than what is needed to control the FWER so that the theoretical threshold is certainly over-conservative -- see step 5 of the proof of Theorem 1 in~\cite{baranowski2019narrowest}. In contrast, Theorem 1 in~\cite{chan2017multi} for univariate change-point problems is based on the minimal threshold, but the proof relies on the important assumption that the number $K$ of change-point remains bounded while $n$ goes to infinity. Besides, it is not clear how one could extend the arguments to more general settings.

\section{Multivariate Gaussian change-point detection}
\label{sec:multi_scale_change_point_estimation_in_gaussian_noise}
We now turn to the multivariate change-point model introduced in \Cref{sec:setting}. Throughout this section, we assume that the random vectors $\Noise_t$ are independently and identically distributed with
$\Noise_t \sim \cN(0, \sigma^2 \bI_p)$. Since we shall apply the general aggregation  procedures introduced in the previous section, our main job here is to introduce a near-optimal testing procedure.


\medskip

Fix some quantity $\delta\in (0,1)$. At the end of the section, $1-\delta$ will correspond to the probability of the event {\bf $\cA \left( \Tg, \cK^*,(\tg_k,\rg_k)_{k\in \cK^*}\right)$} introduced in the previous section. Alternatively, one may interpret $\delta$ as an upper bound of the desired probability that the change-point detection procedure detects a spurious change-points.  Recall that, for a change-point $\tau_k$, $s_k$ stands for the sparsity of the difference $\mu_{k+1}-\mu_k$. The energy of a given change-point $\tau_k$ is $c_0$-high if

\beq\label{EnergyGauss}
r_k \Delta_k^2
\geq
\Csto\sigma^2\left[s_k \Log{1 + \frac{\sqrt{p}}{s_k} \sqrt{\Log{\frac{n}{r_k \delta }}}}
+ \Log{\frac{ n}{r_k \delta }}\right] \enspace ,
\eeq
for some universal constant $\Csto$ to be defined later. We show in this section that when $c_0$ is large enough, all high-energy change-points can be detected. Conversely, it is established in \Cref{sec:lower} that Condition~\eqref{EnergyGauss} is (up to a multiplicative constant) optimal for detecting change-points and cannot be weakened.

Let us now discuss the different regimes contained in Equation~\eqref{EnergyGauss}. In what follows, define
$$\psi^{(g)}_{n,r,s} := s \Log{1 + \frac{\sqrt{p}}{s} \sqrt{\lo}}
+ \lo\ ; \quad \quad \lo := \log\left(\frac{n}{r \delta}\right) \enspace ,$$
in order to alleviate notations.  If $\lo \geq  p/2$, then
$	\psi^{(g)}_{n,r,s}\asymp \lo$ where $u\asymp v$ means that for two positive numerical constants $c_1$ and $c_2$, one has $c_1v\leq u\leq c_2v$. This corresponds to the minimal energy condition for detection in the univariate case, i.e.~when $p=1$; see~\cite{verzelenoptimal_change_point}. The condition $\lo \geq  p/2$ occurs when $p$ is rather small and the scale $r$ is much smaller than $n$.
If $\lo\leq p/2$,
then
\begin{equation*}
	\psi^{(g)}_{n,r,s} \asymp \left\{
\begin{array}{cc}
	\lo &\text{ if } s \leq \frac{\lo}{\log(p)- \log\left(\lo\right)} \\
	s \Log{2\frac{p}{s^2} \lo}& \text{ if } \frac{\lo}{\log(p)- \log\left(\lo\right)} < s   < \sqrt{p\lo}
	\\
	\sqrt{p\lo} &\text{ if } s \geq \sqrt{p\lo}\enspace .
\end{array}
\right.
\end{equation*}

We define $\cK^*\subset [K]$  as the subset of indices such that $\tau_k$ satisfies~\eqref{EnergyGauss}.
For any $k\in \cK^*$, we define $r^*_k$ as the minimum radius $r$ such that an inequality similar to~\eqref{EnergyGauss} is satisfied for $r\Delta_k^2$, namely
\begin{equation}\label{eq:VitesseGauss}
	r^*_k
	=
	\min
	\left\{
		r \in \bbR^+:\quad r{\Delta_k^2}
		\geq
		\Csto\sigma^2\left[s_k \Log{1 + \frac{\sqrt{p}}{s_k} \sqrt{\Log{\frac{n}{r \delta }}}}
	+ \Log{\frac{ n}{r \delta}}\right]
	\right\}\enspace.
\end{equation}


In the following,  we introduce multi-scale tests for respectively dense and sparse change-points. For simplicity, we restrict our attention to the dyadic grid $\cG_D = (\cR, \cD)$ introduced in the previous section (see Equation \eqref{eq:def_dyadic}), the complete grid being used in the next section.

To apply \Cref{th:g}, we will consider an event {\bf $\cA \left( \Tg, \cK^*,(\tg_k,\rg_k)_{k\in \cK^*}\right)$}   in the proof of \Cref{cor:gaus} where the scale $\rg_k \in \cR$ is of the same order as $r^*_k \in \bbR^+$.

\subsection{Dense change-points}
\label{sub:dense_regime}



We focus here on dense change-points for which $s_k$ is possibly as large as $p$. Given $\kappa>0$,  $\tau_k$ is a $\kappa$-\emph{dense} high-energy change-point if
\begin{equation}
	\label{eq:EnergyDenseGauss}
	r_k \Delta_k^2 \geq  \kappa \sigma^2\NoiseBound<p,n,r_k,\delta> \enspace.
\end{equation}
The requirement~\eqref{eq:EnergyDenseGauss} is analogous to \eqref{EnergyGauss} when $s_k\geq [p\log(n/(r_k\delta))]^{1/2}$.  For any $\kappa$-dense high-energy change-point, we define $\rgD_k \in \cR$ as the minimum radius $r\in \cR$ such that an inequality of the same type as \eqref{eq:EnergyDenseGauss} is satisfied for $r\Delta_k^2$,
\begin{align*}
	\rgD_k
	=
	\min
	\left\{
		r \in \cR:\, 8r{\Delta_k^2}
	   \geq \kappa\sigma^2\NoiseBound<p, n,r,\delta>
   	\right\} \enspace.
\end{align*}
Intuitively, $\rgD_k$  corresponds to the smallest scale such that $\tau_k$ is guaranteed to be detected. By definition, we have
$4(\rgD_k-1) \leq r_k$. Let $\tgD_k$ be the best approximation of $\tau_k$ in the grid with scale $\rgD_k$. By definition of the dyadic grid, we have $|\tgD_k - \tau_k| \leq \rgD_k/4$.




For any positive integers $r\in [1;n]$ and $l\in [r+1, n+1 - r]$, we define the  statistic
$\psiD_{l,r}:= \norm{\bC_{l,r}}^2 -  p$.  If  $\theta$ is constant over $[l-r,l+r)$, then the expectation of $\psiD_{l,r}$ is zero.  Recall that the rescaled CUSUM statistic $\bC_{l,r}$ depends on the noise level $\sigma$, and the statistic $\psiD_{l,r}$ therefore requires the knowledge of $\sigma$.
To calibrate the corresponding test $\TD_{l,r}$ rejecting for large values of $\psiD_{l,r}$ we introduce
\begin{align*}
\TD_{l,r} := \1\left\{ \psiD_{l,r}> \tD_r \right\}\ ;\quad \quad 	\tD_r & := 4\NoiseBound<p,2n,r,\delta> \enspace .
\end{align*}


\begin{proposition}\label{prop:dense}
There exists a universal constant  $\kappa_{\mathrm{d}}>0$ and an event $\xDense$ of probability larger than $1-2\delta$ such that 
(i) $\TD_{l,r} = 0$ for all $(l,r) \in \cH_0\cap\cG_D$ and (ii) $\TD_{\tgD_k,\rgD_k} = 1$ for all $\kappa_{\mathrm{d}}$-dense high-energy change-point $\tau_k$.
\end{proposition}

The above proposition ensures that, on the event $\xDense$, the collection of tests $\TD_{l,r}$ detects all dense high-energy change-points at the scale $\rgD_k$ and makes no false positives on the dyadic grid $\cG_D$. If we plugged this collection of tests into the general multiple change-point procedure, then \Cref{th:g} would entail that all $\kappa_{\mathrm{d}}$-dense high-energy change-points are discovered and localized and that $\widehat{\tau}$ does not detect any spurious change-point. In the next subsection, we introduce alternative tests that are tailored to sparse change-points and thereby allow to detect change-points that are not $\kappa_{\mathrm{d}}$-dense high-energy
but still satisfy the energy condition~\eqref{EnergyGauss}.

\subsection{Sparse change-points}

\subsubsection{Energy condition}
\label{subsub:energy_condition}

For a given $1\leq k\leq K$, the change-point $\tau_k$ is a $\kappa$-\emph{sparse} high-energy change-point if $s_k \leq [p \log(n/(r_k\delta))]^{1/2}$ and
\begin{equation}\label{eq:EnergySparseGauss}
	r_k \Delta_k^2 \geq \kappa\sigma^2\NoiseBoundsparse<p, n,r_k,\delta, s_k>\enspace .
\end{equation}

 If $\tau_k$ is a $\kappa$-sparse high-energy change-point, we define $\rgSp_k$ as the minimum scale such that an inequality similar to  \eqref{eq:EnergySparseGauss} is satisfied :
\begin{align*}
	\rgSp_k
	=
	\min
	\left\{
		r \in \cR:\quad 8r{\Delta_k^2}
		\geq \kappa \sigma^2\NoiseBoundsparse<p, n,r,\delta, s_k>
	\right\} \enspace.
\end{align*}
As in the dense case, we have $4(\rgSp_k-1) \leq r_k$. Set $\tgSp_k$ as the best approximation of $\tau_k$ in the grid $\cD_{\rgSp_k}$ at scale $\tau_k$. By definition of the dyadic grid, we have $|\tgSp_k - \tau_k| \leq \rgSp_k/4$. We introduce below two statistics for handling this problem.


\subsubsection{Berk-Jones Test}
\label{subsub:Berk_jones_regime}

The Berk-Jones test~\cite{moscovich2016exact} is a variation of the Higher-Criticism test  originally introduced in \cite{jin2004} for signal detection. It has been previously studied in~\cite{chan2015optimal} for sparse segment detection. 
We decided to use the Berk-Jones test in this paper because of its intrinsic formulation in terms of the quantiles of a Bernoulli distribution, but the Higher-Criticism test would reach the same rates of detection within a constant factor. We use the notation $\bbN^*$ to denote the set of positive itegers. Given
$(l,r)$ in the grid $\cG_D$,  we first introduce $N_{x,l,r}$ as the number of coordinates of $\bC_{l,r}$ that are larger than $x$ in absolute value.
\beq\label{eq:def_count}
N_{x,l,r}= \sum_{i=1}^p \1_{|\bC_{l,r,i}|> x }
\eeq
If $(l,r) \in \cH_0$, then the rescaled CUSUM statistic follows a standard normal distribution and $N_{x,l,r}$ therefore follows a Binomial distribution with parameters $p$ and $2\overline{\Phi}(x)$. The Berk-Jones test amounts to rejecting the null, when at least one of the statistics $N_{x,l,r}$, for $x\in \bbN^*$, is significantly large. Next, we  formalize what we mean by 'large'.

For any $u>0$, any $q_0\in [0,1]$, and positive integer $p_0$, denote
$\overline{Q}(u,p_0,q_0)= \P[\cB(p_0,q_0)> u]$
 the tail distribution function of a Binomial distribution with parameters $p_0$ and $q_0$. Given $\alphabj\in [0,1]$, we then write $\overline{Q}^{-1}(\alphabj,p_0,q_0)$ for the corresponding quantile function,
\[
\overline{Q}^{-1}(\alphabj,p_0,q_0)= \inf_{u}\left[\P[\cB(p_0,q_0)> u]\leq \alphabj \right]\enspace .
\]
Given a scale $r\in \cR$ and a positive integer $x$, we  define the weights
\begin{equation}
	\label{eq:def_alpha_tr}
\alphabj^{(\mathrm{BJ})}_{x,r} = \frac{6\delta r}{\pi^2x^2|\cD_r|n }\enspace .
\end{equation}
This allows us to define the Berk-Jones statistic over $[l-r,l+r)$ as the test rejecting the null when at least one $N_{x,l,r}$ is large.
\beq\label{eq:def_berk_jones}
\Tbj_{l,r}=\max_{x\in \bbN^*}\1\left\{N_{x,l,r}>\overline{Q}^{-1}(\alphabj^{(\mathrm{BJ})}_{x,r},p,2\overline{\Phi}(x)) \right\} \enspace .
\eeq
Equivalently, $\Tbj_{l,r}$ is an aggregated test based on the statistics $N_{x,l,r}$ with weights $\alphabj^{(\mathrm{BJ})}_{x,r}$. From the above remark and a union bound, we deduce that the probability that the collection of tests $\{\Tbj_{l,r}, (l,r)\in \cG_D\}$ rejects a least one false positive is at most $\delta$:
\[
\P\left[\max_{(l,r) \in \cH_0\cap \cG_D}\Tbj_{l,r}= 1\right]\leq \sum_{r\in \cR}\sum_{ l\in \cD_r}\sum_{x\in \bbN^*} \alphabj^{(\mathrm{BJ})}_{x,r}
\leq \sum_{r\in \cR}\sum_{ l\in \cD_r} \frac{\delta r}{|\cD_r|n}\leq \sum_{r\in \cR}\frac{\delta r}{n}\leq \delta\enspace ,
\]
where we recall that $(l,r) \in \cH_0$ if and only if $\Theta$ is constant on $[l-r,l+r)$.
 Although one may think from the definition~\eqref{eq:def_berk_jones} that $\Tbj_{l,r}$ involves an infinite number of $N_{x,l,r}$, this is not the case. Indeed, $N_{x,l,r}$ is a non-increasing function of $x$ whereas for all $x$ such that $2p \overline{\Phi}(x)\leq \alphabj^{(\mathrm{BJ})}_{x,r}$, we have $\overline{Q}^{-1}(\alphabj^{(\mathrm{BJ})}_{x,r},p,2\overline{\Phi}(x))=0$.
Writing $x_{0,r}$ the smallest $x$ such that  $2p \overline{\Phi}(x)\leq \alphabj^{(\mathrm{BJ})}_{x,r}$ we derive
\[
\Tbj_{l,r}=\max_{x=1,\ldots, x_{0,r}}\1\left\{N_{x,l,r}>\overline{Q}^{-1}(\alphabj^{(\mathrm{BJ})}_{x,r},p,2\overline{\Phi}(x)) \right\}\ .
\]
Since, for any $x>0$, we have $\overline{\Phi}(x)\leq e^{-x^2/2}$, one can deduce that $x_{0,r}\leq c[\log(np/(r\delta))]^{1/2}$,
for some numerical constant $c>0$.

\subsubsection{Partial norm statistics}
\label{subsub:max_statistic}
The Berk-Jones test is able to detect change-points $\tau_k$ for which there exists $s$ such that the $s$ largest squared coordinates of $\mu_k - \mu_{k-1}$ are larger than $C (\log(ep/s^2) + \log(n/r_k)/s)$ with a large enough constant $C$. However, it may happen that $\tau_k$ satisfies the energy condition~\eqref{EnergyGauss} and that the $s$ largest coordinates of $\mu_k - \mu_{k-1}$ are negligible compared to $\log(n/r_k)/s$, mainly because $s\mapsto 1/s$ is not summable. To solve this issue, we introduce a second sparse statistic based on the partial sums.
Let
$$\cZ = \left\{1, 2,2^2,\dots,2^{\floor{\log_2(p)}} \right\}$$ denote the dyadic set. Only the sparsities $s \in \cZ$ will be analysed by the partial norm statistic. For any $(l,r)$ in the grid $\cG_D$, we respectively write  $\bC_{l,r,(1)}$, $\bC_{l,r,(2)},\ldots
$ the reordered entries of $\bC_{l,r}$  by decreasing absolute value, that is $|\bC_{l,r,(1)}| \geq \dots \geq |\bC_{l,r,(p)}|$. Then, for $s\in \cZ$, we define the partial CUSUM norm by
\begin{align}\label{eq:partialCUSUMnorm}
	\Ssup_{l,r,s}=  \sum_{i=1}^s\left( \bC_{l,r,(i)}\right)^2 \enspace .
\end{align}
Then, we define	the test $\Tsup_{l,r}$ rejecting the null when at least one of the partial norms is large
\begin{align*}
	\tsup_{r,s}
	:= \tsup_{r,s}(\delta)
	 = 4s\Log{\frac{2ep}{s}} + 4\Log{\frac{n}{r\delta}};  \quad \quad 	\Tsup_{l,r}
	 	 = \max_{s \in \cZ}\1\left\{\Ssup_{l,r,s} > \tsup_{r,s}\right\} \enspace .
\end{align*}




Finally, we define the sparse test by aggregating both the Berk-Jones test and the partial norm test. For any $(l,r)\in \cG_D$,  let
$ \Tsparse_{l,r} = \Tsup_{l,r}\lor \Tbj_{l,r}$. The next proposition controls the error of this collection of tests.

\begin{proposition}\label{prop:Sparse}
	There exists a universal constant  $\kappa_{\mathrm{s}}>0$ and an event $\xSparse$ of probability larger than $1-4\delta$ such that
	(i) $\Tsparse_{l,r} = 0$ for all $(l,r) \in \cH_0\cap\cG_D$ and (ii) $\Tsparse_{\tgSp_k,\rgSp_k} = 1$ for all $\kappa_{\mathrm{s}}$-sparse high-energy change-point $\tau_k$.
	\end{proposition}

Here we introduced two different statistics for the same sparse regime $s_k \leq [p \log(n/(r_k\delta))]^{1/2} $ - the Berk-Jones statistic and the partial sums statistic - mainly to solve a problem of integrability. We made this choice for the sake of simplicity, but we could have used a single test, as presented in \cite{liu2019minimax}
$$ \Psi_{x,l,r}^{(\mathrm{LGS})} = \sum_{i=1}^{p} \left(\bC_{l,r,i}^2 - \Esper{Z | Z \geq x}\right) \1 \{\bC_{l,r,i}^2 \geq x\} \enspace ,$$
where $Z$ follows a standard normal distribution $\cN (0,1)$. This statistic leads to the same type of result as the Berk-Jones statistic when enough coordinates $\mu_k - \mu_{k-1}$ are large in absolute value, and it is comparable to the partial sums statistic when its threshold $x$ becomes low enough.

\subsection{Consequences}

To conclude this section, it suffices to observe that, for $\Csto$ in~\eqref{EnergyGauss}, any $\Csto$-high-energy change-point $\tau_k$ in the sense of
\eqref{EnergyGauss} is either a $\tfrac{\Csto}{2}$-dense or a $\tfrac{\Csto}{2}$-sparse high-energy change-point. Hence, upon defining the test $\Tg_{l,r} = \TD_{l,r}\lor \Tsparse_{l,r}$ for $(l,r)\in \cG_D$, we consider the change-point procedure $\widehat{\tau}$ defined in \Cref{algo_1}. Gathering \Cref{th:g} with \Cref{prop:dense} and \Cref{prop:Sparse}, we obtain the following.

\begin{corollary}\label{cor:gaus}
	There exists a universal constant $\Csto>0$ such that, with probability higher than $1-6\delta$, the estimator $\widehat{\tau}$ satisfies ({\bf NoSp})  and {\bf detects} all $c_0$-high-energy change-points (as defined in~\eqref{EnergyGauss}) $\tau_k$ in the sense
\[
d_{H,1}(\hat \tau, \tau_k) < \frac{r_k^*}{2}\leq \frac{r_{k}}{2}  \enspace ,
\]
where $r^*_k$ is defined in~\eqref{eq:VitesseGauss}.
\end{corollary}
If the change-points are of high-energy, that is $\cK^* = [K]$, then \Cref{cor:gaus} can be reformulated as follows:
\begin{corollary}\label{cor:gaus_hausdorff}
Assume that for all $k = 1,\dots, K$, $\tau_k$ is a $c_0$-high-energy change-point (see~\eqref{EnergyGauss}) where $\Csto$ is the same as in Corollary~\ref{cor:gaus}. Then, with probability higher than $1-6\delta$,  the estimator $\widehat{\tau}$ satisfies $\hat K = K$
and
\[
	|\hat\tau_{k} - \tau_{k}| < \frac{r_k^*}{2}\leq \frac{r_{k}}{2} \ , \quad \quad \text{for all } k=1,\ldots, K\enspace .
\]
\end{corollary}
In particular, one can respectively bound the Hausdorff and the Wasserstein losses, with probability higher than $1-6\delta$ by
\begin{align}\label{eq:HandW}
d_H(\hat \tau, \tau) \leq \max_{k = 1, \dots,K} \frac{r_k^*}{2} \quad \mbox{and} \quad d_W(\hat \tau, \tau) \leq \sum_{k = 1, \dots,K} \frac{r_k^*}{2} \enspace .
\end{align}

In \Cref{sec:lower}, we establish that the Condition~\eqref{EnergyGauss} is (up to a multiplicative constant) unimprovable and corresponds to the detection threshold for multivariate change-points.

\medskip

\Cref{cor:gaus_hausdorff} can be compared to the result of \cite{Wang2018} on multivariate change-point detection in the multiple change-point setting. Using a method based on the CUSUM statistic and assuming that there are only high-energy change-points, the authors also obtain an upper bound on the energy necessary to detect the change-points. However, this result does not adapt to $r_k, \Delta_k, s_k$, and the detection rate is suboptimal in many regimes. Writing $r = \min_{k = 1,\dots,K} r_k$, $\Delta = \min_{k = 1,\dots,K} \Delta_k$ and $s = \max_{k = 1,\dots,K} s_k$, Theorem 5 of \cite{Wang2018} requires two conditions of the type $r\Delta^2 \geq c (\tfrac{n}{r})^4\log(np)$ and $r\Delta^2 \geq c s\tfrac{n}{r} \log(np)$. This detection rate is therefore suboptimal by a polynomial factor in $n/r$ when $r$ is of smaller order than $n$, and by a logarithmic factor $\log(np)$ instead of $\log(1+ \sqrt{p}/s\log(n/r)) + \tfrac{1}{s}\log(n/r)$ when $r$ is of order $n$.
Closer to our results, \cite{chan2017multi} have introduced another bottom-up procedure in the very specific asymptotic setting $n = e^{p^\zeta}$ for $\zeta\in (0,1)$ with a fixed $K$ number of change-points. Assuming that, for each change-point, at least $s$ coordinates of $\mu_{k+1}-\mu_{k+1}$ are larger than $\zeta$ in absolute value, \cite{chan2017multi} establish that their procedure provably detects the change-points as long as
\[
    r s \zeta^2 \geq c \begin{cases}
        \sqrt{p\log (n)} &\text{ if } s \geq 0.5 \sqrt{p\log(n)}\\
        s\Log{\frac{p}{s^2}\Log{n}} &\text{ if } s \leq 0.5 \sqrt{p\log(n)} \enspace .\\
    \end{cases}
 \]
In their specific asymptotic regime and when all non-zero coordinates are of the same order, and all the change-points have a similar length $r_k$, their result is similar to ours up to the logarithmic terms. Indeed, for equispaced change-points, our logarithmic term $\log(n/r_k)= \log(K)$ is much smaller than $\log(n)$. Besides, their result does not handle the presence of low-energy change-points and does not hold beyond the asymptotic regime $n = e^{p^\zeta}$. In contrast, our condition~\eqref{EnergyGauss} for high-energy change-points entails that the detection conditions are qualitatively different for other scalings in $n$ and $p$. On the technical side, our condition~\eqref{EnergyGauss} is of $l_2$ type whereas that in~\cite{chan2017multi} is of minimal non-zero type. Recovering the tight $l_2$ conditions turns out to be much more challenging as we need to handle situations where some coordinates have different orders of magnitude. This is the main reason why we need to resort to a combination of the Berk-Jones and the partial-norm statistics.

\medskip

\noindent
{\bf Comparison to  one change-point problem}.
When one knows that $K\leq 1$ (at most one change-point), then \cite{liu2019minimax} proved that it is possible to detect $\tau_1$ if and only if $r_1\Delta_1^2 \geq c \sigma^2 \big[s_1\log(1+\tfrac{1}{s_1}\sqrt{p\log\log 8n}) + \log\log 8n\big]$. As in the univariate setting, the problem with only one change-point is simpler than for general $K\geq 2$. As for our procedure, Liu et al.~\cite{liu2019minimax} rely on  statistics based on the CUSUM - a chi square statistics in the dense case and a thresholded sum of squared coordinates in the sparse case - to detect and localize $\tau_1$. It turns out that the detection procedure of~\cite{liu2019minimax} adapts to distance $r_1 = \max(\tau_1-1, n+1-\tau_1)$  the boundary, and one could refine their result by stating that $\tau_1$ is detectable if and only if $r_1\Delta_1^2 \geq c \sigma^2 [s_1\log(1+\tfrac{1}{s_1}\sqrt{p\log\log (2n/r_1)}) + \log\log (2n/r_1)]$ which is more smaller when $r_1$ is of the  order of $n$. This refined result is in the same spirit as our bounds for mutiple change-point, but the rate is faster because one obtains  $\log\log(n/r_1)$ - instead of $\log(n/r_k)$ in our case. The reason for this faster rate is due to the relative simplicity of the problem with only one change-point. Indeed, in single change-point detection, there is no need to look for change-points at all positions and scale at the same time, since scale and positions are related. This implies that it is possible to attain faster rates than in multiple change-point detection. The comparison between single and multiple change-point detection is thoroughly done in \cite{verzelenoptimal_change_point} for univariate models.

\noindent

{\bf Computational Cost}. The cost of the tests $\TD_{l,r}$ in the dense regime is $O(rp)$. The computation of the partial norm statistic requires to sort  the coordinates $\bC_{l,r,i}$ of the CUSUM statistic, which takes $O(p(r+\log(p)))$ operations. Since only the thresholds $x \leq c\log(np/(r\delta))^{1/2}$ are needed to compute the Berk-Jones statistic, it holds that, for $\delta \geq (np)^{-c}$ with a numerical constant $c > 0$,  the computational cost of the Berk-Jones statistic is $O(p(r+\log(np)))$. Thus, for each $(l,r)$, the overall computational cost of the test $T_{l,r} = \TD_{l,r}\lor \Tsparse_{l,r}$ is $\Lambda = O(p(r+\log(np)))$, and the computational cost of the whole change-point detection procedure on the dyadic grid is $O(np\log(np))$.


\section{Multi-scale change-point detection with sub-Gaussian noise}
\label{sec:multi_scale_change_point_estimation_in_sub_gaussian_noise}


We now turn to the more general case of sub-Gaussian distributions~\cite{vershynin}.
Given a random variable $Z$, define its $\psi_2$-norm by
$	\norm{Z}_{\psi_2}
	=
	\inf
	\{
		\nx > 0,\  \bbE[\exp(Z^2/x^2)] \leq 2
	\} \enspace.
$
Given $L>0$, a mean zero real random variable  is said to be $L$-sub-Gaussian if $\norm{Z}_{\psi_2}\leq L$. This implies in particular that, for all $\nx \geq 0$, one has $\Proba{|Z| \geq \nx}\leq 2 \exp( - \nx^2/L^2)$.
Throughout this section, we assume that, for $t = 1,\dots,n$, the random vectors $\Noise_t$ are independent, have independent $L$-sub-Gaussian components $\Noise_{t,i}$, for $i=1,\ldots, p$ with variance $\sigma^2$. As in the previous section, we apply the general aggregation  procedures introduced in \Cref{sec:meta_algorithm_for_multi_scale_change_point_detection_on_a_grid}. As a consequence, our main task boils down to introducing a near-optimal multiple testing procedure indexed by a grid for detecting the existence of a change-point. Here, we shall rely on the complete grid $\cG_F=J_n=\left\{(l,r)~:~ r = 1,\dots, \floor{\frac n 2} \text{ and } l = r+1, \dots, n - r\right\}$
whose size is quadratic with respect to $n$. All the results presented in this section are still valid (but with different numerical constants) if we keep the dyadic grid $\cG_D$ as in the previous section. Here, we use the complete grid  as a proof of concept that one can rely on the full collection of possible segments without deteriorating the rates. Still, controlling the behavior of the procedure on the complete grid is technically more involved and requires chaining arguments. A detailed comparison between the complete and dyadic grids is made in \Cref{sec:discussion}.

In order to emphasize the common points with the previous section, we use the same notation $\cK^*$ for the collection of high-energy change-points\footnote{See Equation~\eqref{eq:Energysg} as the energy condition is slightly different in the sub-Gaussian setting.}, $\rg_k$ for the scales associated to the $k$-th change-points\footnote{Re-defined in Equation~\eqref{eq:rkbarsg}.}, $\Psi$ for the statistics, $T$ for the test and $x$ for the thresholds although these quantities are slightly changed to cope with the sub-Gaussian tail distribution.
We follow the same scheme  as for the Gaussian case and first introduce multi-scale tests for dense change-points before turning to sparse change-points. As in the previous section, we consider some $\delta\in (0,1)$ corresponding to the type I error probability.

\poubelle{For any $k\in \cK^*$, we redefine $r^*_k$ as the minimum radius $r$ such that an inequality similar to~\eqref{eq:Energysg} is satisfied for $r\Delta_k^2$, namely
\begin{equation}\label{eq:VitesseSg}
	r^*_k
	=
	\min
	\left\{
		r \in \bbR^+:\quad r{\Delta_k^2}
		\geq
		\Csto L^2
\left[
	\left(
		\sqrt{p \Log{\frac{n}{r\delta} }} \land
		\left(s_k\Log{\frac{ep}{s_k} }\right)
	\right)
	+
	\Log{ \frac{n}{r\delta}}
\right]\right\}.
\end{equation}
}


\subsection{Dense change-points with sub-Gaussian noise}
\label{sub:tests_in_the_dense_regime}

Recall that, for a change-point $\tau_k$, $s_k$ stands for the sparsity of the difference $\mu_{k+1}-\mu_k$.
We focus here on dense change-points for which $s_k$ is possibly as large as $p$. Given $\kappa>0$,  $\tau_k$ is a $\kappa$-dense high-energy change-point if
\begin{equation}\label{eq:Energydsg}
	r_k \Delta_k^2\geq \kappa L^2
	\left(
		\sqrt{p \Log{\frac{n}{r_k\delta}}}
		+
		\Log{\frac{n}{r_k\delta}}
	\right) \enspace.
\end{equation}
This condition is very similar to its counterpart~\eqref{eq:EnergyDenseGauss} for Gaussian noise. Still, we introduce it here for the sake of completeness.  For $k\in [K]$ such that $\tau_k$ is a $\kappa$-dense high-energy change-point, we define $\rgD_k$ as the minimum length such that an inequality similar to~\eqref{eq:Energydsg} is satisfied :
\begin{align*}
	\rgD_k
	=
	\min
	\left\{
		r \in \bbN^* :
		\quad 4r{\Delta_k^2}  \geq \kappa L^2\NoiseBoundsg<p,n,r\delta>
	\right\} \enspace.
\end{align*}
As in the Gaussian case in \Cref{sec:multi_scale_change_point_estimation_in_gaussian_noise}, $\rgD_k$  corresponds to the smallest scale such that $\tau_k$ is guaranteed to be detected. 
For any $\kappa$-dense high-energy change-point, it holds that
$4(\rgD_k-1) < r_k$.
For any positive integers $(l,r)\in \cG_F$, we consider the same  CUSUM-based statistic
$\psiD_{l,r}:= \norm{\bC_{l,r}}^2 -  p$ as for Gaussian noise.
Let $\thresh^{(\mathrm{d})}>0$ be a tuning parameter to be discussed later. To calibrate the corresponding multiple test procedures $(\TD_{l,r})$ with $(l,r)\in \cG_F$ rejecting for large values of $\psiD_{l,r}$ we introduce
\begin{align*}
\TD_{l,r} := \1\left\{ \psiD_{l,r}> \tD_r \right\}\ ;\quad \quad 	\tDsg_r
 = \thresh^{(\mathrm{d})} \frac{L^2}{\sigma^2} \NoiseBoundsg<p,n,r\delta> \enspace.
\end{align*}


\begin{proposition}\label{prop:densesg}
	There exists a numerical constant $\thresh^{(\mathrm{d})}>0$ such that the following holds for any $\kappa_{\mathrm{d}} > 32\thresh^{(\mathrm{d})}$. With probability higher than $1-\delta$, one has (i) $\TDsg_{l,r} = 0$ for all $(l,r) \in \cG_F\cap \cH_0$ and (ii) $\TDsg_{\tau_k, \rgD_k}= 1$ for all  $\kappa_{\mathrm{d}}$-dense high-energy change-points $\tau_k$.
\end{proposition}

In comparison to \Cref{prop:dense} in the previous section, there are two differences. First, we need to cope with sub-Gaussian distribution by applying the Hanson-Wright inequality. Most importantly, the grid $\cG_F$ is much larger than $\cG_D$ so that we cannot simply consider each test $T_{l,r}$ separately and simply apply a union bound as in the previous section. To handle the dependencies between the statistics $\psiD_{l,r}$, we have to apply a chaining argument. In fact, the thresholds $\tDsg_r$ are similar to their counterpart in the previous section, whereas the number $|\cG_F|$ of tests is now proportional to $n^2$. In principle, the benefit of using the full grid $\cG_F$ is that  $(\tau_k,\rgD_k)$ belongs to $\cG_F$ so that we can consider the CUSUM statistic based on a segment $[\tau_k- \rgD_k, \tau_k+ \rgD_k]$ centered around the change-point $\tau_k$. In contrast, $(\tau_k,\rgD_k)$ does not necessarily belong to the dyadic grid $\cG_D$ and we needed to consider its best approximation $(\tgD_k,\rgD_k)$. The segment $[\tgD_k-\rgD_k, \tgD_k+\rgD_k]$ is therefore not centered on $\tau_k$ and the corresponding statistic $\psiD_{\tgD_k,\rgD_k}$ is in expectation smaller than $\psiD_{\tau_k,\rgD_k}$. In summary, both the collections of dense tests $\psiD_{l,r}$ on $\cG_D$ and $\cG_F$ are able to detect
change-points whose energy is, up to some multiplicative constants, higher than $L^2[[p \log(\tfrac{n}{r_k\delta})]^{1/2}+ \log(\tfrac{n}{r_k\delta})]$.

\subsection{Sparse change-points with sub-Gaussian noise}
\label{sub:tests_in_the_sparse_regime}

Unlike in the Gaussian case, we do not know the exact distribution of the noise. As a consequence, the Berk-Jones test and more generally higher-criticism type tests cannot be applied to this setting. This is why we only rely on the partial norm statistic.
Recall that $\cZ = \left\{1, 2,2^2,\dots,2^{\floor{\log_2(p)}} \right\}$ stands for a  dyadic set of sparsities.
For $(l,r)\in \cG_F$ and $s\in \cZ$, we also recall that the partial CUSUM norm is defined as $\Ssup_{l,r,s}=  \sum_{i=1}^s\left(\bC_{l,r,(i)}\right)^2$.
Then, for any $(l,r)\in \cG_F$, the test $\Tsup_{l,r}$ rejects the null when at least one of the partial norms is large
\begin{align*}
	\tDmsg_{r,s}
	 = s+ \thresh^{(\mathrm{p})} \frac{L^2}{\sigma^2}
	   \left[
   			s \Log{\frac{2ep}{s} } + \Log{\frac{n}{r\delta}}
	   \right] ;  \quad \quad 	\Tsup_{l,r}
	 	 = \max_{s \in \cZ}\1\left\{\Ssup_{l,r,s} > \tsup_{r,s}\right\} \enspace ,
\end{align*}
where $\thresh^{(\mathrm{p})}$ is a tuning parameter in \Cref{prop:sparsesg} below.
The partial norm test alone is not able to detect sparse high-energy change-points in the sense of~\eqref{eq:EnergySparseGauss} and we need to introduce a stronger condition on the energy. Given $\kappa>0$,  a change-point $\tau_k$ is a $\kappa$-sparse high-energy change-point in the sub-Gaussian setting if $s_k \leq [p \log(\tfrac{n}{r_k\delta} )]^{1/2}$ and
\begin{equation}\label{eq:Energymsg}
	r_k \Delta_k^2
	\geq
	\kappa L^2
	\left[
		s_k \Log{\frac{ep}{s_k}} + \Log{\frac{n}{r_k\delta}}
	\right]\enspace .
\end{equation}
Both Conditions \eqref{eq:EnergySparseGauss} and \eqref{eq:Energymsg} are compared at the end of the subsection. For a  $\kappa$-sparse high-energy change-point $\tau_k$, we define its scale $\rgSp_k$ by
\begin{equation}
	\rgSp_k
	=
	\min
	\left\{
		r \in \bbN^* :
		\quad 4r{\Delta_k^2} \geq
		\kappa L^2\left[
			s_k \Log{\frac{ep}{s_k} } + \Log{\frac{n}{r\delta}}
		\right]
	\right\} \enspace.
\end{equation}
For any $\kappa$-sparse high-energy change-point, it holds that $4(\rgSp_k - 1) \leq r_k$.

\medskip


\begin{proposition}\label{prop:sparsesg}
	There exists a numerical constant $\thresh^{(\mathrm{p})}>0$ such that the following holds for any $\kappa_{\mathrm{s}} > 32\thresh^{(\mathrm{p})}$. With probability higher than $1-\delta$, one has (i) $\TDmsg_{l,r} = 0$ for all $(l,r) \in \cG_F\cap \cH_0$ and (ii) $\TDmsg_{\tau_k, \rgSp_k}= 1$ for all  $\kappa_{\mathrm{s}}$-sparse high-energy change-point $\tau_k$ in the sense of \eqref{eq:Energymsg}.
\end{proposition}
As for \Cref{prop:densesg}, the proof relies on a careful analysis of the joint distributions of the statistics $\Ssup_{l,r,s}$ to handle the multiplicity of $\cG_F$.

\subsection{Consequences}\label{subsection:consequences}

Let $\Csto>0$ be some constant that we will discuss later. A change-point $\tau_k$ is then said to be a $\Csto$-high-energy change-points --in the sub-Gaussian setting-- if
\beq\label{eq:Energysg}
 r_k \Delta_k^2
\geq
\Csto L^2
\left[
	\left(
		\sqrt{p \Log{\frac{n}{r_k\delta} }} \land
		\left(s_k\Log{\frac{ep}{s_k} }\right)
	\right)
	+
	\Log{ \frac{n}{r_k\delta}}
\right] \enspace .
\eeq
We here re-introduce $\cK^*\subset [K]$  as the subset of indices such that $\tau_k$ satisfies~\eqref{eq:Energysg}.

We gather both tests by considering, for any $(l,r)\in \cG_F$, the test $\Tg_{l,r} = \TDsg_{l,r}\lor \TDmsg_{l,r}$
with tuning parameters $\thresh^{(\mathrm{d})}$ and $\thresh^{(\mathrm{p})}$ as in Propositions~\ref{prop:densesg} and \ref{prop:sparsesg}. Consider any $\Csto > 32(\thresh^{(\mathrm{d})}\lor \thresh^{(\mathrm{p})})$ and any $\Csto$-high-energy change-point $\tau_k$, which is either a $\Csto$-sparse or a $\Csto$-dense high-energy change-point. Defining
\begin{equation}\label{eq:rkbarsg}
	\rg_k = \rgD_k \land \rgSp_k,
\end{equation}
we straightforwardly derive from \Cref{prop:densesg} and \Cref{prop:sparsesg} the following result.
\begin{corollary}\label{cor:test_full_sg}
	There exists two numerical constants $\thresh^{(\mathrm{p})}>0$ and $\thresh^{(\mathrm{d})}>0$ such that the following holds. With probability higher than $1-\delta$, it holds that (i) $\Tg_{l,r} = 0$ for all $(l,r) \in \cG_F\cap \cH_0$ and (ii) $\Tg_{\tau_k, \rg_k}= 1$ for any  $\Csto$-high-energy change-point $\tau_k$ in the sense of \eqref{eq:Energysg}.
\end{corollary}

Then, it suffices to combine this multiple testing procedure  with \Cref{algo_1} to get the  change-point procedure $\widehat{\tau}$. Since, for a high-energy change-point in the sense of \eqref{eq:Energysg}, we have $4(\rg_k-1) <  r_k$, we are in position to apply \Cref{th:g}.

\begin{corollary}\label{cor:subgaus}
	There exist two numerical constant $\thresh^{(\mathrm{p})}>0$ and $\thresh^{(\mathrm{d})}>0$ such that the following holds. With probability higher than $1-\delta$, the estimator	$\widehat{\tau}$ satisfies ({\bf NoSp}) and {\bf detects}  $\Csto$-high-energy change-point $\tau_k$ (as defined in\eqref{eq:Energysg}), that is
\[
	d_{H,1}(\widehat{\tau},  \tau_{k})\leq \rg_k - 1  \leq \frac{r_{k}}{4}\enspace ,
\]
where  $\rg_k$ is defined in \eqref{eq:rkbarsg}.
\end{corollary}
In the case where all change-points are $\Csto$-high-energy change-points in the sense of \eqref{eq:Energysg}, all of them are detected,  and a result similar to \Cref{cor:gaus_hausdorff} holds here, replacing $r_k^*/2$ by $\rg_k - 1$. Also, both the Hausdorff distance and the Wasserstein distance, can be bounded as in Equation~\eqref{eq:HandW} if we  replace $r_k^*/2$ by $\rg_k - 1$.

As already stated, we could have obtained a similar result (but with different constants) using the dyadic grid $\cG_D$ instead of $\cG_F$.
To conclude this section, let us compare the conditions \eqref{eq:Energysg}
and \eqref{EnergyGauss} for high-energy. Define $$\psi^{(sg)}_{n,r,s}=
		\sqrt{p \lo} \land
		\left(s\Log{\frac{ep}{s} }
	\right)
	+
	\lo\enspace ,$$
	where we recall that $\lo = \Log{ \frac{n}{r\delta}}$.
If $\lo\geq  p/2$, then $\psi^{(sg)}_{n,r,s}\asymp \lo$. In low dimension, the energy threshold for multivariate change-point detection is the same as in the univariate setting, see \cite{verzelenoptimal_change_point}. If $\lo\leq p/2$, then
\begin{equation*}
\psi^{(sg)}_{n,r,s}\asymp \left\{
\begin{array}{cc}
	\lo &\text{ if } s \leq \frac{\lo}{\log(p)- \log\left(\lo\right)} \\
	s \Log{e\frac{p}{s} }& \text{ if }\frac{\lo}{\log(p)- \log\left(\lo\right)} < s   < \frac{\sqrt{p\lo}}{\log(p)-\log(\lo)}
	\\
	\sqrt{p\lo} &\text{ if } s \geq \frac{\sqrt{p\lo}}{\log(p)-\log(\lo)}
\end{array}
\right.
\end{equation*}

As a consequence, $\psi^{(sg)}_{n,r,s}$ and $\psi^{(g)}_{n,r,s}$ are of the same order of magnitude for all $s$ when $\lo\geq p/2$. When  $\log(n/r\delta)< p$, they are also of the same order of magnitude except when $s$ is close but smaller than $\sqrt{p\lo}$, for which the ratio $\psi^{(sg)}_{n,r,s}/\psi^{(g)}_{n,r,s}$  between these two quantities can be as large as $\log(p)- \log(\lo)$. This gap corresponds to the regime where the test based on the Berk-Jones statistic defined in Equation~\eqref{eq:def_berk_jones}, used in the Gaussian case, outperforms the test based on the partial CUSUM norm statistic defined in Equation~\eqref{eq:partialCUSUMnorm}.

\medskip

In the definitions of the tests, the tuning constants $\thresh^{(\mathrm p)}$ and $\thresh^{(\mathrm d)}$ are left implicit, although one can find suitable values by following the proofs of Propositions~\ref{prop:densesg} and~\ref{prop:sparsesg}. In practice, the practitioner can calibrate them by a Monte-Carlo method by simulating a Gaussian multivariate times series without any change-points. Then, $\thresh^{(\mathrm p)}$ and
$\thresh^{(\mathrm d)}$ are chosen so that the Family-wise error rate (FWER) of the two collections $(\TD_{l,r})$ and  $\Tsup_{l,r}$ is equal to $\delta$.

\noindent
 {\bf Computational Cost}.
The computational cost of the statistic $T_{l,r} = \TD_{l,r} \lor \Tsup_{l,r}$ is $O(p(r+\log(p)))$. Thus, a naive computation of all the tests $T_{l,r}$ for $(l,r)$ in the complete grid $\cG_F$ requires $O(p\log(p)\sum_{(l,r) \in \cG_F}r) = O(pn(n^2+\log(p)))$ operations. Nevertheless, using the fact that $\sum_{i = l+1}^{l+r}Y_i = (\sum_{i = l}^{l+r-1}Y_i) + Y_{l+r} - Y_l$, it is possible to compute all the tests at scale $r$ with cost $O(np\log(p))$. Since there are $n$ possible scales $r$ on the complete grid, the whole procedure  cost is $O(n^2 p\log(p))$. Using a grid $\cG = \{ (l,r) \in \cG_F ~:~ r \in \cR \}$ that contains dyadic scales and all possible locations $l$ for each scale, the whole change-point detection would then require only $O(np \log(n)\log(p))$ computations, since there are only $\log(n)$ possible scales $r$ for such grids.


\def\NoiseBound<#1,#2,#3>{\left( \sqrt{#1 \Log{\frac{#2}{#3}}} + \Log{ \frac{#2}{#3}} \right)}
\def\NoiseBoundSp<#1,#2,#3,#4>{\left( #4\Log{\frac{e #1\Log{\frac{e #2}{#3}}}{#4 ^2}} + \Log{ \frac{e #2}{#3}} \right)}
\def\proscal<#1,#2>{\langle #1,#2\rangle}

\section{Minimax lower bound}\label{sec:lower}

In this section, we write for any $\Theta \in \mathbb R^{p\times n}$, the distribution  of the time series $Y=(y_1,\ldots, y_n)$ in the model
\eqref{eq:model_multivariate} with Gaussian noise  $\Noise_t \sim \mathcal N(0, \sigma^2\mathbf I_p)$. In  \Cref{sec:multi_scale_change_point_estimation_in_gaussian_noise}, we have established that any change-point satisfying the condition \eqref{EnergyGauss}, that is
    \begin{align*}
        r_k \Delta_k^2
	\geq
	\Csto\sigma^2\left[s_k \Log{1 + \frac{\sqrt{p}}{s_k} \sqrt{\Log{\frac{n}{r_k \delta }}}}
	+ \Log{\frac{ n}{r_k \delta }}\right] \enspace,
    \end{align*}
is detected by our change-point procedure.  We now show that this energy condition is unimprovable from a minimax point of view.
More precisely, let us define, for any $u>0$, the class $\Pginf$ of mean parameters $\Theta$ with arbitrary $K\geq 0$ number of change points and such that any change-point $\tau_k$ for $1\leq k\leq K$ satisfies
\begin{equation}\label{eq:LBen}
	r_k \Delta_k^2 \geq \frac{1}{2}\sigma^2\left[s_k \Log{1 + u\frac{\sqrt{p}}{s_k} \sqrt{\Log{ \frac{n}{r_k }}}} +  u\Log{\frac{n}{r_k }}\right]\enspace.
\end{equation}
For $u$  small enough, it turns out no change-point estimator is able to detect all change-points without estimating any spurious change-point with high probability on the full class $\Pginf$. Still, using this large class provides somewhat pessimistic bounds. For instance, the most challenging distributions  in
$\Pginf$ for the purpose of change-point detection satisfy $s_k=p$ and $r_k=1$ (very close change-points). As a consequence, relying on the full collection $\Pginf$
turns too pessimistic. To establish that our bounds are adaptive with respect to the sparsity $s_k$ and the length $r_k$, we define, for any positive integers $1\leq r\leq  \lfloor n/2\rfloor $ and any $1\leq s\leq p$ the collection
$$ \Pginfrs = \{ \Theta \in \Pginf ~:~  \min_k r_k \geq r \mbox{ and } \max_k s_k \leq s \}\enspace .$$
By convention, constant means $\Theta$ with no change-points ($K=0$) also belong to $\Pginfrs$. In the class $\Pginfrs$, all change-points have a sparsity at most $s$ and a length at least $r$. Hence, $\Pginfrs$ becomes larger when $s$ increases or when $r$ increases. 

\begin{theorem}\label{th:borneinf}
Fix any $u\in (0,1/8)$.  For any $\sigma>0$, $n\geq 2$, $p\geq 1$, any length $1\leq r\leq n/4$, and any sparsity $ 1\leq s\leq p$, we have
$$ \underset{\hat \tau}{\inf}\underset{\Theta \in \Pginfrs}{\sup}\Prob_\Theta( \hat K \neq K) \geq \frac{1}{4}\enspace , $$
where  the infimum is taken over all estimators $\hat \tau$ of the change-point vector $\tau$ and and $\hat{K}= |\hat{\tau}|$.
\end{theorem}
Thus, in the Gaussian setting, if all the change-points have a high-energy in the sense of \Cref{EnergyGauss} but with a smaller multiplicative constant factor, no change-point estimator can consistently estimate the true number of change-points. The next corollary restates this negative results in the same lines as  \Cref{cor:gaus_hausdorff}.
\begin{corollary}\label{cor:LB} Fix any $u\in (0,1/8)$.
	For any $\sigma>0$, $n\geq 2$, $p>1$, any length $1\leq r\leq n/4$, any sparsity $ 1\leq s\leq p$, and any estimator $\hat \tau$, there exists some $\Theta \in \Pginfrs$ such that with $\mathbb P_{\Theta}$-probability larger than $1/4$, at least one of the two following properties is satisfied
	\begin{itemize}
		\item $\hat \tau$ contains at least one {\bf spurious} change-point
		\item at least a change-point $\tau_k$ with $1\leq k \leq K$ is not {\bf detected}, i.e.~there is no change-point estimated in the interval $[ (\tau_{k-1}+ \tau_{k})/2,(\tau_{k} +\tau_{k+1})/2 ]$.
	\end{itemize}
\end{corollary}
This corollary is to be compared to \Cref{cor:gaus_hausdorff} - indeed, the energy condition in Equation~\eqref{eq:LBen} differs from Equation~\eqref{EnergyGauss} only by a numerical multiplicative constant. As a consequence, the energy  condition~\eqref{eq:LBen} is minimal for detection by a change-point estimator that achieves ({\bf NoSp}).

%
\section{Application to other change-point problems}\label{sec:othermodels}

In this section, we apply the general methodology of Section~\ref{sec:meta_algorithm_for_multi_scale_change_point_detection_on_a_grid} to two other problems, namely detection of covariance  and nonparametric change-points. This allows us to obtain the first tight minimax detection conditions for these problems.

\subsection{Covariance change-point detection}\label{ss:cov}

Following Wang et al.~\cite{wang2017optimal},  we consider the covariance change-point model where the covariance matrices $\Sigma_t$ of the centered random vectors $y_t\in \mathbb{R}^p$ are piece-wise constant. Then, the goal is to estimate the times $0<\tau_1<\ldots<\tau_{K} <\tau_{K+1}=n+1$ such that $\Sigma_t$ is varying.  See~\cite{wang2017optimal} for motivations. As in that work, we assume that the random vectors $y_t$ are independent and are sub-Gaussian with a uniformly bounded  Orlicz norm, that is
$\max_{t=1,\ldots, n}\|y_t\|_{\psi_2}\leq B$ for some known fixed $B$. The Orlicz norm of a random vector $y$ is the supremum of the Orlicz norm of any uni-dimensional projection of $y$  -- see e.g.~\cite{vershynin}. If the $y_t$'s follow a normal distribution, this amounts to assuming that $\max_{t=1,\ldots, n}\|\Sigma_t\|_{op}\leq 2B^2$  where $\|.\|_{op}$ is for the operator norm. The purpose of Wang et al.
was to detect small changes in operator norm, that is detecting instants $\tau_{k}$  such that $\Sigma_{\tau_{k}}\neq \Sigma_{\tau_{k}-1}$ with $\|\Sigma_{\tau_{k}}-\Sigma_{\tau_{k-1}}\|_{op}$ possibly small. Apart from the operator norm, other norms have also been considered e.g. in~\cite{dette2020estimating}. Here, we focus on the operator norm as in~\cite{wang2017optimal}.

Recalling the generic procedure introduced in Section~\ref{sec:meta_algorithm_for_multi_scale_change_point_detection_on_a_grid}, we consider the dyadic grid $\cG_D$ and some $\delta\in (0,1)$. For any $(l,r)\in \mathcal{G}$, we respectively write $\widehat{\Sigma}_{l,-r}$ and $\widehat{\Sigma}_{l,r}$ for the empirical covariance matrices
\[
\widehat{\Sigma}_{l,-r}= r^{-1}\sum_{t=l-r}^{l-1}y_ty_t^{T}\ ; \quad \quad    \widehat{\Sigma}_{l,r}= r^{-1}\sum_{t=l}^{l+r-1}y_ty_t^{T} \enspace .
\]
Then, we consider the test $T_{l,r}$ rejecting for large  values of $\|\widehat{\Sigma}_{l,r}- \widehat{\Sigma}_{l,-r}\|_{op}$.
\beq\label{eq:test_covariance}
T_{l,r}= \1\left\{\|\widehat{\Sigma}_{l,r}- \widehat{\Sigma}_{l,-r}\|_{op}\geq c_0 B^2 \left[\sqrt{\frac{p}{r}} + \frac{p}{r}+ \sqrt{\frac{\log(\frac{2n}{\delta r})}{r}} + \frac{\log(\frac{2n}{\delta r})}{r}\right]\right\}\ ,
\eeq
where the numerical tuning constant $c_0$ is set in the proof of the following proposition. Relying on concentration bounds~\cite{koltchinskii2017concentration} for the empirical covariance matrix of sub-Gaussian random vectors, we easily  prove that the FWER of the multiple testing procedure $(T_{l,r})$ with $(l,r)\in \mathcal{G}_D$ is small. Then, we can analyze the type II error probability and plug it into the generic result (Theorem~\ref{th:g}) to control the behavior of the change-point estimator $\widehat{\tau}$. This leads us to the following result.
In the sequel, a change-point $\tau_k$ is said to have a high-energy if
\beq\label{eq:high_energy_covariance}
 r_k \|\Sigma_{\tau_{k}}- \Sigma_{\tau_{k-1}}\|^2_{op}\geq c_1 B^4 \left[\left(p+ \log\left(\frac{2n}{r_k\delta}\right)\right) \wedge  r_k \right]\ ,
\eeq
where the numerical constant $c_1$ is introduced in the proof of the following proposition. We recall that, by definition of the model, we have $\|\Sigma_{\tau_{k}}- \Sigma_{\tau_{k-1}}\|_{op} \leq 4B^2$.

\begin{proposition}\label{prop:covariance}
 There exist positive numerical constants $c_0$, $c_1$, and $c_2$ such that the following holds for any $B>0$ and any sequence of independent centered random vectors ($y_t$) satisfying $\max_{t}\|y_t\|_{\psi_2}\leq B$. With probability higher than $1-\delta$, the change-point estimator $\widehat{\tau}$ satisfies ({\bf NoSp}) and {\bf detects} all high-energy change-points in the sense of~\eqref{eq:high_energy_covariance}.  Besides, any such high-energy change-point $\tau_k$  satisfies
\beq\label{eq:error_localization}
d_{H,1}(\widehat{\tau},\tau_k)\leq c_2 B^4\frac{p+ \log\left(2\delta^{-1}B^{-4}n\|\Sigma_{\tau_{k}}- \Sigma_{\tau_{k-1}}\|^2_{op}\right)}{\|\Sigma_{\tau_{k}}- \Sigma_{\tau_{k-1}}\|^2_{op}} \leq \frac{r_k}{4}\ ,
\eeq
under the same event of probability than $1-\delta$. 
\end{proposition}

Let us compare our condition~\eqref{eq:high_energy_covariance} for detection with Theorem 2 in  Wang et al.~\cite{wang2017optimal}. The authors assume that all the change-points satisfy
\[
 \min_k r_k  \min_k \|\Sigma_{\tau_{k}}- \Sigma_{\tau_{k-1}}\|^2_{op}\geq c'_1 B^4 p\log(n) \enspace .
\]
In addition to the fact that we allow some change-points to have an arbitrarily low energy, our requirement for detection scales like $\sqrt{p}+\sqrt{\log(n/r_k)}$ instead of
$\sqrt{p\log(n)}$.

The next proposition establishes that the latter condition is minimal. By homogeneity, we can only consider the case where $B=3/2$.
We focus our attention on Gaussian distributions so that the distribution of the sequence $(y_1,\ldots, y_n)$ is uniquely defined by the sequence $(\Sigma_1,\ldots, \Sigma_n)$ of covariance matrices. Given an integer $1\leq r\leq n/4$ and $\zeta \in (0,1/\sqrt{2})$, we define $\bar{\cP}(r,\zeta)$ the collection of sequences $\eta=(\Sigma_1,\ldots,\Sigma_n)$ of covariance matrices that satisfy either $\Sigma_t=I_p$ or $\|\Sigma_t\|_{op}= 1+\zeta$. Besides, the corresponding change-points ($\tau_1,\ldots, \tau_K)$ of $\eta$ must  satisfy $\min_{k}r_k\geq r$ and $\min_{k}\|\Sigma_{\tau_k}-\Sigma_{\tau_{k-1}}\|_{op}\geq \zeta$. For $\eta \in \bar{\cP}(r,\zeta)$, we write $\P_{\eta}$ for the corresponding distribution of $(y_1,\ldots, y_n)$.

\begin{proposition}\label{prp:lower:covariance}
There exists a  positive numerical constant $c$ such that, for any $n$, $p$ and any length $1\leq r\leq n/4$ the following holds. Provided that
$r\zeta^2 \leq c(p+ \log(n/r))\wedge \frac{r}{2}$, we have
\[
\underset{\hat \tau}{\inf}\underset{\eta \in \bar{\cP}(r,\zeta)}{\sup}\Prob_\eta( \hat K \neq K) \geq \frac{1}{4}.
\]
\end{proposition}

As a consequence, our procedure $\widehat{\tau}$ achieves the minimal separation condition~\eqref{eq:high_energy_covariance} for change-point detection.
In their work,~\cite{wang2017optimal} obtain faster localization errors than~\eqref{eq:error_localization} to the price of stronger separation conditions. Our focus in this work is to provide optimal detection conditions and we did not try to optimize~\eqref{eq:high_energy_covariance}.

\subsection{Univariate nonparametric change-point  detection}

We now turn to the univariate nonparametric change-point model considered in~\cite{padilla2019optimal1}. Let  $m\geq 1$ be any positive integer.  At each time $t=1,\ldots, n$, the random vector $y_t$ is an $m$-sample of a univariate distribution with cumulative distribution function $F_t$. Then, we aim at detecting a vector $\tau=(\tau_1,\ldots,\tau_K)$ of change-points such that $F_{\tau_k}\neq F_{\tau_{k-1}}$. As in~\cite{padilla2019optimal1}, we quantify the distance between two distributions by the Kolmogorov distance $\|F_1- F_2\|_{\infty}= \sup_{z\in \mathbb{R}}|F_1(z)-F_2(z)|$.

As in the previous subsection, we build a procedure $\widehat{\tau}$ with our generic algorithm on the dyadic grid. Regarding the collection of tests $(T_{l,r})$, we consider two-sample Kolmogorov-Smirnov tests. More precisely, we denote $\widehat{F}_{t}$ the empirical distribution function associated with the sample $y_t$ and we define the test
\[
T_{l,r}=\1\left\{\left\|r^{-1}\left( \sum_{t=l}^{l+r-1}\widehat{F}_{t}- \sum_{t=l-r}^{l-1}\widehat{F}_{t}\right)\right\|_{\infty}\geq \sqrt{2\frac{\log(4n/(\delta r))}{mr}} \right\}\enspace .
\]
In the following, a change-point $\tau_k$ is said to have a high-energy if
\beq\label{eq:high_energy_non_param}
 r_k  \|F_{\tau_{k}}- F_{\tau_{k-1}}\|^2_{\infty}\geq \frac{c_1}{m}
\log\left(\frac{n}{r_k\delta}\right)\ ,
\eeq
where the numerical constant $c_1$ is introduced in the proof of the next proposition. As in Subsection~\ref{ss:cov}, it is straightforward to prove, based on Dvoretzky–Kiefer–Wolfowitz inequality, that the FWER of the multiple testing procedures $(T_{l,r})$ with $(l,r)\in \mathcal{G}_D$ is small. Then, we analyze the type II error probability of this test and plug it into the generic result (Theorem~\ref{th:g}) to control the behavior of the change-point estimator $\widehat{\tau}$.

\begin{proposition}\label{prop:non_parametric}
There exist positive numerical constants  $c_1$ and $c_2$ such that the following holds. With probability higher than $1-\delta$, the change-point estimator $\widehat{\tau}$ satisfies ({\bf NoSp}) and {\bf detects} all high-energy change-points $\tau_k$ in the sense of~\eqref{eq:high_energy_non_param}. Besides, any such high-energy change-points $\tau_k$ satisfies
\beq\label{eq:error_localization_non_param}
d_{H,1}(\widehat{\tau}_{k'},\tau_k)\leq c_2  \frac{\log\left(\delta^{-1}nm\|F_{\tau_{k}}- F_{\tau_{k-1}}\|^2_{\infty}\right)}{m\|F_{\tau_{k}}- F_{\tau_{k-1}}\|^2_{\infty}} \leq \frac{r_k}{4}\ ,
\eeq
under the same event of probability than $1-\delta$. 
\end{proposition}

In~\cite{padilla2019optimal1}, the authors introduce a procedure detecting all the change-points  provided that
\[
\min_k r_k  \min_k\|F_{\tau_{k}}- F_{\tau_{k-1}}\|^2_{\infty}\geq c_1\frac{\log(n)}{m}\enspace .
\]
Comparing this last condition with~\eqref{eq:high_energy_non_param}, we observe that our logarithmic term  is tighter and that we allow arbitrarily low-energy change-points.

The next proposition establishes that the condition~\eqref{eq:high_energy_non_param} is unimprovable. Given an integer $1\leq r\leq n/4$ and $\zeta\in (0,1/4)$, we focus our attention  on the collection $\bar{\cP}(r,\zeta)$ of sequences $(F_1,\ldots, F_n)$ of distributions such that the corresponding change-points ($\tau_1,\ldots, \tau_K)$ satisfy $\min_{k}r_k\geq r$ and $\min_{k}\|F_{\tau_k}-F_{\tau_{k-1}}\|_{\infty}\geq \zeta$. For $\eta \in \bar{\cP}(r,\zeta)$, we write $\P_{\eta}$ for the corresponding distribution of the sequence $(y_1,\ldots,y_n)$.

\begin{proposition}\label{prop:lower_non_parametric}
	There exists a  positive numerical constant $c$ such that, for any $n$, $p$ and any length $1\leq r\leq n/4$ the following holds. Provided that
	$r\zeta^2 \leq c'\log(n/r)/m$, we have
	\[
	\underset{\hat \tau}{\inf}\underset{\eta \in \bar{\cP}(r,\zeta)}{\sup}\Prob_\eta( \hat K \neq K) \geq \frac{1}{4}.
	\]
\end{proposition}


\section{Discussion}\label{sec:discussion}

\subsection{Noise distribution for multivariate change-point detection}

\paragraph*{Comparison between Gaussian and sub-Gaussian rates.} In this work, we have studied two types of noise distribution: Gaussian (\Cref{sec:multi_scale_change_point_estimation_in_gaussian_noise}) and general sub-Gaussian distributions (\Cref{sec:multi_scale_change_point_estimation_in_sub_gaussian_noise}) without further knowledge on the distribution functions.  Since the Gaussian setting is a specific instance of the sub-Gaussian setting, it is clear that the minimax lower bounds from \Cref{sec:lower} apply in both settings. As described in the previous subsection, the performances in the sub-Gaussian case almost match those in the Gaussian setting except for $s_k$ slightly lower but close to $\sqrt{p \log(en/r_k)}$. Indeed, in that regime, Berk-Jones or Higher-Criticism type statistics heavily rely on the probability distribution function of the noise, which is not available in the general sub-Gaussian case. Still, we could slightly improve the sub-Gaussian rates if we further assume that the noise components are identically distributed with common CDF $F$.
\begin{itemize}
\item If $F$ is known (know noise distribution), then one may adapt Berk-Jones test by replacing $\bar \Phi(x)$ in Equation~\eqref{eq:def_berk_jones} by $F(-x) + (1 - F(x))$. This would allow us to recover
the exact same detection condition as in the Gaussian setting.
\item If $F$ is unknown and if there are not too many change-points, one could hope to estimate the quantiles of the CUSUM statistic at each scale $r$ and plug them into a Berk-Jones statistics. This goes however beyond the scope of this paper.
\end{itemize}

\paragraph*{Unknown variance or more general variance matrix.} We assumed in the sparse multivariate sections that the variance  $\sigma^2$ is known. Whereas the partial norm test only requires the knowledge of an upper bound on $\sigma$, the dense statistic $\Psi_{l,r}^{(\mathrm{d})}$ requires the exact knowledge of the variance.
As soon as there are not too many change-points, it is possible to roughly estimate $\sigma$ and therefore accommodate the partial norm test with an unknown variance. In contrast, the dense statistics needs to be replaced by a $U$-statistics. Consider any even positive integer $r$ and define
\[
\widetilde{\bC}_{l,r}(Y) = \frac{\sqrt{r}}{2}\left(\frac{2}{r}\sum_{t = 1}^{r/2} Y_{l-2(t-1)-1} - \frac{2}{r}\sum_{t =1 }^{r/2} Y_{l+2(t-1)} \right)\ ,\quad
\widetilde{\bC}'_{l,r}(Y) = \frac{\sqrt{r}}{2}\left(\frac{2}{r}\sum_{t = 1}^{r/2} Y_{l-2t} - \frac{2}{r}\sum_{t =1 }^{r/2} Y_{l+2(t-1)+1} \right) \enspace ,
\]
where  $\widetilde{\bC}_{l,r}(Y)$ and $\widetilde{\bC}'_{l,r}(Y)$ are independent. If there is one change-point at position $l$ and no other change-points in $(l-r,l+r)$, then these statistics are identically distributed and we consider  $\widetilde{\Psi}_{l,r}^{''\mathrm{(d)}} = \langle \widetilde{\bC}_{l,r}(Y), \widetilde{\bC}_{l,r}'(Y)\rangle$ whose expectation is null when there are no change-points in the segment. As a consequence, $\widetilde{\Psi}_{l,r}^{''\mathrm{(d)}}$ does not require the knowledge of $\sigma$; only an upper bound of $\sigma$ is required to calibrate the corresponding test. Such a $U$-statistics has already been introduced in~\cite{wang2019inference} and analyzed in an asymptotic setting.  Unfortunately, since we can only consider even $r$, this precludes us to detecting change-points that are very close together with $r_k=1$.

In the general case where there is spatial covariance in the noise, that is $\mathrm{var}(\epsilon_t)=\Sigma$ for an unknown but general $\Sigma$, we can still use the same $U$-statistic described in the previous paragraph for the dense case. For the sparse case, one could use the supremum norm of the CUSUM statistics as in
Jirak~\cite{jirak2015uniform} and Yu and Chen~\cite{yu2017finite}. To calibrate those tests, we need to estimate both the Frobenius and the operator norm of $\Sigma$, which seems to be doable as soon as there are not too many change-points. If the spatial covariance matrix $\mathrm{var}(\epsilon_t)$ is unknown and even allowed to change with time, we suspect that the problem becomes intrinsically more involved. 

\subsection{Optimal Localization rates}

In this work, we mainly considered the problem of {\bf detecting} change-points in the mean of a  random vector. We provided tight conditions on the energy so that a change-point is detectable.  When such a change-point $\tau_k$ is detected,  \Cref{cor:gaus} states that its position is estimated up to an error of $r_k^*$, which is also of the order of  $\sigma^2 \Psi^{(g)}_{n,\overline{r}_k,s_k} \Delta_k^{-2}$-- see the definition~\eqref{eq:VitesseGauss}.
It is not clear whether this error is optimal or not.

In the univariate setting ($p=1$),~\cite{verzelenoptimal_change_point} has established that, above the detection threshold, a specific change-point position $\tau_k$ can be localized at the rate $\sigma^2 \Delta_k^{-2}$. In the multivariate setting, the situation is more tricky and there are certainly several localization regimes beyond the detection threshold. It is an interesting direction of research to pinpoint the exact localization rate between $\sigma^2 \Delta_k^{-2}$ and $\sigma^2 \Psi^{(g)}_{n,\overline{r}_k,s_k} \Delta_k^{-2}$. We leave this for future work.

\subsection{On the choice of the grid in the generic algorithm}

Our general procedure is defined for almost any arbitrary grid. Optimal procedures with the dyadic grid are introduced in Sections~\ref{sec:multi_scale_change_point_estimation_in_gaussian_noise} and~\ref{sec:othermodels}, whereas we use a near-optimal procedure on the  complete grid in Section~\ref{sec:multi_scale_change_point_estimation_in_sub_gaussian_noise}. 

From a computational perspective, the procedure's worst-case complexity is proportional to the size $|\cG|$ of the grid $\cG$. In that respect, the dyadic grid and more generally the $a$-adic grids benefit from a linear size whereas the size of the complete grid is quadratic.

From a mathematical perspective, it is much easier to control the behaviour of the procedure for an $a$-adic grid by a simple Bonferroni correction on all the statistics as it turns out that this correction is sufficient for our purpose -- see the proofs of Section~\ref{sec:multi_scale_change_point_estimation_in_gaussian_noise}. In constrast, controlling larger collections of tests turns out to be much more challenging as one needs to carefully take into account the dependences between the test statistics, which becomes all the more challenging for complex models. As an example, we introduced in Section~\ref{sec:multi_scale_change_point_estimation_in_gaussian_noise} Berk-Jones statistics to achieve the tight minimax condition for change-point detection. Unfortunately, we did not manage to apply a suitable chaining argument to these statistics and were therefore unable to control the behavior of the corresponding change-point detection procedure on the complete grid.

From a purely statistical perspective, it is difficult to appreciate the respective benefits of denser or sparser grids. On the one hand, for denser grids, the approximation $\overline{\tau}_k$ of $\tau_k$ at scale $r$ will be closer to $\tau_k$ so that the corresponding test $T_{\overline{\tau}_k,r}$ may be more powerful. On the other hand, for a denser grid, the tests possibly suffer from a higher price for multiplicity. This price can be mild if one takes into account the dependences between the tests. Still, except perhaps in the univariate Gaussian change-point model for which delicate controls of the CUSUM process exist, it is  challenging to provide theoretical guidance towards the best choice of the grid.


\subsection{Optimality of the generic algorithm in a broader context.}

\Cref{algo_1} aggregates homogeneity tests and provides theoretical guarantees on the event $\cA \left( \Tg, \cK^*,(\tg_k,\rg_k)_{k\in \cK^*}\right)$ - i.e.~the event where the outcomes of the tests are consistent - as stated in \Cref{th:g}. In the possibly sparse high-dimensional mean change-point model, we introduced a suitable multiple testing procedure which, when combined with  \Cref{algo_1}, leads to a minimax optimal change-point detection procedure.

\medskip

We described in \Cref{sec:meta_algorithm_for_multi_scale_change_point_detection_on_a_grid} how to adapt this approach to other change-point problems and this was already illustrated in Section~\ref{sec:othermodels} with covariance and nonparametric problems.  One may then wonder whether this roadmap still leads to minimax optimal procedures for general problems. Consider the general setting from~\Cref{sec:introduction} where we are interested in detecting change-points in  $\left(\Gamma \left(\bbP_{t}\right)\right)_{t\in [n]}$.
Upon endowing the space $\cV$ with some distance  $d$, we define, for any $k$,
$$\bar \Delta_k = d\left(\Gamma\left(\bbP_{\tau_k}\right), \Gamma\left(\bbP_{\tau_{k-1}}\right)\right)\enspace ,$$
which corresponds to the change-point height.
Then, one may wonder how large $\bar \Delta_k$ has to be - as a function of $r_k$ - so that a change-point detection procedure achieving the no-spurious property ({\bf NoSp}) with high probability is able to detect $\tau_k$. In this discussion, we restrict our attention to independent observations, that is the random variables $y_t$ are assumed to be independent and we consider the dyadic grid  $\cG_D$.

Fix $\alphabj\in (0,1)$. At each scale $r\in \{1,2,\ldots, 2^{\lfloor \log_2(n)\rfloor-1}\}$ and for each $l\in \cD_r$, with $\cD_r$ defined in \eqref{eq:def_dyadic}, we consider the testing problem $H_{0,l,r}:\{\mathbb{P}:\ \Gamma(\mathbb{P}_{l-r})=\ldots=\Gamma(\mathbb{P}_{l+r-1})\}$  versus
\[
H_{\rho,l,r}:\left\{\mathbb{P}:\begin{array}{c} \Gamma(\mathbb{P}_{l-r})=\ldots =\Gamma(\mathbb{P}_{l-m-1})\\ \Gamma(\mathbb{P}_{l-m})=\ldots=\Gamma(\mathbb{P}_{l+r-1})\\ d(\Gamma(\mathbb{P}_{l-m-1}),\Gamma(\mathbb{P}_{l-m})\geq \rho)
\end{array}
\text{ for some integer }m\in [-r/2,r/2]\right\}
\]
This amounts to testing whether there is a single change-point near $l$ of height at least $\rho$ in the segment $(l-r,l+r)$. Given $\delta\in (0,1)$ and a test $T$ we define the $\delta$-separation distance of $T$ by
\[
\rho_{l,r}^*(T,\delta)=\inf\left\{\rho:\  \sup_{\mathbb{P}\in H_{0,l,r}}\mathbb{P}(T=1)\lor  \sup_{\mathbb{P}\in H_{\rho,l,r}}\mathbb{P}(T=0)\leq \delta\right\}\ .
\]
This corresponds to the minimal change-point height that is detected by the test $T$. Then, the {\it minimax} separation distance $\rho^*_{l,r}(\delta)$ is simply
$\inf_{T}\rho_{l,r}(T,\delta)$, i.e.~the infimum over all tests $T$ of the separation distance. By translation invariance of the testing problem, note that  $\rho^*_{l,r}(\delta)$ does not depend on $l$  and is henceforth denoted $\rho^*_{r}(\delta)$.
\medskip

For any $(l,r)$, take any test $T_{l,r}$ (nearly)\footnote{Since the minimax separation distance is defined as an infimum, it is not necessarily achieved by a test. Still, we can build a test whose separation distance is arbitrarily close to the optimal one. We neglect the additive error term for the purpose of the discussion.} achieving the
minimax separation distance  $\rho^*_{r}(\delta|\mathcal{D}_r|^{-1}\beta_r)$ with $\beta_r= 6\log^{-2}_2(n/r))\pi^{-2}$. Then, it follows from a simple union bound on the dyadic grid that, with probability higher than $1-\delta$, the collection of tests $T_{l,r}$, where $(l,r)$ belongs to the dyadic grid, does not detect any false positive and detects any change-point $\tau_k$ such that $\overline{\Delta}_k$ is higher than $\rho^*_{\tilde{r}_k}(\delta|\mathcal{D}_{\tilde{r}_k}|^{-1}\beta_{\tilde{r}_k})$, where $\tilde{r}_k$ is the largest scale in $\cR$ such that $4(\tilde{r}_{k}-1)\leq r_k$. As a consequence of \Cref{th:g}, the corresponding detection procedure achieves, with probability higher than $1-\delta$, the property $({\bf NoSp})$ and {\bf detects} any change-point satisfying the energy condition  $\overline{\Delta}_k\geq \rho^*_{\tilde{r}_k}(r \delta\beta_r/2n)$.

\medskip

Conversely, we believe that this energy condition is almost tight. Indeed, fix any even range $r\geq 2$. To simplify the discussion suppose that $n/(2r)$ is an integer. We consider a specific instance of the problem where the statistician knows that there are  $n/(2r)-1$ evenly-spaced change-points respectively at $2r+1,4r+1,\ldots, n-2r+1$
that allow to reduce the change-point detection problem to $n/(2r)$ change-point detection problem in intervals $(l-r, l+r]$ for $l=r+1,3r+1,5r+1,\ldots$. Furthermore, it is known that, in each such segment, there exists at most one change-point that is situated in $[l - 0.5r,l+0.5r]$, and if the change-point is present then its height is at least $\rho= \rho^*_{r}(\delta)- \zeta$ for $\zeta$ arbitrarily small. Since all  $n/(2r)-1$ evenly-spaced change-points $2r+1,4r+1,\ldots, n-2r+1$ are  known to the statistician, detecting all remaining change-points is equivalent to building an $n/(2r)$ multiple test of the hypotheses
$H_{0, l,r}$ versus $H_{\rho, l,r}$ for $l=r+1,3r+1,5r+1,\ldots$. If a change-point procedure achieves $({\bf NoSp})$ and  ${\bf detects}$ all change-points with radius at least $r/2$ and height at least $\rho$ with probability at least $1-\delta$, then one is able, with probability uniformly higher than $1-\delta$, to simultaneously perform without error $n/(2r)$ independent tests $H_{0, l,r}$ versus $H_{\rho, l,r}$. Since any single test must endure an error with probability at least $\delta$ in the worst case, no collection of independents tests  is able to endure less than $1-(1-\delta)^{n/(2r)}$. When $n/r$ is large and $\delta < 2r/n$, the latter is of the order of $\delta 2r/n$. Based on this, we conjecture that no change-point procedure  is able to achieve, with probability higher than $1-\delta$ the property $({\bf NoSp})$, and also to ${\bf detect}$ all change-points with radius at least $r/2$ and height at least $\rho^*_{r}(2r\delta/n)- \zeta$ for $\zeta>0$ arbitrarily small.

\medskip

Comparing the performances of our procedure with the negative arguments that we just outlined, we see that aggregating optimal tests on a dyadic grid allows to detect change-points with (almost) uniform height higher $\rho^*_{\tilde{r}_k}(r_k \delta\beta_{r_k}/(2n))$ whereas, as explained above, we conjecture that a change-point $\tau_k$ can be detected only if   $\bar \Delta_k\geq \rho^*_{r_k}(2r_k\delta/n)$.
Since $\tilde{r}_{k}\geq (r_k/8)\vee 1$- as we considered the dyadic grid when constructing $\tilde{r}_{k}$ - the difference between these two bounds is mostly due to the term $\beta_r$ which is of the order of $\log^2(n/r)$. Whereas it is possible to detect change-points at a given scale with a test of type I error probability $2r\delta/n$, our multi-scale procedure relies on a collection of single tests with type I error probability of the order of $r\delta/n/ \log^2(n/r)$. This mild mismatch - that we introduce to deal with the multiplicity of scales - of order $\log^{2}(n/r)$ is harmless for the Gaussian mean-detection problem. Indeed, one may deduce from our analysis in \Cref{sec:multi_scale_change_point_estimation_in_gaussian_noise} that
$\rho^*_{r_k}(2r_k\delta/n)$ is of the same order as $\rho^*_{\tilde{r}_k}(\delta|\mathcal{D}_{\tilde{r}_k}|^{-1}\beta_{\tilde{r}_k})$.
\medskip

In conclusion, one can build through Algorithm~\ref{algo_1} an almost optimal change-point procedure in any model provided that we are given optimal homogeneity tests of the form $H_{0,l,r}$ versus $H_{\rho,l,r}$. This provides a universal reduction of the problem of change-point detection to the problem of homogeneity testing.

\section{Numerical Experiments}\label{sec:numerical}

In this section, we illustrate the behavior of our procedure to detect change-points in a sparse high-dimensional setting~\eqref{eq:model_multivariate}. 

\noindent 
{\bf Performance Measure}. To assess the quality of change-point estimator $\hat{\tau}$, we first measure whether the estimated number of change-points $\widehat{K}=|\hat{\tau}|$ is equal to the true number $K$ of change-points. We also define the 
$\mathbf{SAND}$ loss as the proportion of $\bS$purious estimated change-points $\bA$nd true change-points that are $\bN$ot $\bD$etected:
$$\mathbf{SAND}((\tau_k), (\hat \tau_{k'})) = \tfrac{1}{K}\sum_{k=1}^K \left| |[(\tau_k + \tau_{k-1})/2,(\tau_k + \tau_{k+1})/2] \cap \{\hat \tau_k, k \in [\hat K]\}| - 1\right| \enspace .$$

\noindent
{\bf Change-point Detection Methods}. 
In the experiments, we implemented the bottom-up aggregation procedure \Cref{algo_1} with partial norm tests $T^{(\mathrm{p})}$ and dense test $T^{(\mathrm{d})}$ corresponding to \Cref{sec:multi_scale_change_point_estimation_in_sub_gaussian_noise} on a semi-complete grid $\cG_F = \{(l, r) ~:~ l \in \{r+1, \dots, n-r+1,~ r\in \cR\}$ - 
we take scales $r$ in the dyadic set for computational purposes. On a location $l$ and a scale $r$, each test statistic can be seen as a partial norm test relying on the statistic $\Ssup_{l,r,s}$ defined in \Cref{sub:tests_in_the_sparse_regime} and a threshold $\mathrm{Thresh}(r,s)$ which is either equal to $x^{(\mathrm{d})}_{r}$ when $s=d$ - see \Cref{sub:tests_in_the_dense_regime} - or to $x^{(\mathrm{p})}_{r,s}$ when $s \in \cZ_r := \{1, 2, 4, \dots, 2^{\floor{\log_2(s_{\max})}}\}$ with $s_{\max} := \frac{\sqrt{p\lo}}{\log(p)-\log(\lo)}$ - see \Cref{subsection:consequences} for the definition of the boundary between sparse and dense regimes $s_{\max}$. We actually do not use the definition of $x^{(\mathrm{d})}_{r}$ and $x^{(\mathrm{p})}_{r,s}$ for our thresholds $\mathrm{Thresh}(r,s)$ since they rely on constants that are not necessarily tight, but we rather calibrate them by a Monte-Carlo method using $10.000$ independant samples. For each sample consisting in a time series made of $n$ gaussian normal centered vector in $\mathbb{R}^p$, and for each $r \in \cR, s \in \cZ_r \cup \{p \}$, we compute the maximum over all $l$ of the statistics $\Ssup_{l,r,s}$. Considering the list of all the $10.000$ maximums and taking $\delta=5\%$, $\mathrm{Thresh}(r,s)$ is then defined as the $(1- \delta/(2|\cR||\cZ_r|))$-quantile if $s \in \cZ_r$ and as the $(1 - \delta/(2 |\cR|))$-quantile if $s = p$, so that, by a union bound, the total probability of finding a false positive is less than $\delta$. Note that  this calibration step only depends on $n$, $p$, and $\sigma$ and only needs to be performed once and for all. 

 We compare our procedure with the inspect method of~\cite{Wang2018} which is available as an $R$ package. The tuning parameters of inspect are computed with the automatic method defined in the same R package.

 \bigskip

In all the following experiments, we fix the dimension $p=100$ and the sample size $n=200$. We generate a piecewise constant signal $(\eta_t)_{t=1}^{n}$ in $\bbR^p$ with possible change-points $(\tau_1, \dots, \tau_K)$ using one of the three following settings. We then add a scaling factor $\alpha > 0$ and apply our procedure to the data $y_t = \alpha\eta_t + \varepsilon_t$, which amounts to setting $\theta_t = \alpha \eta_t$ in model \Cref{eq:model_multivariate}. We fix the variance of all the coordinates of $\varepsilon_t$ to be equal to one. Increasing $\alpha$ on a grid with step $0.1$ allows us to experimentally identify a transition between the regime where we do not detect precisely the change-points - in which case the two losses tend to be close to one - and the regime where we do detect the change-points - in which cases the losses are smaller.  We consider three simulation settings: 
\begin{enumerate}
\item {\bf Segment}. We generate a signal $\eta$ which is zero everywhere, except on $[80, 100]$ where we set it equal to a random vector $\Delta$ with $\|\Delta\| = 1$ and $\|\Delta\|_0 = s$, for $s = 1, 20, 100$. In each one of these cases, we choose the location of the $s$ non null coordinates of $\Delta$ uniformly at random and their value uniformly at random in the set $\{-1/\sqrt{s}, 1/\sqrt{s}\}$. Each time, $\eta$ has $2$ true change-points, and we generate the noise $(\epsilon_t)$ as independent centered and normalized gaussian vectors. 
\item {\bf Multiple Change points}. We generate $10$ uniform random locations $\tau_1< \tau_2< \ldots <\tau_{10}$ on $[1, 200]$. For each location $\tau_i$, we generate a uniform random integer $s_i \in [1, 100]$ and a vector $\Delta_i$ as in the segment setting with $\|\Delta_i \| = 1$ and $\|\Delta_i\|_0 = s_i$.  We generate a uniform random real number $N_i \in [1, 5]$ and define the time series $\eta_i$ by $(\eta_i)_t=N_i \Delta_i\1_{t\geq \tau_i}$.  Finally, the signal $\eta=\sum_{i=1}^{10}\eta_i$ has exactly $10$ change-points with random locations. As previously, the noise components $(\varepsilon_t)$ follow independent centered and standard gaussian vectors.
\item {\bf Time-dependencies}. We use the same signal as in the segment setting with $s=20$ but we move away from our assumptions by considering time dependencies. More precisely, the $(\varepsilon_t)$'s are now defined according to an AR process such taht $\varepsilon_{t+1} = \rho \varepsilon_t + \sqrt{1 - \rho^2}\varepsilon'_{t+1}$ for $t \geq 0$ where $(\varepsilon'_t)$ are independent centered and normalized gaussian vectors, $\rho = 0.05$ for the simulation and by convention $\varepsilon_0\sim \cN(0,I_p)$. 
\end{enumerate}

\paragraph*{Risk estimation with Monte-Carlo.}
In  each setting, we generate $500$ independent samples and compute the twpo losses $\mathbf{SAND}((\tau_k), (\hat \tau_{k'}))$ and $\1\{\hat K \neq K \}$. We estimate the risks  $\mathbb E[\mathbf{SAND}((\tau_k), (\hat \tau_{k'}))]$ and $\mathbb P ( K \neq \hat K)$ by averaging the loss over the $500$ trials. We also compute  95\% confidence intervals.

\paragraph*{Results.}
In the segment setting - see \Cref{fig:worst_case_simu_s1}, \ref{fig:worst_case_simu_s20}, \ref{fig:worst_case_simu_s100}, the risks tend to decrease as $\alpha$ increases since the higher $\alpha$, the higher the energy of the generated change-points are. As $s$ increases, we can see that both methods need a higher scaling factor to achieve the same risk, which translates the fact that the higher $s$, the more energy is needed to detect a change-point with vector $\Delta$ of sparsity $s$. In the segment settings, our bottom-up procedure tends to achieve significantly smaller loss than the inspect method on average. It is not the case in the multiple change-points setting - see \Cref{fig:simu_K10} - where the inspect method tends to perform slightly better.
In the setting with time-dependencies - see \Cref{fig:time_dependent} - the risks are worse than the corresponding setting without time-dependencies - see \Cref{fig:worst_case_simu_s20} - mainly because adding time-dependencies tends to create more spurious change-points (i.e. false positives).

\paragraph*{Computation time} Our code is implemented with python 3.9 and it mainly uses the convolution function conv1d from pytorch 1.12.1 to compute the Cusum statistics. Simulations are run on CPU (Intel(R) Core(TM) i7-10510U CPU @ 1.80GHz) with 32Go of memory. Running our method on pure noise - i.e. $\theta_t = 0$ for all $t$ - takes $101 \pm 2$ ms while the inspect method takes only $18 \pm 2$ ms to run on average, but optimizing our code is out of the scope of this paper. All the experiments are described in the  repository  \url{https://github.com/epilliat/multicpdetec}.
\begin{figure}[H]
	\begin{minipage}[c]{0.48\textwidth}
		\includegraphics[scale=0.32]{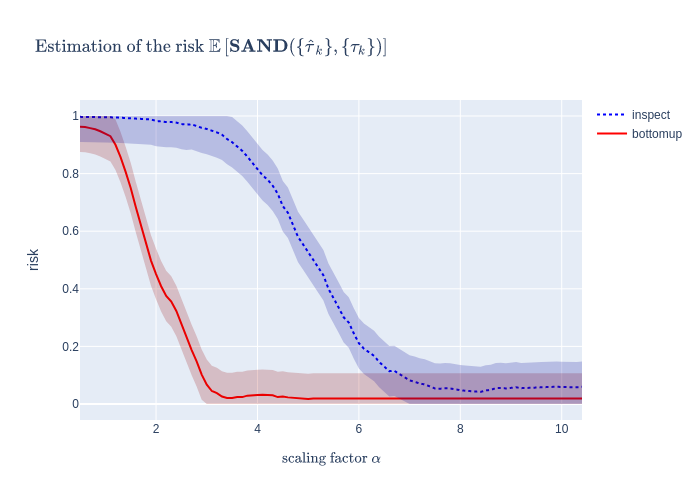}	
	\end{minipage}
	\ ~ \
	\begin{minipage}[c]{0.48\textwidth}
		\hspace{-0.7 cm}\includegraphics[scale=0.32]{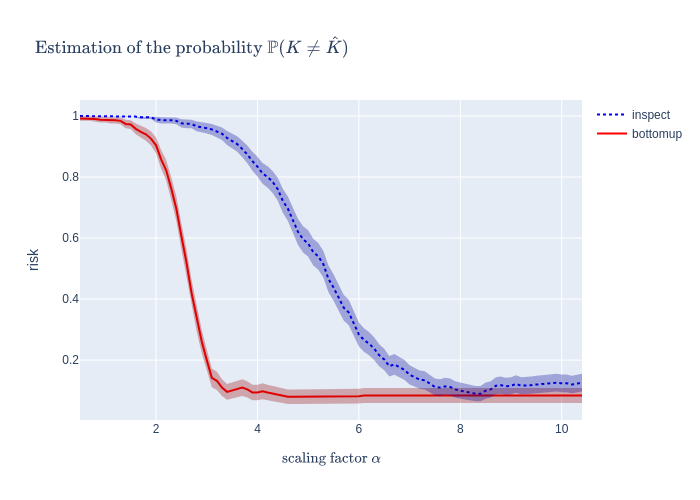}	
	\end{minipage}

		\caption{Estimation of $\mathbb{E}[\mathbf{SAND}((\tau_k), (\hat \tau_{k'}))]$ and $\mathbb{P}(\hat K \neq K)$ in the segment setting with $s=1$. }
		\label{fig:worst_case_simu_s1}
\end{figure}

\begin{figure}[H]
	\begin{minipage}[c]{0.48\textwidth}
		\includegraphics[scale=0.32]{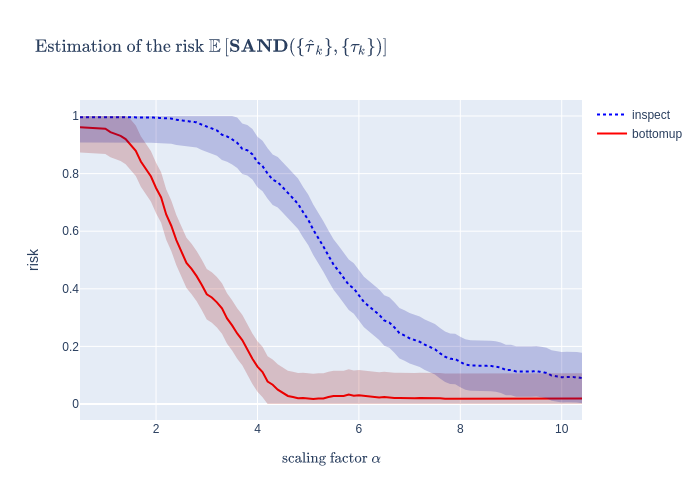}	
	\end{minipage}
	\ ~ \
	\begin{minipage}[c]{0.48\textwidth}
		\hspace{-0.7 cm}\includegraphics[scale=0.32]{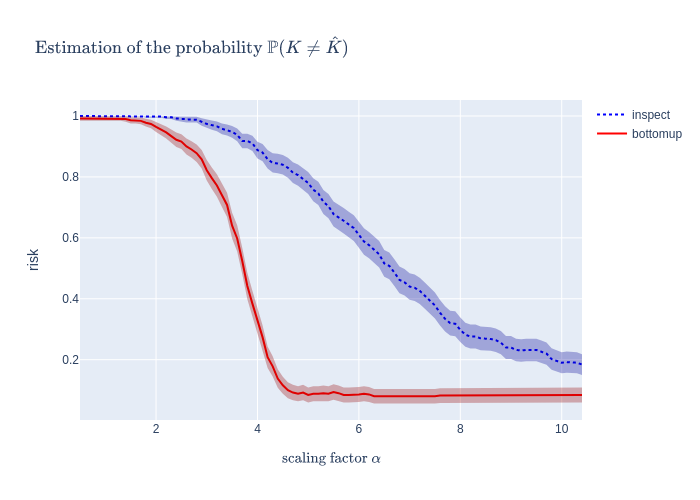}	
	\end{minipage}

		\caption{Estimation of $\mathbb{E}[\mathbf{SAND}((\tau_k), (\hat \tau_{k'}))]$ and $\mathbb{P}(\hat K \neq K)$ in the segment setting with $s=20$.}
		\label{fig:worst_case_simu_s20}
\end{figure}

\begin{figure}[H]
	\begin{minipage}[c]{0.48\textwidth}
		\includegraphics[scale=0.32]{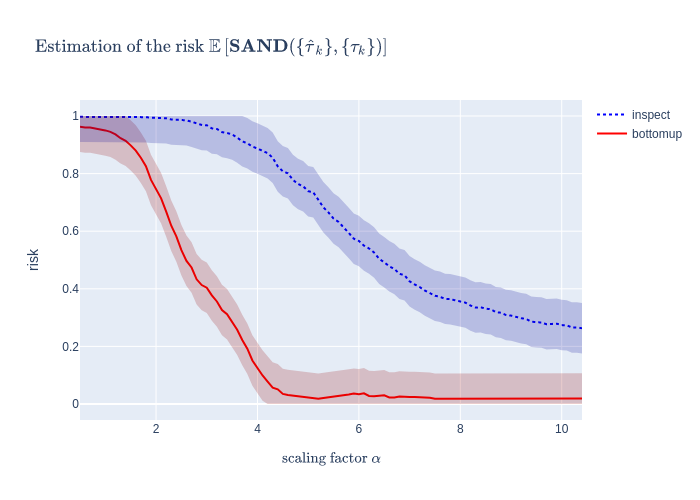}	
	\end{minipage}
	\ ~ \
	\begin{minipage}[c]{0.48\textwidth}
		\hspace{-0.7 cm}\includegraphics[scale=0.32]{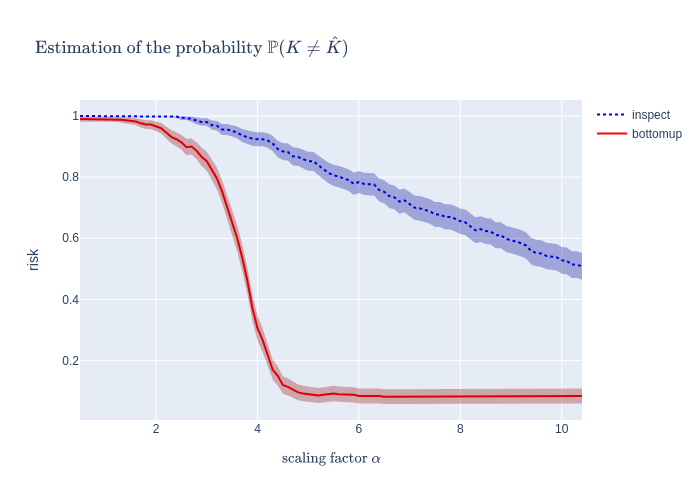}	
	\end{minipage}

		\caption{Estimation of $\mathbb{E}[\mathbf{SAND}((\tau_k), (\hat \tau_{k'}))]$ and $\mathbb{P}(\hat K \neq K)$ in the segment setting with $s=100$.}
		\label{fig:worst_case_simu_s100}
\end{figure}

\begin{figure}[H]
	\begin{minipage}[c]{0.48\textwidth}
		\includegraphics[scale=0.32]{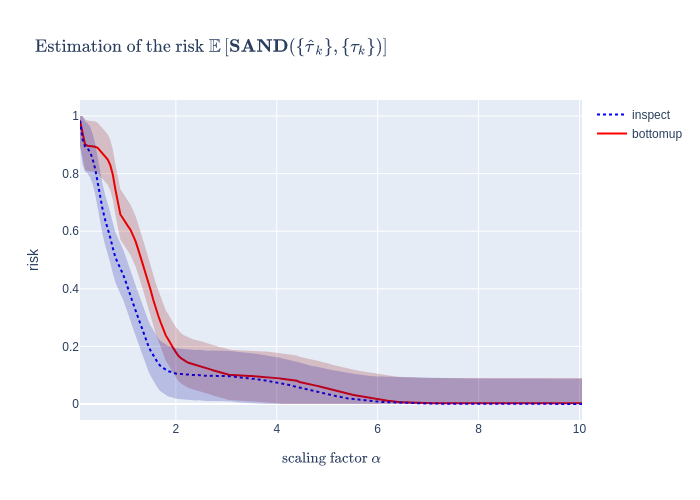}	
	\end{minipage}
	\ ~ \
	\begin{minipage}[c]{0.48\textwidth}
		\hspace{-0.7 cm}\includegraphics[scale=0.32]{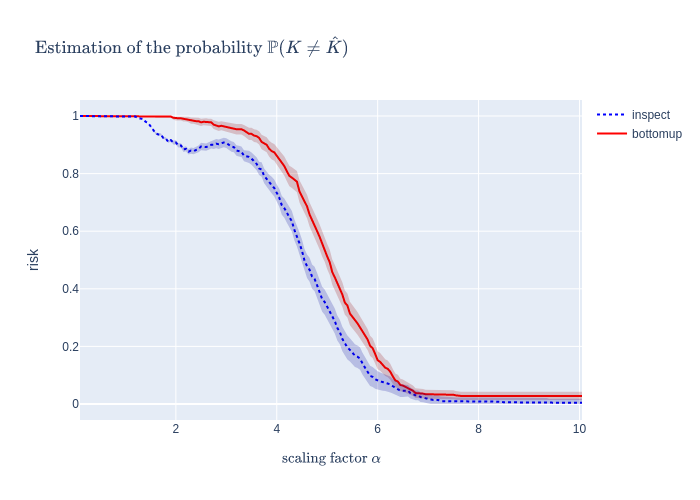}	
	\end{minipage}

		\caption{Estimation of $\mathbb{E}[\mathbf{SAND}((\tau_k), (\hat \tau_{k'}))]$ and $\mathbb{P}(\hat K \neq K)$ in a multiple change-point setting with $K = 10$ where change-points have random norms in $[1, 5]$ and random sparsities in $[1, p]$.}
		\label{fig:simu_K10}
\end{figure}

\begin{figure}[H]
	\begin{minipage}[c]{0.48\textwidth}
		\includegraphics[scale=0.32]{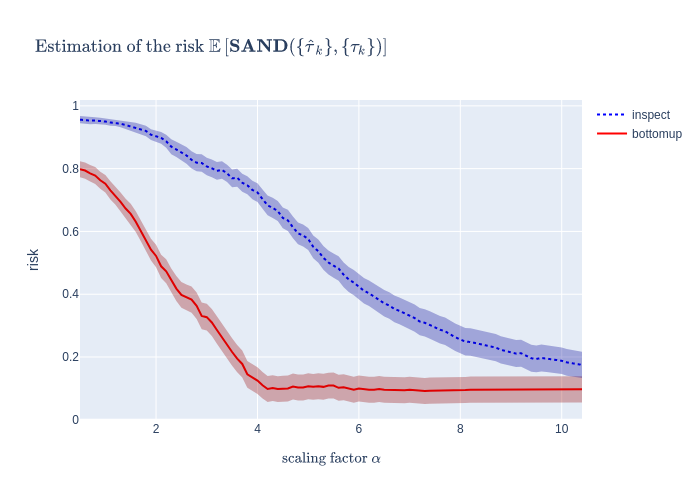}	
	\end{minipage}
	\ ~ \
	\begin{minipage}[c]{0.48\textwidth}
		\hspace{-0.7 cm}\includegraphics[scale=0.32]{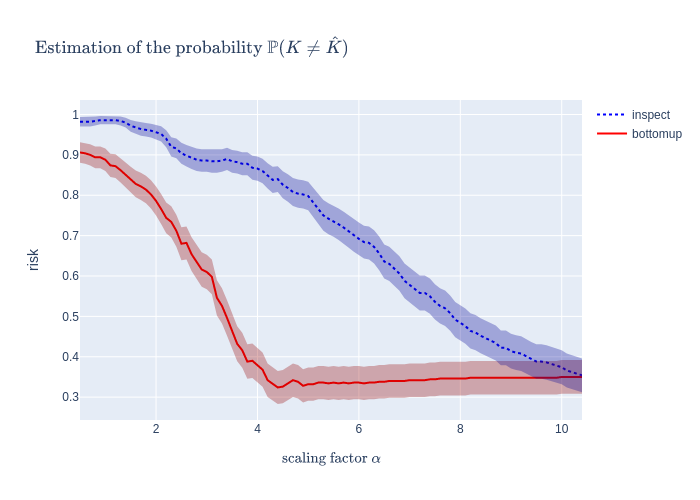}	
	\end{minipage}

		\caption{Estimation of $\mathbb{E}[\mathbf{SAND}((\tau_k), (\hat \tau_{k'}))]$ and $\mathbb{P}(\hat K \neq K)$ in the segment setting with $s = 20$ but with time-dependent noise that have an auto-correlation of $\rho = 5\%$.}
		\label{fig:time_dependent}
\end{figure}

\paragraph*{Acknowledgements.} The work of A. Carpentier is partially supported by the Deutsche Forschungsgemeinschaft (DFG) Emmy Noether grant MuSyAD (CA 1488/1-1), by the DFG - 314838170, GRK 2297 MathCoRe, by the FG DFG, by the DFG CRC 1294 'Data Assimilation', Project A03, by the Forschungsgruppe FOR 5381 "Mathematical Statistics in the Information Age - Statistical Efficiency and Computational Tractability", Project TP 02, by the Agence Nationale de la Recherche (ANR) and the DFG on the French-German PRCI ANR ASCAI CA 1488/4-1 "Aktive und Batch-Segmentierung, Clustering und Seriation: Grundlagen der KI" and by the UFA-DFH through the French-German Doktorandenkolleg CDFA 01-18 and by the SFI Sachsen-Anhalt for the project RE-BCI. The work of E. Pilliat and N. Verzelen has been partially supported by ANR-21-CE23-0035 (ASCAI). The authors are grateful to two anonymous referees for their helpful comments that improved the presentation of the manuscript.

\newpage
\appendix

\section{An alternative Algorithm}\label{sec:algo2}

In \Cref{algo_2} below, we also introduce a variant of the procedure, where instead of merging relevant interesting intervals at the same scale, we only keep one of them. More precisely, we choose the convention of discarding the interval $[l-r+1,l+r-1]$ if there exists $l'<l$ such that $T_{l',r}=1$ and $[l-r+1, l+r-1] \cap[l'-r+1, l'+r-1] \neq \emptyset$. Alternatively, we could have chosen to discard one of the intervals at random.
\begin{algorithm}
	\caption{\label{algo_2} Variant bottom-up aggregation procedure of multiscale tests}
	\KwData{$y_t, t = 1 \dots n$ and local test statistic $(\Tg_{l,r})_{(l,r) \in \cG}$}
	\KwResult{$(\hat \tau_k)_{k\leq \hat K}$}
	$\SCP  = \emptyset$\ $\cT=\emptyset$\;
	\For{$r \in \cR$}{
		\For{$l \in \cD_r$ s.t. $\Tg_{l,r} = 1$}{
		\If{$[l-r+1, l+r-1]\cap \SCP = \emptyset$}{
				$\SCP \gets \SCP \cup [l-r+1,l+r-1]$\;
        $\cT \gets \cT \cup \{l\}$\;
		 }
	}}
	\Return $\cT$
\end{algorithm}

\section{Proofs}
\label{sec:proofs}
\subsection{Proof of \Cref{th:g}}
\label{subsub:proof_of_th:g}
Let $\Theta \in \bbR^{n\times p}$, $\Tg$ be a local test statistic, $\cK^*$ be a set of indices of significant change-points and $(\tg_k,\rg_k)_{k \in \cK^*}$ be elements of the grid $\cG$ that satisfy \Cref{eq:hypothese_scale}. We assume that $\cA(\Theta, \Tg, \cK^*, (\tg_k,\rg_k)_{k \in \cK^*})$ holds, that is:
\begin{enumerate}
	\item {\bf (No False Positive)}\label{it:nofp}
  $\Tg_{l,r} = 0$ for all $(l,r)\in \cH_0 \cap \cG$, where $\cH_0$ is defined by \Cref{eq:Ho}
	\item  \label{it:sjdetect}{\bf (Significant change-point detection)} for every $k \in  \cK^*$, we have $\Tg_{\tg_k,\rg_k} = 1$.
\end{enumerate}

For every  $r \in \cR$ define
$$\begin{array}{ll}
	\CPd_r &= \{ l \in \CP_r ~:~ \exists k \in \cK^* \mbox{ s.t. } \tau_k \in [l-r+1, l+r-1]\}, \\
	\SCPd_r &= \underset{l \in \CPd_r}{\bigcup} [l-r+1,l+r-1].
\end{array}
$$
In other words, for all $r \in \cR$, $\CPd_r$ is the subset of $\CP_r$ for which each interval of detection $[l-r+1,l+r-1]$ contains a significant change-point.
The next proposition recursively analyzes the detection sets corresponding to significant change-points $(\SCPd_r)_{r\geq 1}$. The first inclusion means that significant change-points which can be detected with a local statistic with radius smaller than $r$ are detected  before step $r$, while the second inclusion means that each connected component of $\underset{r\in \cR}{\bigcup}\SCPd_r$ is included in a close neighborhoods of some significant change-point $\tau_k$, $k \in \cK^*$.

\begin{proposition}\label{prop:double-inclusion}
For all $r \in \cR\cup \{0\}$, we have the double inclusion
\begin{align}
	\left\{ \tau_k ~:~ k \in \cK^* \text{ and } \rg_k \leq r \right\}
	\subset \bigcup_{r'\leq r, r'\in \cR} \SCPd_{r'}
	\subset \bigcup_{k \in \cK^*} [\tau_k-2(\rg_k-1),\tau_k + 2(\rg_k-1)] \enspace.
\end{align}
\end{proposition}
The next proposition shows that for each step $r\in \cR$, the subset of detection corresponding to non significant change-point is disjoint from $\bigcup_{r'\in \cR}\SCPd_{r'}$.
\begin{proposition}\label{prop:control-other-change-points}
	For all $r \in \cR$, we have
		$$\bigcup_{l \in \CP_r \setminus \CPd_r} [l-r+1,l+r-1] \cap \left(\bigcup_{r'\in \cR}\SCPd_{r'}\right) = \emptyset \enspace .$$
\end{proposition}
Recall that $\left(\Cgc_k\right)_{k =1,\dots,\hat K}$ are defined as the connected component of $\bigcup_{r \in \cR} \SCP_r$. To ease the notation, re-index $(\Cgc_k)$ so that $\tau_k$ is the closest true change-point to $\hat \tau_k = \frac{\min \Cgc_k + \max\Cgc_k}{2}$. Since there is no false positive, $\tau_k \in \Cgc_k$.\\
By \Cref{prop:control-other-change-points}, the two closed subset $\bigcup_{r\in \cR}\bigcup_{l \in \CP_r \setminus \CPd_r} [l-r+1,l+r-1]$ and $\bigcup_{r\in \cR}\SCPd_r$ are disjoint. For all $k \in \cK^*$, it holds by \Cref{prop:double-inclusion} that $\tau_k \in \bigcup_{r \in \cR} \SCPd_{r}$, so that $\Cgc_k$ is a connected component of $\bigcup_{r\in \cR}\SCPd_r$ containing the significant change-point $\tau_k$. In particular, $\hat K \geq |\cK^* |$.
We have
\begin{itemize}
	\item By \Cref{prop:double-inclusion}, $\Cgc_k \subset [\tau_k - 2(\rg_k - 1), \tau_k + 2(\rg_k - 1)]$ for every $k \in \cK^*.$ Thus
	$$ |\hat\tau_{k} - \tau_k| \leq (\rg_k - 1) < \frac{r_k}{4}.$$
	\item For all $k \in [K] \setminus \cK^*,$ either $\tau_k$ does not belong to $\bigcup_{r \in \cR} \SCP_r$ and it is simply not detected, or it is the closest true change-point to $\hat \tau_k = \frac{\min \Cgc_k + \max\Cgc_k}{2}$ so that
	$$\hat \tau_{k} \in \left[ \tau_k - \frac{\tau_k + \tau_{k-1}}{2},\tau_k + \frac{\tau_k + \tau_{k+1}}{2} \right] \enspace.$$
	In particular,
	$$\{\hat \tau_{k'}, k' \leq \hat K\} \subset \left[ \tau_{1} - \frac{\tau_{1} - \tau_{0}}{2},\tau_{K} + \frac{\tau_{K+1} - \tau_{{K}}}{2} \right] \enspace.$$
	\item Finally, if there exists two estimated change-points $\hat \tau_{k_1}, \hat \tau_{k_2}$ in $ \left[ \tau_k - \frac{\tau_k + \tau_{k-1}}{2},\tau_k + \frac{\tau_k + \tau_{k+1}}{2} \right]$, then either $\Cgc_{k_1}$ or $\Cgc_{k_2}$ does not contain $\tau_k$. Then $\Theta$ is constant on $\Cgc_{k_1}$ or on $\Cgc_{k_2}$ and we obtain a contradiction since there is no false positive.
\end{itemize}

This concludes the proof of \Cref{th:g}.

\begin{proof}[Proof of \Cref{prop:double-inclusion}]
To prove the proposition, we do an induction on $r \in \cR\cup \{0\}$.
The case $r=0$ is trivial since by definition, $\SCP_0 = \emptyset$.
Let $r \in \cR$ and assume that the double inclusion \Cref{prop:double-inclusion} holds for all $r' < r, r' \in \cR\cup \{0\}$.

\paragraph*{First inclusion:} Let $k \in \cK^*$ be such that $\rg_k = r$ and assume that the corresponding significant change-point $\tau_k$ has not been detected before step $r$, that is $\tau_k \not\in \underset{r' < r}{\bigcup} \SCPd_{r'}$. Since $k \in \cK^*$, this implies in particular that $\tau_k \not\in \underset{r' < r}{\bigcup} \SCP_{r'}$. Let us show that $\tau_k \in \SCP_r$. To this end we prove that
\begin{align}\label{eq:point1g}
	[\tg_k - r+1, \tg_k + r -1] \cap \underset{r'<r,r' \in \cR}{\bigcup} \SCP_{r'}= \emptyset
\end{align} and

\begin{align}\label{eq:point2g}
\Tg_{\tg_k,r} = 1,
\end{align}
which will be enough since $\abs{\tg_k - \tau_k} \leq \rg_k - 1 = r - 1$.
\begin{itemize}
	\item {\bf Proof of \Cref{eq:point1g}:} Assume for the sake of contradiction that there exists an integer $z$ which belongs to $[\tg_k - r+1, \tg_k +r-1] \cap \underset{\substack{r'<r \\r' \in \cR}}{\bigcup} \SCP_{r'}$. There exists $r'< r$ such that $z \in \SCP_{r'}$ and $l(z) \in \CP_{r'}$ such that $z \in [l(z) - r'+1, l(z) + r'-1]$. Since $\tau_k \not \in \underset{r'<r}{\bigcup} \SCP_{r'}$, we have $\tau_k \not \in [l(z) - r'+1, l(z) + r'-1]$. Moreover,
	\begin{align*}
		|l(z) - \tau_k|
		& \leq |l(z) - z| + |z-\tg_k| + |\tg_k - \tau_k|\\
		& \leq (r'-1) + (r - 1) +  |\tg_k - \tau_k|\\
		& < r_k - r' \enspace,
	\end{align*}
	Where the last inequality comes from the hypothesis $3(\rg_k - 1) + \abs{\tg_k - \tau_k} \leq r_k$
	Consequently,
	\begin{align*}
		[l(z) - r', l(z) + r']
		\subset
		[\tau_k - r_k, \tau_k + r_k) \setminus \{\tau_k\} \enspace,
	\end{align*}
	so that $\theta$ is constant on $[l(z) - r', l(z) + r')\cap\bbN$. Thus, $(l(z),r') \in \cH_0$ and $l(z) \not \in \CP_{r'}$ since there is no false positive. This gives a contradiction and concludes the proof of \Cref{eq:point1g}.
	\item {\bf Proof of \Cref{eq:point2g}:} This is simply a consequence of the fact that significant change-point are detected on the grid (See Item \labelcref{it:sjdetect} in the definition of $\cA$).
\end{itemize}
We have just shown that $\tau_k \in \SCP_r$ and hence $\tau_k \in \SCPd_r$ so that the first inclusion holds at step $r$.
\paragraph*{Second inclusion} : Let $x$ be an element of $\SCPd_r$. There exists $l(x) \in \CPd_r$ such that\\
$x \in [l(x) - r+1, l(x) +r-1]$.
By definition of $\CPd_r$, there exists a significant change-point $\tau_k$ ( i.e. such that $k \in \cK^*$) belonging to $ [l(x) - r+1, l(x) + r-1]$.

We necessarily have $\rg_k \geq r$. Indeed, if  $\rg_k < r,$ then by the induction hypothesis, $\tau_k \in \SCPd_{r'}$ for some $r' < r$, which contradicts the fact that $\SCPd_{r'}$ is disjoint from $[l(x) - r+1, l(x) + r-1] \subset \SCPd_r$. Consequently,
\begin{align*}
	|l(x) - \tau_k| + r-1
	& \leq 2r - 2 \\
	& \leq 2 (\rg_k -1) \\
\end{align*}
Thus
\begin{align*}
	x \in [l(x) -r+1, l(x) + r-1]
	\subset
	[\tau_k - 2(\rg_k-1), \tau_k + 2(\rg_k-1)] \enspace.
\end{align*}
We have just shown that $\SCPd_r \subset \underset{k\in \cK^*}{\bigcup} [\tau_k - 2(\rg_k-1), \tau_k + 2(\rg_k-1)].$\\
Therefore, the proposition is verified at step $r$ and the induction is proved.
\end{proof}
\begin{proof}[Proof of \Cref{prop:control-other-change-points}]
	Let $k \in \cK^*$ and $\Cgc_k$ be the detected connected component containing the significant change-point $\tau_k$
	$$\Cgc_k = \underset{r' \in \cR}\bigcup \SCPd_{r'} \cap \left[ \tau_k - 2(\rg_k - 1), \tau_k + 2(\rg_k - 1)\right] \enspace .$$
	We know from \Cref{prop:double-inclusion} that $C_k$ is a connected component of $\bigcup_{r' \in \cR} \SCPd_{r'}$ and we want to prove now that $C_k$ does not overlap with $\bigcup_{l \in \CP_r \setminus \CPd_r} [l-r+1, l+r-1]$ for some $r \in \cR$. Let $r_0$ be such that $C_k$ is the connected component of $\SCP_{r_0}$,
	$$ C_k \subset \SCPd_{r_0} \enspace .$$
	Such an $r_0$ exists and is unique since the sets $\left(\SCPd_{r'}\right)$ are disjoint. We have from \Cref{prop:double-inclusion} that $\tau_k \in \bigcup_{r' \in \cR, r' \leq \rg_k} \SCPd_{r'}$ so that
	$$r_0 \leq \rg_k \enspace .$$
	Let $r \in \cR$ and $l \in \CP_r \setminus \CPd_r$ and assume without loss of generality that $l + r - 1 < \tau_k$. Since there is no false positive, $(l,r) \not \in \cH_0$ and there exists at least one true change-point in the interval of detection $[l-r+1, l+r-1]$. Denote $\tau_a, \dots, \tau_b$ with $a \leq b$ the true change-points belonging to $[l-r+1,l+r-1]$. By definition of $\CP_r \setminus \CPd_r$, $\tau_a, \dots, \tau_b$ are not significant change-points, i.e. $a, a+1, \dots,b \not \in \cK^*$. We consider the two cases $r > \rg_{k}$ and $r \leq \rg_{k}$
	\begin{itemize}
	\item \underline{$r > \rg_{k}$} : In that case, since the sets $(\SCP_{r'})$ are disjoint and $\Cgc_{k} \subset \SCPd_{r_0}$, we have $\Cgc_{k}\cap [l-r+1, l+r-1] = \emptyset$.
	\item \underline{$r \leq \rg_{k}$} : In that case, we have
	$$ l+r-1 \leq \tau_b + 2(r - 1) \leq \tau_b + 2(\rg_k- 1)< \tau_{k} - 2(\rg_{k} - 1) \enspace,$$
	where we used the fact that $4(\rg_{k} - 1) < r_{k} \leq \tau_{k} - \tau_b$. Since by \Cref{prop:double-inclusion} we have $\Cgc_{k} \subset [l-r+1, l+r-1]$, we also have in that case $\Cgc_{k}\cap [l-r+1, l+r-1] = \emptyset$.
\end{itemize}
This concludes the proof of the proposition.
\end{proof}


\subsection{Proofs for Gaussian multivariate change-point detection}
\label{sub:proofs_in_the_gaussian_setting}
From now on, we use the following notation for all $(l,r) \in J_n$.
\begin{itemize}
	\item For any $(v_1,\dots,v_n)$ with $v_t \in \bbR^p$, the left mean and right mean of $v$ on $[l-r,l+r)$ are denoted by
	$$\bar v_{l,+r} = \frac{1}{r}\sum_{t = l}^{l+r-1} v_t \quad \bar v_{l,-r} = \frac{1}{r}\sum_{t = l-r}^{l-1} v_t \enspace .$$
	\item The population term of the CUSUM statistic $\bC_{l,r}$ is written
	$$U_{l,r} = \sqrt{\frac{r}{2}}\left(\bar \theta_{l,+r} - \bar \theta_{l,-r}\right)\enspace.$$
	\item With these notation, we write $v_{l,+r,i},v_{l,-r,i}, U_{l,r,i}$ for the $i^{th}$ coordinate of the vector $v_{l,+r}, v_{l,-r}, U_{l,r}$.
	\item We define, for $1\leq s\leq p$, the order statistics  $U_{l,r,(s)}$ by $|U_{l,r,(1)}|\geq |U_{l,r,(2)}|\geq \ldots |U_{l,r,(p)}|$.
\end{itemize}
\subsubsection{Proof of \Cref{prop:dense}}
\label{subsub:proof_of_prop:dense}

\paragraph*{Step 0: Consequence of Equation~\eqref{eq:EnergyDenseGauss} on the grid.}
Let $k \in [K]$ and assume that $\tau_k$ is a $\kappa_{\mathrm{d}}$-dense high-energy change-point (see Equation~\eqref{eq:EnergyDenseGauss}). We have that
\begin{equation}
\label{eq:EnergyDenseGaussGrid}
\begin{aligned}
	\norm{ U_{\tgD_k,\rgD_k}}^2
	&\geq \frac{9}{16}\norm{ U_{\tau_k,\rgD_k}}^2 \\
	&\geq \frac{9}{16\times 12}\kappa_{\mathrm{d}}\NoiseBound<p,n,\rgD_k,\delta>,
\end{aligned}
\end{equation}
since by definition $\|\tau_k - \tgD_k\| \leq \rgD_k/4$, so that $||\overline\theta_{\tgD_k,+\rgD_k} - \overline\theta_{\tgD_k,-\rgD_k}||^2 \geq \frac{9}{16}||\overline\theta_{\tau_k,+\rgD_k} - \overline\theta_{\tau_k,-\rgD_k}||^2.$

\paragraph*{Step 1: Introduction of useful high probability events.} Remark that 
\begin{align*}
	\frac{r}{2}
	\left[\norm{\overline{y}_{l,+r} - \overline{y}_{l,-r} }^2 - \norm{ \overline\theta_{l,-r} - \overline\theta_{l,+r}}^2
	\right] - \sigma^2 p
	=
	r \proscal<\overline{\Noise}_{l,+r} - \overline{\Noise}_{l,-r},\overline{\theta}_{l,+r} - \overline{\theta}_{l,-r}>
	+
	\frac{r}{2} \norm{\overline{\Noise}_{l,+r} - \overline{\Noise}_{l,-r} }^2 - \sigma^2 p \enspace.
\end{align*}

The first term, written as
\begin{align*}
	r \proscal<\overline{\Noise}_{l,+r} - \overline{\Noise}_{l,-r},\overline{\theta}_{l,+r} - \overline{\theta}_{l,-r}> \enspace,
\end{align*}
is a crossed term between the noise and the mean vector $\theta$. \Cref{lem:crossedtermsgrid} states that near the change-points and on the grid defined by the sets $\cR, \cD_r$, it is jointly controlled with high probability.
\begin{lemma}\label{lem:crossedtermsgrid}
Let $1\geq \delta>0$. The event
\begin{align*}
\begin{aligned}
	\xDense_1
	& = \bigcap_{k \in [K]}
	\Bigg\{ \rgD_k\abs{ \proscal<\overline{\Noise}_{\tgD_k,+\rgD_k} - \overline{\Noise}_{\tgD_k,-\rgD_k},\overline{\theta}_{\tgD_k,+\rgD_k} - \overline{\theta}_{\tgD_k,-\rgD_k}>} \\
	& \leq \frac{1}{8} \rgD_k\norm{ \overline\theta_{\tgD_k,+\rgD_k} - \overline\theta_{\tgD_k,-\rgD_k}}^2
	+ 16\sigma^2 \log\left(2\frac{n}{\rgD_k \delta}\right) \Bigg\} \enspace.
\end{aligned}
\end{align*}
holds with probability larger than $1 - \delta$.
\end{lemma}

The second term, written as
\begin{align*}
	\frac{r}{2}
	\norm{\overline{\Noise}_{l,+r} - \overline{\Noise}_{l,-r} }^2
	- \sigma^2 p \enspace,
\end{align*}
is a term of pure noise. \Cref{lem:purenoisegrid} states that it is controlled jointly with high probability on the grid defined by the sets $\cR, \cD_r$.
\begin{lemma}\label{lem:purenoisegrid}
Let $1\geq \delta>0$. The event
\begin{align*}
\begin{aligned}
	\xDense_2
	=
	\bigcap_{r \in \cR} \bigcap_{l \in \cD_r}
	\left\{
		\abs{\frac{r}{2} \norm{\overline{\Noise}_{l,+r} - \overline{\Noise}_{l,-r} }^2 - \sigma^2 p} \leq 4 \sigma^2 \left[ \sqrt{p \Log{2\frac{n}{r\delta} }} + \Log{ 2\frac{n}{r\delta}}  \right]
	\right\} \enspace,
\end{aligned}
\end{align*}
holds with probability larger than $1 - \delta$.
\end{lemma}

Set now
\begin{align*}
	\xDense := \xDense
		  = \xDense_1 \cap \xDense_2 \enspace.
\end{align*}
Note that
\begin{align*}
	\bbP(\xDense) \geq 1 - 2\delta \enspace.
\end{align*}

\paragraph*{Step 2: Study in the `no change-point' situation.}

Consider $r\in \cR, l\in \cD_r$ such that $\{\tau_k, k \in [K]\} \cap [l-r, l+r) = \emptyset$. Note that since $\{\tau_k, k \in [K]\} \cap [l-r, l+r) = \emptyset$, we have $\overline\theta_{l,-r} = \overline\theta_{l,+r}$ so that
\begin{align*}
	\frac{r}{2}\norm{\overline\theta_{l,-r} - \overline\theta_{l,+r}}^2
	=
	0 \enspace,
\end{align*}
and
\begin{align*}
	r \proscal<\overline{\Noise}_{l,+r} - \overline{\Noise}_{l,-r},\overline{\theta}_{l,+r} - \overline{\theta}_{l,-r}>  = 0.
\end{align*}
Moreover we have on $\xDense$ that - see \Cref{lem:purenoisegrid}
\begin{align*}
	\abs{\frac{r}{2}
	\norm{\overline{\Noise}_{l,+r} - \overline{\Noise}_{l,-r} }^2 - \sigma^2 p} \leq 4 \sigma^2
	\left[
		\sqrt{p \Log{2\frac{n}{r\delta} }} + \Log{2 \frac{n}{r\delta} }
	\right]
	= \sigma^2\tD_r \enspace.
\end{align*}
And so
\begin{align*}
	 \psiD_{l,r}\leq  \tD_r \enspace,
\end{align*}
so that
\begin{align*}
	\TD_{l,r} = 0\enspace,
\end{align*}
on $\xDense$.
This concludes the proof of the first part of the proposition.

\paragraph*{Step 3: Study in the `change-point' situation.}

Consider $k\in [K]$ $\tau_k$ is a $\kappa_{\mathrm{d}}$-dense high-energy change-point - that is Equation\eqref{eq:EnergyDenseGauss} holds.
We have from \Cref{eq:EnergyDenseGaussGrid} that for $\kappa_{\mathrm{d}}$ large enough,
\begin{align*}
	\frac{\rgD_k}{2}
	\norm{\overline\theta_{\tgD_k,-\rgD_k} - \overline\theta_{\tgD_k,+\rgD_k}}^2  &\geq \frac{9}{16\times 12}\kappa_{\mathrm{d}} \sigma^2\NoiseBound<p,n,\rgD_k,\delta>\\
	&> 4 \sigma^2 \tD_{\rgD_k} \enspace.
\end{align*}
So on $\xDense$ this implies that - see \Cref{lem:crossedtermsgrid}
\begin{align*}
	\rgD_k \abs{\proscal<\overline{\Noise}_{\tgD_k,+\rgD_k} - \overline{\Noise}_{\tgD_k,-\rgD_k},\overline{\theta}_{\tgD_k,+\rgD_k} - \overline{\theta}_{\tgD_k,-\rgD_k}>}
	\leq
	\frac{ \rgD_k}{4} \norm{ \overline\theta_{\tgD_k,+\rgD_k} - \overline\theta_{\tgD_k,-\rgD_k}}^2.
\end{align*}
Moreover we have on $\xDense$ that - see \Cref{lem:purenoisegrid}
\begin{align*}
	\abs{\frac{\rgD}{2}
	\norm{\overline{\Noise}_{\tgD_k,+\rgD_k} - \overline{\Noise}_{\tgD_k,-\rgD_k} }^2 - \sigma^2 p}
	\leq
	4 \sigma^2
	\left[
		\sqrt{p \Log{2\frac{n}{\rgD_k} \delta^{-1}}}
		+ \Log{ 2\frac{n}{\rgD_k} \delta^{-1}}
	\right]
	= \sigma^2\tD_{\rgD_k}.
\end{align*}
And so on $\xDense$, combining the three previous displayed equations implies
\begin{align*}
	\psiD_{\tgD_k,\rgD_k}
	\geq
	\frac{\frac{\rgD_k}{2} \norm{\overline\theta_{\tgD_k,+\rgD_k} - \overline\theta_{\tgD_k,-\rgD_k}}^2}{2\sigma^2} -  \tD_{\rgD_k}> (2 - 1)\tD_{\rgD_k} = \tD_{\rgD_k} \enspace,
\end{align*}
so that
\begin{align*}
	\TD_{\tgD_k,\rgD_k} = 1 \enspace.
\end{align*}
This concludes the proof of the second part of the proposition.

\begin{proof}[Proof of \Cref{lem:crossedtermsgrid}]

Let $k \in [K]$. Since the vectors $\Noise_t$ are i.i.d.~and distributed as $\cN(0, \sigma^2 \bI_p)$, it holds that
\begin{align*}
	\rgD_k \proscal<\overline{\Noise}_{\tgD_k,+\rgD_k} - \overline{\Noise}_{\tgD_k,-\rgD_k},\overline{\theta}_{\tgD_k,+\rgD_k} - \overline{\theta}_{\tgD_k,-\rgD_k}> \sim \cN\left(0, 2\rgD_k\sigma^2 \norm{ \overline\theta_{\tgD_k,+\rgD_k} - \overline\theta_{\tgD_k,-\rgD_k}}^2\right)
	\enspace.
\end{align*}
And so for $\delta_k>0$, it holds with probability larger than $1 - \delta_k$ it holds that
\begin{align*}
\begin{aligned}
	\rgD_k \abs{\proscal<\overline{\Noise}_{\tgD_k,+\rgD_k} - \overline{\Noise}_{\tgD_k,-\rgD_k},\overline{\theta}_{\tgD_k,+\rgD_k} - \overline{\theta}_{\tgD_k,-\rgD_k}>}\\
	\leq
	2\sigma \norm{ \overline\theta_{\tgD_k,+\rgD_k} - \overline\theta_{\tgD_k,-\rgD_k}}\sqrt{\rgD_k \log(2\delta_k^{-1})} \enspace.
\end{aligned}
\end{align*}
Let us set $\delta_k = \frac{(\rgD_k)^2\delta}{2n^2}$. Note that
\begin{align*}
	\sum_{k \in [K]} \delta_k
	=
	\sum_{r \in R} \sum_{k\in [K] : \rgD_k = r}
	\frac{(\rgD_k)^2\delta}{2n^2} \leq \sum_{r \in R} \sum_{l \in D_r}
	\frac{r^2\delta}{2n^2}
	\leq
	\sum_{r \in \cR}  \frac{r\delta}{2n} \leq \delta \enspace,
\end{align*}
since $r_k \geq \rgD_k$ and $|\cD_r| \leq 2n/r$, and also by definition of $\cR$ which implies $\sum_{r \in \cR} \frac{r}{n} \leq 1$. And so if $\delta \leq 1$, then with probability larger than $1-\delta$, for any $k \in [K]$, we have
\begin{align*}
	\rgD_k \abs{\proscal<\overline{\Noise}_{\tgD_k,+\rgD_k} - \overline{\Noise}_{\tgD_k,-\rgD_k},\overline{\theta}_{\tgD_k,+\rgD_k} - \overline{\theta}_{\tgD_k,-\rgD_k}>}
	\leq
	& 2\sigma \norm{ \overline\theta_{\tgD_k,+\rgD_k} - \overline\theta_{\tgD_k,-\rgD_k}}\\
	& \sqrt{2\rgD_k \log\left(2\frac{n}{\rgD_k }\delta^{-1}\right)}\enspace.
\end{align*}
This implies in particular that with probability larger than $1-\delta$, for any $k \in [K]$, we have
\begin{align*}
	\rgD_k \abs{\proscal<\overline{\Noise}_{\tgD_k,+\rgD_k} - \overline{\Noise}_{\tgD_k,-\rgD_k},\overline{\theta}_{\tgD_k,+\rgD_k} - \overline{\theta}_{\tgD_k,-\rgD_k}>}
	\leq
	& \frac{ \rgD_k}{2} \frac{\norm{ \overline\theta_{\tgD_k,+\rgD_k} - \overline\theta_{\tgD_k,-\rgD_k}}^2}{4}\\
	& + 16 \sigma^2\log\left(2\frac{n}{\rgD_k}\delta^{-1}\right) \enspace.
\end{align*}
\end{proof}

\begin{proof}[Proof of \Cref{lem:purenoisegrid}]

Let $r \in \cR$ and $l \in \cD_r$. Since the vectors $\Noise_t$ are i.i.d.~and distributed as $\cN(0, \sigma^2\bI_p)$, it holds that
\begin{align*}
	\frac{r}{2} \norm{\overline{\Noise}_{l,+r} - \overline{\Noise}_{l,-r} }^2 \sim \sigma^2\chi^2_p,
\end{align*}
which implies by properties of the $\chi^2_p$ distribution - see e.g. Lemma 1 of \cite{Laurent00} - that for any $\delta_r>0$ we have with probability larger than $1-\delta_r$
\begin{align*}
	\abs{\frac{r}{2} \norm{\overline{\Noise}_{l,+r} - \overline{\Noise}_{l,-r} }^2 - \sigma^2 p}
	\leq
	2\sigma^2\sqrt{p\log(2/\delta_r)} + 2\sigma^2 \log(2/\delta_r) \enspace.
\end{align*}
If we set, for $\delta>0$, $\delta_r = \frac{r^2\delta}{2n^2}$, we have that with probability larger than $1 - \frac{r\delta}{n}$, that $\forall l \in D_r$
\begin{align*}
	\abs{
		\frac{r}{2}
		\norm{\overline{\Noise}_{l,+r} - \overline{\Noise}_{l,-r} }^2
		- \sigma^2 p
		}
	\leq
	2\sigma^2\sqrt{p\log(2/\delta_r)} + 2\sigma^2 \log(2/\delta_r) \enspace,
\end{align*}
since $|\cD_r| \leq 2n/r$. And so with probability larger than $1 - \delta$, for all $r \in \cR$ and $l \in \cD_r$
\begin{align*}
	\abs{
		\frac{r}{2}
		\norm{\overline{\Noise}_{l,+r} - \overline{\Noise}_{l,-r} }^2
		- \sigma^2 p
		}
	\leq
	2 \sigma^2 \sqrt{p\log(2/\delta_r)} + 2\sigma^2 \log(2/\delta_r) \enspace,
\end{align*}
since $\sum_{r \in \cR} \frac{r}{n} \leq 1$. And so finally for $\delta \leq 1$ and with probability larger than $1-\delta$, for all $r \in \cR$ and $l \in \cD_r$
\begin{align*}
	\abs{
		\frac{r}{2}
		\norm{\overline{\Noise}_{l,+r} - \overline{\Noise}_{l,-r} }^2
		- \sigma^2 p
		}
	\leq
	4\sigma^2
	\left[
		\sqrt{p\Log{2\frac{n}{r}\delta^{-1}}} + \Log{2\frac{n}{r} \delta^{-1}}
	\right] \enspace.
\end{align*}
This concludes the proof.
\end{proof}
\subsubsection{Proof of  \Cref{prop:Sparse}}
\paragraph*{Step 1 : Analysis of the Berk-Jones statistics}
We first define a threshold $\tbj_{r,s}$ for the Berk-Jones statistics for all $r,s \geq 1$
\beq\label{eq:definition_t_s}
		\tbj_{r,s}=\min\left\{x\geq 2: \overline{\Phi}(x)\leq \frac{s^2}{28^2p\log(2\alphabj^{-1}_{x,r})}\right\}\enspace ,
		\eeq
where we recall that $\alphabj_{x,r}$ are the weights defined by \Cref{eq:def_alpha_tr}:
\begin{equation*}
\alphabj_{x,r} = \frac{6\delta r}{\pi^2x^2|\cD_r|n }\enspace .
\end{equation*}
Remark that $(\tbj_{r,s})$ is nonincreasing with $s$ and define for all $r \geq 1$
\beq\label{eq:condition_s_t_s}
	\sm_r= \min\left\{s \in \cZ ~:\quad s \geq \frac{28}{3}\Log{2\alphabj^{-1}_{\tbj_{r,s},r}}\right\}\enspace .
		\eeq
The second point of the following proposition ensures that if there exists $s \in \cZ$ such that $ U_{l,r,(s)} \geq t_s$ for some $s \geq \sm_r$, for $(l,r) = (\tgSp_k,\rgSp_k)$,  then $\Tbj_{l,r} = 1$ with high probability. We recall that $|U_{l,r,(1)}|\geq \dots \geq |U_{l,r,(p)}|$ are the sorted absolute values of the coordinate of $U_{l,r}$ and that $\cH_0$ is defined by \Cref{eq:Ho}.
\begin{proposition}\label{prop:BJ}
	There exists an event $\xBJ$ of probability larger than $1-2\delta$ such that the following holds:
	\begin{itemize}
	\item  $\Tbj_{l,r} = 0$ for any $(l,r) \in \cH_0 \cap G$.
	\item  For all $k \in [K]$, if there exists $s\in \cZ$ such that $s \geq \sm_{\rgSp_k}$ and $U_{\tgSp_k,\rgSp_k,(s)} > \tbj_{\rgSp_k,s}$, then $\Tbj_{\tgSp_k,\rgSp_k} = 1$.
	\end{itemize}
	\end{proposition}

\paragraph*{Step 2 : Analysis of the partial norm statistics}
Since it may happen that $\tau_k$ is a sparse high-energy change-point but there is no $s \geq \sm_{\rgSp_k}$ such that $U_{\tgSp_k,\rgSp_k,(s)} \geq \tbj_{\rgSp_k,s}$, we use the following proposition on the partial norm test statistic $\Tsup_{l,r}$:
	\begin{proposition}\label{prop:Sup}
		There exists an event $\xPN$ of probability larger than $1-2\delta$ such that the following holds:
		\begin{itemize}
			\item  $\Tsup_{l,r} = 0$ for any $(l,r) \in \cH_0 \cap G$.
			\item for any $k\in [K]$, if there exists $s \in \cZ$ such that
				\begin{equation}\label{EnergySparseGaussMax}
					\sum_{s' = 1}^s\left|U_{\tgSp_k,\rgSp_k,(s')}\right|^2 > 4\tsup_{\rgSp_k,s} \enspace ,
				\end{equation}
			\end{itemize}
			then $\Tsup_{\tgSp_k,\rgSp_k} = 1$.
		\end{proposition}
\paragraph*{Step 3 : Combination of the two Statistics}
		Let us return to the proof of \Cref{prop:Sparse}. To conclude the proof, it suffices to show that if $\tau_k$ is a $\kappa_{\mathrm{s}}$-sparse high-energy change-point - see \Cref{eq:EnergySparseGauss} - for some large enough constant $\kappa_{\mathrm{s}}$, then the result of one of the two preceding propositions holds. This is precisely what the following lemma shows.
\begin{lemma}\label{lem:LocalisationEnergy}
	There exists a constant $\kappa_{\mathrm{s}}$ such that if $\tau_k$ is a $\kappa_{\mathrm{s}}$-sparse high-energy change-point, then one of the following propositions is true:
	\begin{itemize}
		\item There exists $s \in \cZ $ such that $s > \sm_{\rgSp_k}$ and $\left|U_{\tgSp_k,\rgSp_k,(s)}\right| > \tbj_{\rgSp_k,s}$.
		\item There exists $s \in \cZ$ such that $s \leq \sm_{\rgSp_k}$ and $\sum_{s' = 1}^s\left|U_{\tgSp_k,\rgSp_k,(s')}\right|^2 > 4\tsup_{\rgSp_k,s}.$
	\end{itemize}
\end{lemma}

	\begin{proof}[Proof of \Cref{prop:BJ}]
		The first part of the proposition is a simple consequence of the definition together with an union bound.
		\begin{eqnarray*}
			\P\left[\max_{(l,r) \in \cH_0}\Tbj_{l,r}= 1\right]&\leq& \sum_{r\in \cR}\sum_{ l\in \cD_r}\sum_{x\in \bbN^*} \alphabj^{(\mathrm{BJ})}_{x,r}\\
			&\leq& \sum_{r\in \cR}\sum_{ l\in \cD_r} \frac{\delta r}{|\cD_r|n}\leq \sum_{r\in \cR}\frac{\delta r}{n}\leq \delta.
			\end{eqnarray*}
		We focus on the second part of the proposition. To ease the reading, we introduce some notation
		\beqn
		\gamma_{x,r}= \overline{Q}^{-1}[\alphabj_{x,r},p,2\overline{\Phi}(x)]\ ; \quad \eta_{x,r,s}=\overline{Q}^{-1}[1- \alphabj_{x,r}/2,p-s,2\overline{\Phi}(x)]\enspace ; \\
		\psi_{x,r,s}(u)= \overline{Q}^{-1}[1- \alphabj_{x,r}/2,s,\overline{\Phi}(x-u)+ \overline{\Phi}(x+u)]\enspace ,
		\eeqn
		for $x\geq 0$.
		In fact, $\gamma_{x,r}$ is the threshold of the statistics $N_{x,l,r}$. As for $\eta_{x,r,s}$, it stands for the contribution to $N_{x,l,r}$ of the $(p-s)$ coordinates $i$ such that $\theta_{\cdot,i}$ is constant over $[l-r,l+r)$. Finally, $\psi_{x,r,s}(u)$ stands for the contribution to $N_{x,l,r}$ of the $s$ coordinates $i$ whose population CUSUM statistics $U_{l,r,i}$ is equal to $u$.

		\begin{lemma}\label{lem:comparaison_quantile}
		Consider any $r\in \cR$ and $l\in \cD_r$. If for some positive integers $s$ and $x$ we have
		\beq\label{eq:condition_quantile}
		\psi_{x,r,s}(|U_{l,r,(s)}|)> \gamma_{x,r}- \eta_{x,r,s}\ ,
		\eeq
			then $\P[\Tbj_{l,r}=1]\geq 1- \alphabj_{x,r}$.
		\end{lemma}

		Denote $\cH[\theta]$ the collection of $(l,r)$ with $r\in \cR$ and $l \in \cD_r$ that satisfy Condition~\eqref{eq:condition_quantile} for some $s$ and some $x$. We easily deduce from the above Lemma together with an union bound that, with probability higher than $1-\delta$, $\Tbj_{l,r}=1$ for all $(l,r)\in \cH[\theta]$.

		\medskip

		Let us now provide a more explicit characterisation of $\cH[\theta]$ with the following Lemma.

		\begin{lemma}\label{lem:t_s}
		For any $1\leq s\leq p$ and $r\in \cR$ define $x_s$ by
    \[
		x_{s}:= \tbj_{r,s}=\min\left\{x\geq 2: \overline{\Phi}(x)\leq \frac{s^2}{28^2p\log(2\alpha^{-1}_{x,r})}\right\}\enspace .\]
		We have $\psi_{x_s,r,s}(t_s)> \gamma_{x_s,r}- \eta_{x_s,r,s}$ provided that
		\beq\label{eq:inequality_s_t_s}
		s\geq \frac{28}{3}\log(2\alphabj^{-1}_{x_s,r})\enspace .
		\eeq
		\end{lemma}
	Combining \Cref{lem:t_s} and \Cref{lem:comparaison_quantile}, we conclude the proof of the proposition.
	\end{proof}
	\label{subsub:proof_of_propsparse}
\begin{proof}[Proof of \Cref{lem:comparaison_quantile}]
	Denote $S$ any subset of size $s$, such that for any $j\in S$, $|U_{l,r,j}|\geq |U_{l,r,(s)}|$. Define
	\beqn
	N^{(1)}_{x,l,r}= \sum_{i=1}^p \1_{i\notin S}\1_{|\bC_{l,r,i}|> x },\quad\quad N^{(2)}_{x,l,r}= \sum_{i=1}^p \1_{i\in S}\1_{|\bC_{l,r,i}|> x }
	\eeqn
	Since, for any $x>0$,  the function
	$u\mapsto \overline{\Phi}(x+u)+\overline{\Phi}(x-u)$ is non-decreasing. As a consequence, the random variable $N^{(1)}_{x,l,r}$ is stochastically dominated by a Binomial distribution with parameters $(p-s,2\overline{\Phi}(x))$. Besides,
	$N^{(2)}_{x,l,r}$ is stochastically dominated by a Binomial distribution with parameters $(s,\overline{\Phi}(x+|U_{l,r,(s)}|)+\overline{\Phi}(x-|U_{l,r,(s)}|))$. We obtain
	\beqn
	\P[T^{(BJ)}_{l,r}=0]&\leq& \P[N_{x,l,r}\leq  \gamma_{x,r}]\leq
	\P[N^{(1)}_{x,l,r} < \eta_{x,r,s} ] +\P[ N^{(2)}_{x,l,r}\leq \gamma_{x,r}-  \eta_{x,r,s} ] \\
	&\leq & \frac{\alphabj_{x,r}}{2}+ 1 - \overline{Q}[\gamma_{x,r}-  \eta_{x,r,s},s,\overline{\Phi}(x-|U_{l,r,(s)}|)+ \overline{\Phi}(x+|U_{l,r,(s)}|)]\\
	&\leq &  \frac{\alphabj_{x,r}}{2}+  \frac{\alphabj_{x,r}}{2}\leq \alphabj_{x,r}\enspace .
	\eeqn

\end{proof}
\begin{proof}[Proof of \Cref{lem:t_s}]
	From Bernstein inequality, we deduce that, for any positive integers $s$ and $x$,
	\beqn
	\gamma_{x,s}&\leq& 2p\overline{\Phi}(x)+ 2\sqrt{p\overline{\Phi}(x)\log(\alphabj^{-1}_{x,r})}+ \frac{2}{3} \log(\alphabj^{-1}_{x,r})\enspace ;\\
	\eta_{x,r,s}&\geq &2(p-s)\overline{\Phi}(x)-  2\sqrt{p\overline{\Phi}(x)\log(2\alphabj^{-1}_{x,r})} - \frac{2}{3} \log(2\alphabj^{-1}_{x,r})\enspace .
	\eeqn
	Hence, it follows that
	\beqn
	\gamma_{x,s}-\eta_{x,r,s}&\leq & 2s \overline{\Phi}(x)+4 \sqrt{p\overline{\Phi}(x)\log(2\alphabj^{-1}_{x,r})}+ \frac{4}{3} \log(2\alphabj^{-1}_{x,r})\enspace .
	\eeqn
	For $u=x$, we have $\overline{\Phi}(x-u)+ \overline{\Phi}(x+u)\geq \overline{\Phi}(0)= 1/2$ and we derive from Bernstein inequality that
	\beqn
	\psi_{x,r,s}(t)\geq \frac{s}{2} - \sqrt{s\log(2\alphabj^{-1}_{x,r})} - \frac{2}{3}\log(2\alphabj^{-1}_{x,r})\enspace .
	\eeqn
	As a ce, $\psi_{x,r,s}(t)> \gamma_{x,s}-\eta_{x,r,s}$
	as long as
	\[
	s(1- 4\overline{\Phi}(x))> 12\sqrt{p\overline{\Phi}(x)\log(2\alphabj^{-1}_{x,r})} + \frac{12}{3}\log(2\alphabj^{-1}_{x,r}) \enspace .
	\]
	Provided that we take $x\geq 2$, the latter holds if
	\beq\label{eq:cond_s_sparse}
	s\geq 14\sqrt{p\overline{\Phi}(x)\log(2\alphabj^{-1}_{x,r})} + \frac{14}{3}\log(2\alphabj^{-1}_{x,r})
	\eeq
	In view of the definition~\eqref{eq:definition_t_s} of $x_s$, we have $14\sqrt{p\overline{\Phi}(x_s)\log(2\alphabj^{-1}_{x_s,r})}\leq s/2$. Hence, under Condition~\eqref{eq:condition_s_t_s}, \eqref{eq:cond_s_sparse} holds and we conclude that
	 $\psi_{x_s,r,s}(x_s)> \gamma_{x_s,s}-\eta_{x_s,r,s}$.

\end{proof}




\begin{proof}[Proof of \Cref{prop:Sup}]

The following lemma ensures that the partial norm test returns $0$ with high probability jointly at all positions where there is no change-point. We write $\Choose_p^{s}$ for the set of all combinations of $s$ indices taken from $[p]$.
\begin{lemma}[concentration of the pure noise for the second sparse statistic]\label{lem:purenoise2}
If $1 \geq \delta > 0$, then the event
\begin{align*}
	\xPN_1 =
			 \Bigg\{
			 	\forall r\in \cR, l \in \cD_r, s \in \cZ \quad \max_{\SetS\in \Choose_p^{s}}
			 	\sum_{i\in \SetS} \frac{r}{2\sigma^2}
			 	\left(
			 		\bar \Noise_{l, +r,i}
			 		-  \bar\Noise_{l, -r,i}
		 		\right)^2\leq \tsup_{r,s}
		 	 \Bigg\} \enspace
\end{align*}
holds with probability higher than $1 - \delta$.
\end{lemma}

We now state the following lemma, which ensures that the partial norm test returns $1$ with high probability jointly at relevant positions which are close to a change-point.
\begin{lemma}[concentration on the change-points for the second sparse statistic]\label{lem:cp2}
We write $\bar{\cK}^*$ for the set of $k \in [K]$ such that
\begin{itemize}
	\item $s_k \leq \sqrt{p \Log{\frac{n}{r_k\delta} }}$
	\item
	$\sum_{s' = 1}^s\left|U_{\tgSp_k,\rgSp_k,(s')}\right|^2 \geq 4\tsup_{\rgSp_k,s}$
\end{itemize}
If $1 \geq \delta >0$, the event
\begin{align*}
	\xPN_2 = \Bigg\{
		\forall k \in \bar{\cK}^* ~:
		\exists s \in \cZ~\mbox{ s.t.}~
		\Ssup_{\tgSp_k,\rgSp_k,s} > \tsup_{\rgSp_k,s}
	\Bigg\} \enspace,
\end{align*}
holds with probability higher than $1-\delta$.
\end{lemma}
\Cref{lem:purenoise2,lem:cp2} directly imply the result of the proposition.
\end{proof}

\begin{proof}[Proof of \Cref{lem:purenoise2}]
Let $r \in \cR, l \in \cD_r, s \leq \sm_r$ and $\SetS \in \Choose^{s}_p$.
Let $\delta > 0$, $\delta_{r,s} = \left(\frac{r}{n}\right)^2 \left(\frac{s}{2ep}\right)^{s}\delta$. Since $\sqrt{\frac{r}{2\sigma^2}}
\left(
	\bar \Noise_{l, +r,i}
	-  \bar\Noise_{l, -r,i}
\right)$
follows a $\cN (0,1)$ distribution for all $l,r,i$, we have by Bernstein's inequality that with probability larger than $1-\delta_{r,s},$

\begin{align*}
	\sum_{i\in \SetS}
	\left(
		\bar \Noise_{l,+r,i} - \bar \Noise_{l,-r,i}
	\right)^2
	& \leq
		s + 2\sqrt{s\Log{\frac{1}{\delta_{r,s}}}}
	    + \Log{\frac{1}{\delta_{r,s}}}\\
	& \leq
		2\left( s + \Log{\frac{1}{\delta_{r,s}}}\right)\\
	& =
		2\left(s + s\Log{\frac{2ep}{s}} +\Log{\frac{n^2}{r^2\delta}}\right)\\
	& \leq
		4\left( s\Log{\frac{2ep}{s}}+ \Log{\frac{n}{r\delta}}\right)\enspace.
\end{align*}

Since the number of such $\SetS$ is smaller than $\left(\frac{ep}{s}\right)^{s}$, a union bound gives
\begin{align*}
	\Proba{\xPN_1}
	& \geq 1 - \sum_{r \in \cR}\sum_{l \in \cD_r}\sum_{s \in \cZ}\abs{\Choose^s_p}\left(\frac{s}{2ep}\right)^s \left(\frac{r}{n}\right)^2 \delta\\
	& \geq 1 - \sum_{r \in \cR}\sum_{l \in \cD_r}\sum_{s \in \cZ}\left(\frac{1}{2}\right)^s \left(\frac{r}{n}\right)^2 \delta\\
	& \geq 1 - \delta \enspace,
\end{align*}
which yields the result.
\end{proof}

\begin{proof}[Proof of \Cref{lem:cp2}]
Let $k \in \bar{\cK^*}$, and $s \in \cZ$ such that
\begin{equation}
	\label{EnergyGridMax}
		\sum_{i =1}^{s}U^2_{\tgSp_k,\rgSp_k,(i)}
		>
		4\tsup_{\rgSp_k,s}\enspace.
	\end{equation}
	To ease the reading, we write $(\tau,r) = (\tgSp_k,\rgSp_k)$. Then on the event $\xPN_1$ which holds with probability $1 - \delta$, we have
\begin{align*}
	\Ssup_{\tau,r,s}
	& =
	\max_{\SetS\in \Choose_p^{s}} \sum_{i \in \SetS } \frac{r}{2\sigma^2}
	\left(
		\bar\theta_{\tau, +r,i}
		+ \bar \Noise_{\tau, +r,i}
		- \bar\theta_{\tau, -r,i}
		- \bar \Noise_{\tau, -r,i}
	\right)^2 \\
	& \geq
	\max_{\SetS\in \Choose_p^s} \sum_{i \in \SetS}
	\frac{1}{2}U^2_{\tau, r,i}
	- \frac{r}{2\sigma^2}
	  \left(
	  		\bar \Noise_{\tau, +r,i}
	  		- \bar \Noise_{\tau, -r,i}
	  \right)^2\\
	& >
	2\tsup_{r,s} - \tsup_{r,s}\\
	& =
	\tsup_{r,s} \enspace,
\end{align*}
where in the second inequality, we used the fact that $(a+b)^2 \geq \frac{1}{2}a^2 - b^2$ for all $a,b \in \bbR$.
\end{proof}

\begin{proof}[Proof of \Cref{lem:LocalisationEnergy}]
	First remark that there exists a large enough constant $C$ such that for all $r,s \geq 1$,
	\begin{align*}
		\left(\tbj_{r,s}\right)^2 &\leq C\Log{\frac{ep}{s^2}\Log{\frac{n}{r\delta}}}\\
		\sm_r &\leq C \Log{\Log{\frac{ep}{\sm_r^2}}\frac{n}{r\delta}}\enspace ,
	\end{align*}
	where we recall that $\sm_r$ is defined by \Cref{eq:condition_s_t_s} and $\tbj_{r,s}$ by \Cref{eq:definition_t_s}.
	These two inequalies come from the fact that for all $t \geq 2$ and all $A > 0$,  if $t \leq A + \Log{t}$ then $t \leq 2A$.
	Assume that for all $s' = \sm_{\rgSp_k}+1, \dots, s_k$ we have $|U_{\tgSp_k,\rgSp_k,(s')}| < \tbj_{\rgSp_k,s'}$. To ease the notation, we write $\bar s = \sm_{\rgSp_k}\land s_k$ and in what follows we prove that $\sum_{s' = 1}^{\bar s}|U_{\tgSp_k,\rgSp_k,(s')}|^2 > 4\tsup_{\rgSp_k,\bar s}$ when $\kappa_{\mathrm{s}}$ is a large enough constant.
	We have
	\begin{align*}
		\sum_{s' = \sm_{\rgSp_k}+1}^{s_k} U_{\tgSp_k,\rgSp_k,(s')}^2
		& \leq C_1\sum_{i = 0}^{\floor{\Log{s_k}}} 2^i
			   \Log{\frac{ep}{2^{2i}}\Log{\frac{n}{\rgSp_k\delta}}}\\
		& \leq C_1 s_k \Log{2e\Log{\frac{n}{\rgSp_k\delta}}}
				 + C_1\sum_{i = 0}^{\floor{\Log{s_k}}} 2^i \Log{\frac{p}{2^{2(i+1)}}}\enspace,
	\end{align*}
	for some universal constant $C_1$. To handle the second term remark that since $x \mapsto \Log{\frac{p}{x^2}}$ is decreasing, we have
	\begin{align*}
		\sum_{i = 0}^{\floor{\Log{s_k}}} 2^i \Log{\frac{p}{2^{2(i+1)}}}
		& \leq \int_1^{2s_k} \Log{\frac{p}{x^2}} dx\\
		& = 2s_k\Log{\frac{p}{(2s_k)^2}} + 2s_k - 1\\
		& \leq 2s_k \Log{\frac{p}{s_k^2}} \enspace,
	\end{align*}
	and thus
	\begin{align*}
		\sum_{s' = \sm_{\rgSp_k}+1}^{s_k} U_{\tgSp_k,\rgSp_k,(s')}^2
		& \leq 2C_1s_k \Log{2e \frac{p}{s_k^2}
			   \Log{\frac{n}{\rgSp_k \delta}}},
	\end{align*}
	which finally gives
	\begin{align*}
		\sum_{s' =1}^{\bar s} U_{\tgSp_k,\rgSp_k,(s')}^2
		& \geq
		\frac{9}{16}\rgSp_k \Delta_k^2 - 2C_1s_k \Log{ \frac{2ep}{s_k^2}
		\Log{\frac{n}{\rgSp_k\delta}}}\\
		& \geq
				 4\tsup_{\rgSp_k,\bar s}.
	\end{align*}
In the first inequality we used the fact that
$$\abs{\tgSp_k - \tau_k} \leq \frac{1}{4}\rgSp_k,$$
so that for all $i$,
	\begin{align*}
		\abs{\bar\theta_{\tgSp_k,+\rgSp_k,i} - \bar\theta_{\tgSp_k,-\rgSp_k,i}}
		& = \frac{1}{\rgSp_k}\abs{ \left(\rgSp_k + \tgSp_k - \tau_k\right)\mu_{k,i} - \left(\rgSp_k - \tgSp_k + \tau_k\right)\mu_{k-1,i}} \\
		& \geq \left(1 - \frac{\abs{\tgSp_k - \tau_k}}{\rgSp_k} \right)\abs{\mu_{k,i} - \mu_{k-1,i}}\\
		& > \frac{3}{4} \abs{\mu_{k,i} - \mu_{k-1,i}} = \frac{3}{4}U_{k,i} \enspace.
	\end{align*}
In the second inequality, we used the fact that
\begin{itemize}
	\item $8\rgSp_k{\Delta_k^2}
	\geq \kappa_{\mathrm{s}} \sigma^2\NoiseBoundsparse<p, n,\rgSp_k,\delta, s_k>$ for a large enough constant $\kappa_{\mathrm{s}}$ (see \Cref{eq:EnergySparseGauss}),
	\item $x \mapsto x\Log{\frac{ep}{x^2}}$ is increasing for $x \leq p$, so that $s_k$ can be replaced by $\bar s$,
	\item $\bar s \leq C \Log{\Log{\frac{ep}{\bar s^2}}\frac{n}{r\delta}}.$
\end{itemize}
This concludes the proof of the lemma.
\end{proof}

\subsubsection{Proof of \Cref{cor:gaus}}
Let $\xDense$ and $\xSparse$ be two events such that \Cref{prop:dense} and \Cref{prop:Sparse} hold respectively with constants $\kappa_{\mathrm{d}}, \kappa_{\mathrm{s}}$ and with probability $1-2\delta$ and $1 - 4\delta$, and write $\xi = \xDense \cap \xSparse$. From now on, we work on the event $\xi$, which holds with probability $1 - 6\delta$. Let us choose $\Csto \geq 2(\kappa_{\mathrm{d}} \lor \kappa_{\mathrm{s}})$ in \Cref{EnergyGauss}. For all $k$ such that $\tau_k$ is a $\Csto$-high-energy change-point, define
	\begin{equation*}\label{eq:rgrille}
	(\tg_k,\rg_k) = \left\lbrace
	\begin{aligned}
		(\tgD_k,\rgD_k) \text{ if } s_k &> \sqrt{p \Log{\frac{n}{r_k\delta}}}\\
		(\tgSp_k,\rgSp_k) \text{ if } s_k &\leq \sqrt{p \Log{\frac{n}	{r_k\delta}}}\enspace .
	\end{aligned}
	\right.
	\end{equation*}
	$(\rg_k,\tg_k)$ is well defined. Indeed,	If $s_k \leq \sqrt{p\Log{\frac{n}{r_k\delta}}}$ then

	$$ s_k \Log{1 + \frac{\sqrt{p}}{s_k} \sqrt{\Log{\frac{n}{r_k \delta }}}}
	+ \Log{\frac{ n}{r_k \delta}} \geq \frac{1}{2}\left(s_k\Log{\frac{p}{s_k^2} \Log{\frac{n}{r_k\delta} }}  + \Log{\frac{n}{r_k \delta}}\right) \enspace .$$
	Now if $s_k \geq \sqrt{p\Log{\frac{n}{r_k\delta}}}$ then using $\Log{1+x} \geq \frac{x}{2}$ for $x \in [0,1]$ we have
	$$s_k \Log{1 + \frac{\sqrt{p}}{s_k} \sqrt{\Log{\frac{n}{r_k \delta }}}}
	+ \Log{\frac{ n}{r_k \delta}} \geq \frac{1}{2}\NoiseBound<p,n,r_k\delta> \enspace .$$

	According to \Cref{th:g}, it is sufficient to prove that the event $\cA \left( \Theta, \Tg, \cK^*,(\tg_k,\rg_k)_{k\in \cK^*}\right)$ defined in \Cref{sub:general_analysis} holds on $\xi$:
\begin{enumerate}
	\item {\bf (No false positive):} for every $r \in \cR$ and $l \in \cD_r$, if $\Theta$ is constant on $[l - r, l + r)$ then $$\Tg_{l,r} = \TD_{l,r}\lor \Tsparse_{l,r} = 0,$$
	by \Cref{prop:dense} and \Cref{prop:Sparse}.
	\item {\bf (High-energy change-point detection):} for every $k$ such that $\tau_k$ has $\Csto$-high-energy, it holds by definition of $\rgD_k$ and $\rgSp_k$ that
	$$4(\rg_k - 1) \leq r_k.$$
	Moreover, $\Tsparse_{\tg_k,\rg_k} = 1$ if $(\tg_k, \rg_k) = (\tgD_k,\rgD_k)$ by \Cref{prop:Sparse} and $\TD_{\tg_k,\rg_k} = 1$  if $(\tg_k, \rg_k) = (\tgSp_k,\rgSp_k)$ by \Cref{prop:dense}.

\end{enumerate}
\Cref{th:g} ensures that for all $k \in [K]$ such that $\tau_k$ is a $\Csto$-high-energy change-point, there exists $k' \in [\hat{K}]$ such that
	$$|\hat\tau_{k'} - \tau_{k}| \leq \rg_{k} - 1.$$ 
It remains to show that $$\rg_{k} - 1 \leq \frac{r_k^*}{2},$$
where $r^*_k$ is define by  \Cref{eq:VitesseGauss}.
	Using $\Log{1+x} \geq \frac{x}{2}$ for $x \in [0,1]$ and $\Log{1+x} \geq \Log{x}$ for $x \geq 1$ we have
	$$8\rg_k{\Delta_k^2} \leq 4(\kappa_{\mathrm{d}} \lor \kappa_{\mathrm{s}})\left[s_k \Log{1 + \frac{\sqrt{p}}{s_k} \sqrt{\Log{\frac{n}{\rg_k \delta }}}}
	+ \Log{\frac{ n}{\rg_k \delta}}\right],$$
	when $\rg_k \geq 2$.
	Thus $2(\rg_k-1) \leq r^*_k$ for $\Csto \geq 2(\kappa_{\mathrm{d}} \lor \kappa_{\mathrm{s}})$. This concludes the proof of \Cref{cor:gaus}.


\subsection{Proofs for sub-Gaussian multivariate change-point detection}
\label{sub:proofs_in_the_sub_gaussian_setting}
We recall that in this section, we work on the complete grid $\cG_F = J_n$ defined in \Cref{sec:meta_algorithm_for_multi_scale_change_point_detection_on_a_grid}.

\subsubsection{Proof of \Cref{prop:densesg}}
\label{subsub:proof_of_propdensesg}

\paragraph*{Step 1: Introduction of useful high probability events.} We first introduce two events $\xsgD_1$ and $\xsgD_2$ on which the noise can be controlled. Remark that by a simple computation, the noise can be decomposed as follows :
\begin{align*}
	\frac{r}{2}
	\left[
		\norm{\overline{y}_{l,+r} - \overline{y}_{l,-r} }^2
		- \norm{ \overline\theta_{l,-r} - \overline\theta_{l,+r}}^2
	\right] - \sigma^2 p
	=
	r \proscal
	<\overline{\Noise}_{l,+r} - \overline{\Noise}_{l,-r},
	 \overline{\theta}_{l,+r} - \overline{\theta}_{l,-r}>
	+ \frac{r}{2} \norm{\overline{\Noise}_{l,+r} - \overline{\Noise}_{l,-r} }^2
	- \sigma^2 p \enspace.
\end{align*}

The first term written as
\begin{align*}
	r \proscal
	<\overline{\Noise}_{l,+r} - \overline{\Noise}_{l,-r},
	 \overline{\theta}_{l,+r} - \overline{\theta}_{l,-r}>
\end{align*}
is a crossed term between the noise and the mean vector $\theta$.
\Cref{lem:crossedterms} states that for $l$ equal to a true change-point $\tau_k$ and $r$ of order $r_k^*$, it is controlled on event $\xsgD_1$ with high probability.

\begin{lemma}[concentration of the crossed terms]\label{lem:crossedterms}
Assume that $\kappa$ is a large enough universal constant. The event
\begin{align*}
    \xsgD_1
    = \Bigg\{\forall k \in[K]~\mathrm{s.t.~Equation~\eqref{eq:Energydsg}~holds~for~}k,~ \\
    \rgD_k\abs{ \proscal<\overline{\Noise}_{\tau_k,+\rgD_k} - \overline{\Noise}_{\tau_k,-\rgD_k},\overline{\theta}_{\tau_k,+\rgD_k} - \overline{\theta}_{\tau_k,-\rgD_k}>}
\leq \frac{\rgD_k}{4}\norm{ \overline\theta_{\tau_k,+\rgD_k} - \overline\theta_{\tau_k,-\rgD_k}}^2 \Bigg\}
\end{align*}
holds with probability higher than $1 - \delta$.
\end{lemma}

The second term written as
\begin{align*}
	\frac{r}{2}
	\norm{\overline{\Noise}_{l,+r} - \overline{\Noise}_{l,-r}}^2
	- \sigma^2 p \enspace,
\end{align*}
is a term of pure noise. \Cref{lem:purenoise} states that it is controlled on event $\xsgD_2$ with high probability.

\begin{lemma}[concentration of the pure noise]\label{lem:purenoise}
There exists a constant $\Cxi >0$ such that the event
\begin{align*}
	\xsgD_2 =
	\left\{
		\forall (l,r)\in J_n,
		~\abs{
			\frac{r}{2}
			\norm{\overline{\Noise}_{l,+r}
			      - \overline{\Noise}_{l,-r} }^2
			      - \sigma^2 p
			 }
		 \leq \Cxi L^2
		 	\left(
		 		\sqrt{p \Log{\frac{n}{r\delta}}}
		 		+ \Log{ \frac{n}{r\delta}}
	 		\right)
	\right\}
\end{align*}
holds with probability higher than $1 - 2\delta$.
\end{lemma}

Set now
\begin{align*}
	\xsgD :=  \xsgD_1 \cap \xsgD_2 \enspace.
\end{align*}
Note that
\begin{align*}
	\bbP(\xsgD) \geq 1 - 3\delta \enspace.
\end{align*}

\paragraph*{Step 2: Study in the `no change-point' situation.}
We remind that $\cH_0$ stands for elements $(l,r)$ such that there is no change-point in $[l-r,l+r)$ and that it is defined in \Cref{eq:Ho}.
Consider $(l,r) \in J_n\cap \cH_0$. Note that since $\{\tau_k, k \in [K]\} \cap [l-r, l+r) = \emptyset$, we have $\overline\theta_{l,-r} = \overline\theta_{l,+r}$ so that
\begin{align*}
	\frac{r}{2}
	\norm{\overline\theta_{l,-r} - \overline\theta_{l,+r}}^2
	= 0 \enspace,
\end{align*}
and
\begin{align*}
	r \proscal<\overline{\Noise}_{l,+r} - \overline{\Noise}_{l,-r},
			   \overline{\theta}_{l,+r} - \overline{\theta}_{l,-r}> = 0
			   \enspace.
\end{align*}
Moreover we have on $\xsgD$ that - see \Cref{lem:purenoise}
\begin{align*}
	\abs{
		\frac{r}{2}
		\norm{\overline{\Noise}_{l,+r} - \overline{\Noise}_{l,-r} }^2
		- \sigma^2 p
		}
	\leq \Cxi L^2 \left(
						\sqrt{p \Log{\frac{n}{r\delta}}} + \Log{ \frac{n}{r\delta}}
				  \right)
	\leq \sigma^2 \tDsg_r,
\end{align*}
for $\thresh \geq \Cxi$ - note that $\Cxi>0$ is a universal constant. And so
\begin{align*}
	\psiDsg_{l,r}\leq  \tDsg_r \enspace,
\end{align*}
so that
\begin{align*}
	\TDsg_{l,r} = 0 \enspace,
\end{align*}
on $\xsgD$. This concludes the proof of the first part of the proposition.

\paragraph*{Step 3: Study in the `change-point' situation.}

Consider $k\in [K]$ such that $\tau_k$ is a $\kappa$-dense high-energy change-point - see Equation~\eqref{eq:Energydsg}. We have
\begin{align*}
	\begin{aligned}
	\frac{\rgD_k}{2}
	\norm{\overline\theta_{\tau_k,-\rgD_k}
		  - \overline\theta_{\tau_k,+\rgD_k}
		  }^2
	& \geq \frac{\kappa}{8} L^2
		   \left(
		   		\sqrt{p \Log{\frac{n}{\rgD_k\delta}}}
		   		+ \Log{ \frac{n}{\rgD_k\delta}}
   		   \right).
	\end{aligned}
\end{align*}
So on $\xsgD$ choosing $\kappa$ large enough implies that - see \Cref{lem:crossedterms}
\begin{align*}
	\rgD_k
	\abs{
		\proscal<\overline{\Noise}_{\tau_k,+\rgD_k} - \overline{\Noise}_{\tau_k,-\rgD_k},
				 \overline{\theta}_{\tau_k,+\rgD_k} - \overline{\theta}_{\tau_k,-\rgD_k}>
		}
	\leq
	\frac{\rgD_k}{4} \norm{ \overline\theta_{\tau_k,+\rgD_k}
	- \overline\theta_{\tau_k,-\rgD_k}}^2\enspace.
\end{align*}
Moreover we have on $\xsgD$ that - see \Cref{lem:purenoise}
\begin{align*}
	\abs{\frac{\rgD_k}{2}
	\norm{
		\overline{\Noise}_{\tau_k,+\rgD_k} - \overline{\Noise}_{\tau_k,-\rgD_k}
		 }^2
	 - \sigma^2 p}
	 \leq  \Cxi L^2 \left(
	 					\sqrt{p \Log{\frac{n}{\rgD_k\delta}}} + \Log{ \frac{n}{\rgD_k\delta}}
 					\right)
	 \leq \sigma^2 \tDsg_{\rgD_k}\enspace,
\end{align*}
for $\thresh \geq \Cxi$ - note that $\Cxi>0$ is a universal constant. Thus on $\xsgD$, combining the three previous displayed equations implies
\begin{align*}
	\psiDsg_{\tau_k,\rgD_k}
	&\geq \frac{\rgD_k}{4 \sigma^2} \norm{
							  \overline\theta_{\tau_k,+\rgD_k}
		                      - \overline\theta_{\tau_k,-\rgD_k}
		 					  }^2
		 - \tDsg_{\rgD_k}\\
	&\geq \left(\frac{\Csto}{16} - \thresh\right) \frac{L^2}{\sigma^2}
		 \left(
		 	  \sqrt{p \Log{\frac{n}{\rgD_k\delta}}}
		 	  + \Log{ \frac{n}{\rgD_k\delta}}
		 \right)>
	\tDsg_{\rgD_k}\enspace,
\end{align*}
since $\kappa > 32\thresh$.
And so on $\xsgD$:
\begin{align*}
	\TDsg_{\tau_k,\rgD_k} = 1 \enspace.
\end{align*}
This concludes the proof of the second part of the proposition.

\begin{proof}[Proof of \Cref{lem:crossedterms}]

Let $k$ be in $[K]$ and such that Equation~\eqref{eq:Energydsg} is satisfied. Remark that $\theta$ is constant on $[\tau_k - \rgD_k,\tau_k)$ and is equal to $\mu_{k-1}$, and is also constant on $[\tau_k,\tau_k + \rgD_k)$ and is equal to $\mu_k$.
First, from the definition of the $\psi_2$-norm of a vector, there exists a universal constant $C >0$ such that for all $k = 1 \dots K$,
\begin{align*}
	\norm{\rgD_k\proscal<\overline{\Noise}_{\tau_k,+\rgD_k} - \overline{\Noise}_{\tau_k,-\rgD_k},\overline{\theta}_{\tau_k,+\rgD_k} - \overline{\theta}_{\tau_k,-\rgD_k}>}_{\psi_2} &\leq \rgD_k\norm{\overline{\Noise}_{\tau_k,+\rgD_k} - \overline{\Noise}_{\tau_k,-\rgD_k}}_{\psi_2} \abs{\mu_k - \mu_{k-1}}\\
	&\leq C\sqrt{\rgD_k} \norm{\Noise_1}_{\psi_2} \abs{\mu_k - \mu_{k-1}}\\
	&\leq C\sqrt{\rgD_k}L\abs{\mu_k - \mu_{k-1}}\\
	&\leq C L \sqrt{r_k \Delta_k^2} \enspace.
\end{align*}
Thus by definition of sub-Gaussianity, for all $t > 0$,
\begin{align*}
	\Proba{
		\rgD_k
		\abs{
			\proscal<\overline{\Noise}_{\tau_k,+\rgD_k} - \overline{\Noise}_{\tau_k,-\rgD_k},
					 \overline{\theta}_{\tau_k,+\rgD_k} - \overline{\theta}_{\tau_k,-\rgD_k}>
			 }\geq t
		   }
	\leq \exp \left(-c \frac{t^2}{L^2r_k \Delta_k^2} \right)\enspace,
\end{align*}
for some constant $c>0$.
Finally we apply the concentration inequality to $t = \frac{r_k\Delta_k^2}{4}$ - remembering that $\tau_k$ is a $\kappa$-dense high-energy change-point in the sense of Equation~\eqref{eq:Energydsg} - and sum over $k$ to obtain a union bound over $\xi_2^c$ :
\begin{align*}
	\Proba{\xi_2^c}
	& \leq \sum_{k=1}^K
		   \Proba{
		   		 r\abs{\proscal<\overline{\Noise}_{\tau_k,+\rgD_k}
		   		 				- \overline{\Noise}_{\tau_k,-\rgD_k},
		   		 				\overline{\theta}_{\tau_k,+\rgD_k}
		   		 				- \overline{\theta}_{\tau_k,-\rgD_k}>
	 				   }
			     \geq \frac{r_k \Delta_k^2}{4}
		   		 } \\
	& \leq \sum_{k=1}^K \exp \left( -c\frac{r_k \Delta_k^2}{16L^2} \right)\\
	& \leq \sum_{k=1}^K \exp
			\left(
				-c' \kappa\log \left(\frac{n}{\rgD_k} \delta^{-1} \right)
				\right) & (c' = c/16)\\
	& \leq \sum_{k=1}^K \left( \frac{\rgD_k}{n} \right)^{c'\kappa} \delta^{c'\kappa}\\
	& \leq \delta \enspace,
\end{align*}
where the last inequality comes from the fact that $\sum_{k=1}^K \rgD_k \leq n$ and the fact that $\kappa$ is chosen large enough so that $c'\kappa \geq 1$.
\end{proof}

\begin{proof}[Proof of \Cref{lem:purenoise}]
Remark first that by homogeneity, we can assume without loss of generality that $L = 1$. To provide a proof, we will use the Hanson-Wright inequality in high dimension, which is a way to control quadratic forms of the noise.

\begin{lemma}[Hanson-Wright inequality in high dimension] \label{HW}
Let $A = (a_{ij})$ be a $m\times m$ matrix and $\Noise_1, \dots, \Noise_m$ be sub-Gaussian vectors of dimension $p$ with norm smaller than $1$. Then
\begin{align*}
	\Proba{
		  \abs{\sum_{1\leq i,j \leq m} a_{i,j} \proscal<\Noise_i,\Noise_j>  - \Esper{\sum_{1\leq i,j \leq m} a_{i,j} \proscal<\Noise_i,\Noise_j>}} \geq t
	  	  }
  	\leq 2\exp \left(
  					 -c\min \left(\frac{t^2}{ p \|A\|_F^2},
  					 			  \frac{t}{ \|A\|_{op}}\right)
			   \right)\enspace,
\end{align*}
where $c$ is an absolute constant, $\|A\|_F^2 = \sum\limits_{i,j} a_{i,j}^2$ is the squared Frobenius norm of $A$ and $\|A\|_{op}$ is the operator norm of $A$.
\end{lemma}
The proof of this lemma relies on the classical Hanson Wright inequality that is proved for example in \cite{rudelson2013hanson}. To prove the proposition, we will use a chaining argument. To this end, we let $(N_{\nnu})_{\nnu\geq 0}$ be the following covering sets of $J_n$ :
\begin{align*}
	N_{\nnu}
	=
	J_n \cap\left\{i2^{\kappa_1 - {\nnu}}, i \in\bbN \right\}^2 \enspace,
\end{align*}
where we define $\kappa_1 = \left\lfloor\log_2 (n)\right\rfloor$, and more generally
$\kappa_r = \left\lfloor\log_2 (n/r)\right\rfloor$ for $r = 1, \dots n$.
Remark that the higher $u$ is, the finer the covering set $N_u$ is, and $N_{\kappa_1} = J_n$.
\poubelle{For simplicity, we write $N_u^1$ (resp. $N_u^2$) the projection of $N_u$ onto the first coordinate (resp. the second coordinate).}
For all $u \geq 0$, we define the projection map $\pi_{\nnu}$ from $J_n$ to $N_u$ by
\poubelle{$$
 \left\{
    \begin{array}{ll}
    	\pi_{\nnu}^1(l,r) =  \underset{\hat l \in N_{\nnu}^1}{\argmin} \abs{\hat l-l}\\
        \pi_{\nnu}^2(l,r) = \underset{\hat r \in N_{\nnu}^2}{\argmin} \abs{\hat r -r}.

    \end{array}
\right.
$$}
\begin{align*}
	\pi_u{(l,r)}
	=
	\argmin_{(\hat l,\hat r) \in N_u}
	\left(|\hat l - l| + |\hat r - r| \right) \enspace.
\end{align*}
In the sequel, we will use the slight abuse of notation for $(l,r)$ in $J_n$ :
\begin{align*}
	(l_u,r_u) = \pi_u (l,r) \enspace.
\end{align*}
A useful lemma to control the distance between $(l,r)$ and its projection $(l_u, r_u)$ can be stated as follow.
\begin{lemma}\label{lem:mesh}
For all $(l,r) \in J_n$ and $0 \leq u \leq \kappa_1$ such that $N_\nnu\neq \emptyset$,
\begin{align*}
	|l_\nnu - l| + |r_\nnu - r| \leq 2\frac{n}{2^\nnu} \enspace.
\end{align*}
\end{lemma}

Let $(l,r) \in J_n$. From know on, we write $\Noise_{l,+r}=r\bar\Noise_{l,+r} = \sum_{t = l}^{l+r-1}\Noise_{t}$ and $\Noise_{l,-r} =  r\bar\Noise_{l,+r}$.
The chaining relation can be written as
\begin{align*}
	\frac{r}{2} \Diffcarremoy{l}{r} - \sigma^2 p
	& = \frac{1}{2r} \left[\Diffcarre{l_{\kappa_r}}{r_{\kappa_r}}
	    - 2r_{\kappa_r}\sigma^2 p\right] \\
 	& \quad + \frac{1}{2r}\sum_{v = \kappa_r}^{\kappa_1}\left[
 	  \Diffcarre{l_{v+1}}{r_{v+1}} - \Diffcarre{l_v}{r_v}
 	  - 2(r_{v+1}-r_v)\sigma^2 p \right] \enspace.
\end{align*}
Remark that the chaining summation starts at scale $u = \kappa_r$ so that $\tfrac{n}{2^u} \asymp r$. The first term of the chaining is an approximation on the grid at level $u$ of the term $\frac{r}{2} \Diffcarremoy{l}{r} - \sigma^2 p$. The second term can be viewed as an error term, and we will show that it is of the same order as the first term. Since both terms are quadratic forms of the noise, we will need an upper bound on the norm of their corresponding matrix to apply the Hanson Wright inequality - see \Cref{HW}.
\begin{lemma}[Control of the Frobenius norm]\label{lem:controlnorme}
Let $(l,r)$ be a fixed element of $J_n$.
Let $A$ and $B$ be the corresponding matrix of the two following quadratic form :
\begin{align*}
	\Noise^T A \Noise = \Diffcarre{l}{r} \quad \mbox{and} \quad
	\Noise^T B \Noise = \Diffcarre{l}{r} - \Diffcarre{l'}{r'} \enspace.
\end{align*}
Then
\begin{align*}
	\norm{A}_F^2
	&\leq 16r^2\\
	\norm{B}_F^2
	&\leq 24 \left( \abs{l-l'} + \abs{r-r'} \right) (r +r' + |l-l'|) \enspace.
\end{align*}
\end{lemma}
The following lemma aims at upper bounding the first term of the chaining relation with high probability.

\begin{lemma}\label{lem:concentrationN}
There exists a constant $C_N$ such that for all $n$, the event
\begin{align*}
	\xsgD_N
	=
	\bigcap_{u \geq 0}
	\bigcap_{\substack{(l,r) \in N_u\\ r \leq 3\frac{n}{2^u}}}
	\left\{
		\abs{\Diffcarre{l}{r} - 2r \sigma^2 p}
		\leq C_Nr\NoiseBoundd<p,2^u \delta^{-1}>
	\right\} \enspace.
\end{align*}
holds with probability higher than $1 - \delta$.
\end{lemma}
For $u = \kappa_r$, $(l_u, r_u) \in N_u$ \Cref{lem:mesh} gives $r_u \leq r + 2\frac{n}{2^u} \leq  3\frac{n}{2^u}$. Consequently, on the event $\xsgD_N$, we obtain
\begin{align*}
	\abs{\frac{1}{2r}\Diffcarre{l_{\kappa_r}}{r_{\kappa_r}}
	- \frac{r_{\kappa_r}}{r}\sigma^2 p}
	\leq
	C_N' \NoiseBoundsg<p,n,r\delta>\enspace,
\end{align*}
for $C_N'$ a large absolute constant.
To upper bound the second term, we use the following lemma :

\begin{lemma}\label{lem:concentrationD}
For all $(l,r) \mbox{ and } (l',r')$ in $J_n$, set
{\small
\begin{align*}
	\xsgD_{\Delta,v}(l,r,l',r')
	=
	\left\{
		 \abs{\Diffcarre{l'}{r'}- \Diffcarre{l}{r} - 2(r'-r)\sigma^2 p}
		 \leq C_\Delta\sqrt{\frac{rn}{2^v}}\NoiseBoundd<p,2^v \delta^{-1}>
	\right\} \enspace.
\end{align*}
}
There exists a constant $C_\Delta$ such that, for all $n$, the event
\begin{align*}
	\xsgD_{\Delta}
	=
	\bigcap_{v \geq 0}
	\left\{
		\xsgD_{\Delta,v}
		\left( l,r,l', r' \right) \mbox{ holds for all }
		((l,r),(l',r')) \in N_v \times N_{v+1}
		\mbox{ s.t. } |l-l'| + |r-r'| \leq 3\frac{n}{2^v}
	\right\} \enspace.
\end{align*}
holds with probability higher than $1 - \delta$.
\end{lemma}
For $v \geq \kappa_r$, $((l_v,r_v),(l_{v+1},r_{v+1})) \in N_v \times N_{v+1}$ and by \Cref{lem:mesh},
\begin{align*}
	|r_v - r_{v+1}| + |l_v - l_{v+1}|
	& \leq |r_v - r| + |l_v - l| + |r - r_{v+1}| + |l - l_{v+1}|\\
	& \leq 3\frac{n}{2^v}.
\end{align*}
Therefore, on the event $\xsgD_\Delta$,
\begin{align*}
	& \quad \abs{\frac{1}{2r}\sum_{v = \kappa_r}^{\kappa_1-1 }\left[ \Diffcarre{l_{v+1}}{r_{v+1}} - \Diffcarre{l_v}{r_v} - 2(r_{v+1}-r_v)\sigma^2 p\right]} \\
	& \leq C_\Delta \frac{1}{2r}\sum_{v=\kappa_r}^{\kappa_1 -1} \sqrt{\frac{r_v n}{2^v}}\NoiseBoundd<p,2^v\delta^{-1}>\\
	& \leq C_\Delta' \sum\limits_{v'\geq 0} \frac{1}{2^{v'}}\NoiseBoundsg<p,n2^{v'},r\delta>\\
	& \leq C_\Delta'' \NoiseBoundsg<p,n,r\delta>,
\end{align*}
where $C_\Delta' ,C_\Delta''$ are large absolute constants. Hence, letting $\Cxi = C_N' + C_\Delta''$ we obtain
\begin{align*}
	\xsgD_N \cap \xsgD_\Delta \subset \xsgD_2 \enspace,
\end{align*}
which must be of probability higher than $1 - 2\delta$.
\end{proof}

\begin{proof}[Proof of \Cref{lem:mesh}]
	Since the mesh of the grid $N_u$ is equal to $2^{\kappa_1 - u} \leq \frac{n}{2^u}$, there exists $(\tilde{l}, \tilde{r}) \in N_u$ such that
	\begin{align*}
		|l- \tilde{l}|
		\leq \frac{n}{2^u} \quad \text{and} \quad |r- \tilde{r}|
		\leq \frac{n}{2^u} \enspace.
	\end{align*}
	\end{proof}

\begin{proof}[Proof of \Cref{lem:controlnorme}]
Let us write
\begin{align*}
	\Noise^T A \Noise
	=
	\sum_{l-r \leq i,j < l+r}a_{ij}\proscal<\Noise_i,\Noise_j> \quad \mbox{and} \quad
	\Noise^T B \Noise
	=
	\sum_{ m_1\leq i,j < m_2 }b_{ij}\proscal<\Noise_i,\Noise_j>,
\end{align*}
where $m_1 =\min(l-r, l'-r')$, $m_2 = \max(l+r,l'+r')$. Remark that for all $i,j$ in $[l-r,l+r)$, $a_{ij} \leq 2$. This gives the first inequality. \\
For the second inequality, assume without loss of generality that $l \leq l'$. As for the first inequality, $b_{ij} \leq 2$ for all $i,j \in [m_1,m_2)$. Remark that $b_{ij}$ can be non zero only if $(i,j)$ is in one of the following cases :
\begin{enumerate}
	\item $i \mbox{ or } j$ is in $[\min(l+r,l'+r'), \max(l+r,l'+r'))$
	\item $i \mbox{ or } j$ is in $[\min(l-r,l'-r'), \max(l-r,l'-r'))$
	\item $i \mbox{ or } j$ is in $[l, l').$
\end{enumerate}
Hence there is at most $(4( |l-l'| + |r-r'|) + 2|l-l'|)(r+r' +|l-l'|)$ non zero $b_{ij}$, and we obtain the second inequality.
\end{proof}

\begin{proof}[Proof of \Cref{lem:concentrationN}]
The probability of $(\xsgD_N)^c$ can be written as :
\begin{align*}
	\begin{aligned}
	\Proba{(\xsgD_N)^c}
	= \bbP
	\Bigg(
		\exists u \geq 0,\exists (l,r) \in N_u  ~\text{ s.t. }~ r \leq 3\frac{n}{2^u} \mbox{ and }\\ \abs{\Diffcarre{l}{r} - 2r \sigma^2 p} \leq C_N r\NoiseBoundd<p,2^u\delta^{-1}>
	\Bigg) \enspace.
	\end{aligned}
\end{align*}
First, fix $u \geq 0$ and $(l,r) \in N_u$ such that $r \leq 3\frac{n}{2^u}$.

Applying the first inequality of \Cref{lem:controlnorme} and the Hanson-Wright inequality - see \Cref{HW}, we obtain for all $t \geq 0$
\begin{align*}
	\Proba{\abs{\Diffcarre{l}{r} - 2r \sigma^2 p} \geq t}
	&
	\leq 2 \exp
		 \left(
		 	-c\min \left(\frac{t^2}{pr^2}, \frac{t}{r}\right)
	 	\right) \enspace,
\end{align*}
where $c$ is an absolute constant.
Choosing
\begin{align*}
	t &= C_N r\NoiseBoundd<p,2^u \delta^{-1}> \enspace,
\end{align*}
we obtain
\begin{align*}
	\Proba{\abs{\Diffcarre{l}{r} - 2r \sigma^2 p}
	\geq C_N r\NoiseBoundd<p,2^u\delta^{-1}>}
	\leq C\left( \frac{1}{2^u} \right)^{cC_N} \delta^{cC_N} \enspace,
\end{align*}
where $c,C$ are absolute constants. Since the cardinal of $N_u$ is upper bounded by $ 2^{2u +2}$,
A union bound on each $N_u$ for each $u \geq 0$ gives :
\begin{align*}
	\Proba{(\xsgD_N)^c}
	& \leq \sum_{u \geq 0} C\abs{N_u}
	       \left(\frac{1}{2^u} \right)^{cC_N} \delta^{cC_N}\\
	& \leq \sum_{u \geq 0} 4C
		   \left(\frac{1}{2^u} \right)^{2 - cC_N} \delta^{cC_N}\enspace,
\end{align*}
which is convergent.
For $C_N$ large enough, we obtain $\Proba{\xi_N^c} \leq 1 - \delta$.
\end{proof}

\begin{proof}[Proof of \Cref{lem:concentrationD}]

\begin{align*}
	\Proba{(\xsgD_\Delta)^c}
	& = \Proba{
			   \exists v \geq 0, \exists ((l,r),(l',r')) \in N_v \times N_{v+1}
			   \mbox{ s.t. } |l-l'| + |r-r'| \leq 4 \frac{n}{2^v}
			   \mbox{ and } (\xsgD_{\Delta,v}\left( l,r,l', r' \right))^c
			   \mbox{ holds }
			   } \enspace.
\end{align*}
First fix $v \geq 0 $ and $((l,r),(l',r')) \in N_v \times N_{v+1}$.
Remark that by definition of $N_v$,
\begin{align*}
	r \geq \frac{n}{2^{v+1}} \enspace.
\end{align*}
Thus,
\begin{align*}
	r + r' + |l-l'| \leq 2r + |l-l'| + |r-r'| \leq 10r \enspace.
\end{align*}

Then by \Cref{lem:controlnorme}, letting $B$ be the matrix such that $\Noise^T B\Noise = \Diffcarre{l'}{r'} - \Diffcarre{l}{r}$, we obtain
\begin{align*}
	\norm{B}^2 \leq \norm{B}_F^2 \leq 40r\frac{n}{2^v} \enspace.
\end{align*}
Thus, by the Hanson Wright inequality - see \Cref{HW},
\begin{align*}
	\Proba{
		  \abs{\Noise^T B_{\nnu} \Noise - \Esper{\Noise^T B_{\nnu} \Noise}}
		  \geq t
		  }
	\leq
	2\exp\left(
			  -c \min \left(\frac{2^v }{pnr}t^2,\sqrt{\frac{2^v }{nr}}t\right)
		  \right) \enspace.
\end{align*}
From now on, we choose
\begin{align*}
	t & = C_\Delta\sqrt{\frac{rn}{2^v}}\NoiseBoundd<p,2^v\delta^{-1}> \enspace.
\end{align*}
There are at most $2^{4v + 6}$ elements in $N_v \times N_{v+1}$.
Therefore, a union bound on $v \geq 0$ and $N_v \times N_{v+1}$ gives
\begin{align*}
	\Proba{(\xsgD_\Delta)^c} & \leq \sum\limits_{u \geq 0}2|N_v \times N_{v+1}|\left( 2^v \right)^{-cC_\Delta}\delta^{cC_\Delta}\\
	& \leq \sum_{u \geq 0}2^7 \left( 2^v \right)^{4 - cC_\Delta}\delta^{cC_\Delta}\\
	& \leq C\delta^{cC_\Delta},
\end{align*}
where the last inequality holds if $C_\Delta$ is large enough, for $c,C$ universal constants.
\end{proof}

\subsubsection{Proof of \Cref{prop:sparsesg}}
\label{subsub:proof_of_prop:sparsesg}

\paragraph*{Step 1: Introduction of useful high probability events.} Let $s\leq p$ and consider $S \in \Choose_p^s$. In what follows and for an vector $u \in \bbR^p$, we write $u^{(S)}$ for the vector $u$ restricted to the set $S$.

Remark that by a simple computation, the noise can be decomposed as follows :
\begin{align*}
	\frac{r}{2}
	\left[
		\norm{\yS_{l,+r} - \yS_{l,-r} }^2 - \norm{ \tS_{l,-r} - \tS_{l,+r}}^2
	\right] - \sigma^2 s\\
	=
	r \proscal<\NS_{l,+r} - \NS_{l,-r},\tS_{l,+r} - \tS_{l,-r}>
	+ \frac{r}{2} \norm{\NS_{l,+r} - \NS_{l,-r} }^2
	- \sigma^2 s \enspace.
\end{align*}
The first term written as
\begin{align*}
	r \proscal<\NS_{l,+r} - \NS_{l,-r},\tS_{l,+r} - \tS_{l,-r}> \enspace,
\end{align*}
is a crossed term between the noise and the mean vector $\theta$. \Cref{lem:crossedterms} states that for $l$ equal to a true change-point $\tau_k$, $r$ of order $r_k^*$, and $S$ being the corresponding support of the change-point, it is controlled on event $\xsgSp_1$ with high probability.

\begin{lemma}\label{lem:crossedtermsm}
For $k\in[K]$, let us write $S_k\subset[K]$ for the support of $\mu_k- \mu_{k-1}$. Assume that $\Csto$ is a large enough universal constant. The event
\begin{align*}
    \xsgSp_1
    := \xsgSp_1(\delta)
     = \Bigg\{\forall k \in[K]~\mathrm{s.t.~Equation~\eqref{eq:Energymsg}~holds~for~}k,~ \\
    \rgD_k\abs{ \proscal<\NSk_{\tau_k,+\rgD_k} - \NSk_{\tau_k,-\rgSp_k},\tSk_{\tau_k,+\rgSp_k} - \tSk_{\tau_k,-\rgSp_k}>}
	\leq
	\frac{\rgD_k}{4}\norm{ \overline\theta_{\tau_k,+\rgSp_k} - \overline\theta_{\tau_k,-\rgSp_k}}^2 \Bigg\}
\end{align*}
holds with probability higher than $1 - \delta$.
\end{lemma}
The proof of this lemma follows directly from the one of \Cref{lem:crossedterms}, restricting the term corresponding to change-point $k$ to $S_k$ - and diminishing the deviation by doing so.

The second term written as
$$\frac{r}{2} \norm{\NS_{l,+r} - \NS_{l,-r} }^2 - \sigma^2 s$$
is a term of pure noise. \Cref{lem:purenoiseS} states that it is controlled on event $\xsgSp_2(S)$ with high probability.
\begin{lemma}\label{lem:purenoiseS}
There exists a constant $\Cxi >0$ such that the event
\begin{align*}
\begin{aligned}
	\xsgSp_2(S)
	:= \xsgSp_2(S,\delta)
	 = \Bigg\{
	 \forall (l,r)\in J_n,~\abs{\frac{r}{2} \norm{\NS_{l,+r} - \NS_{l,-r} }^2 - \sigma^2 s} \\
	\leq
	\Cxi L^2 \left(
			 	  \sqrt{s \Log{\frac{n}{r\delta}}} + \Log{ \frac{n}{r\delta}}
		 	  \right) \Bigg\}
\end{aligned}
\end{align*}
holds with probability higher than $1 - 2\delta$.
\end{lemma}
The proof of this lemma is exactly the same as the one of \Cref{lem:purenoise}, restricting all vectors to $S$.

Set $\delta_s = \delta/(2^s {\binom{p}{s}})$. \Cref{lem:purenoiseS} implies that with probability larger than $1 - 2 \delta$, $\forall (l,r)\in J_n$, $\forall S \subset [p]$
$$\abs{\frac{r}{2} \norm{\NS_{l,+r} - \NS_{l,-r} }^2 - \sigma^2 s}
\leq  \Cxi L^2 \left( \sqrt{s \Log{\frac{n}{r\delta_s}}} + \Log{ \frac{n}{r\delta_s}}  \right).$$
And so since ${\binom{p}{s}} \leq \left(\frac{ep}{s}\right)^s$, we have probability larger than $1 - 2\delta$, $\forall (l,r)\in J_n$, $\forall S \subset [p]$
\begin{align*}
	\begin{aligned}
	\abs{\frac{r}{2} \norm{\NS_{l,+r} - \NS_{l,-r} }^2 - \sigma^2 s}
	&\leq  \Cxi L^2 \left( \sqrt{s \Log{\frac{n}{r\delta}} + s\log\left(\frac{2ep}{s}\right)} + \Log{ \frac{n}{r\delta}} + s\log\left(\frac{2ep}{s}\right) \right)\\
	&\leq  4\Cxi L^2 \left(\Log{ \frac{n}{r\delta}} + s\log\left(\frac{2ep}{s}\right) \right).
	\end{aligned}
\end{align*}
And so the event
\begin{equation}\label{eq:purenoiseS}
\begin{aligned}
    \xsgSp_2:= \xsgSp_2(\delta) = \Bigg\{  \forall (l,r)\in J_n, \forall S \subset [p], \abs{\frac{r}{2} \norm{\NS_{l,+r} - \NS_{l,-r} }^2 - \sigma^2 s} \\
\leq  4\Cxi L^2 \left(\Log{ \frac{n}{r\delta}} + s\log\left(\frac{2ep}{s}\right) \right)\Bigg\}
\end{aligned}
\end{equation}
has probability larger than $1 - 2\delta$.

Set now
$$\xsgSp :=  \xsgSp_1 \cap \xsgSp_2.$$
Note that
$$\bbP(\xsgSp) \geq 1 - 3\delta.$$

\paragraph*{Step 2: Study in the `no change-point' situation.}

Consider $(l,r) \in J_n$ such that $\{\tau_k, k \in [K]\} \cap [l-r, l+r) = \emptyset$, and $S \subset [p]$. Note that since $\{\tau_k, k \in [K]\} \cap [l-r, l+r) = \emptyset$, we have $\tS_{l,-r} = \tS_{l,+r}$ so that
$$\frac{r}{2}\norm{ \tS_{l,-r} - \tS_{l,+r}}^2  = 0,$$
and
$$r \proscal<\NS_{l,+r} - \NS_{l,-r},\tS_{l,+r} - \tS_{l,-r}>  = 0.$$
Moreover we have on $\xsgSp$ that - see Equation~\eqref{eq:purenoiseS}
$$\abs{\frac{r}{2} \norm{\NS_{l,+r} - \NS_{l,-r} }^2 - \sigma^2 s} \leq 4\Cxi L^2 \left(\Log{ \frac{n}{r\delta}} + s\log\left(\frac{2ep}{s}\right) \right) \leq  \sigma^2\tDmsg_r,$$
for $\thresh \geq 4\Cxi$ - note that $\Cxi>0$ is a universal constant. And so
$$ \psiDmsg_{l,r}\leq  \tDmsg_r,$$
so that on $\xsgD$,
$$\TDmsg_{l,r} = 0 \enspace .$$
This concludes the proof of the first part of the proposition.

\paragraph*{Step 3: Study in the `change-point' situation.}

Consider $k\in [K]$ such that $\tau_k$ is a $\kappa$-sparse high-energy change-point, - see Equation~\eqref{eq:Energymsg}. Since $S_k$ is the support of $\mu_k - \mu_{k-1}$ - and therefore of $\overline{\theta}_{\tau_k,-\rgSp_k} - \overline\theta_{\tau_k,+\rgSp_k}$ - we have
\begin{align*}
	\begin{aligned}
	\frac{\rgSp_k}{2}\norm{ \tSk_{\tau_k,-\rgSp_k} - \tSk_{\tau_k,+\rgSp_k}}^2  &\geq \frac{\kappa}{8} L^2  \left( s_k\Log{ \frac{2ep}{s_k} } + \Log{ \frac{n}{\rgSp_k\delta} }  \right) \enspace .
	\end{aligned}
\end{align*}
So on $\xsgSp$ this implies that - see \Cref{lem:crossedtermsm}
 $$\rgD_k\abs{ \proscal<\NSk_{\tau_k,+\rgSp_k} - \NSk_{\tau_k,-\rgSp_k},\tSk_{\tau_k,+\rgSp_k} - \tSk_{\tau_k,-\rgSp_k}>}  \leq \frac{\rgSp_k}{4} \norm{ \overline\theta_{\tau_k,+\rgSp_k} - \overline\theta_{\tau_k,-\rgSp_k}}^2.$$
Moreover we have on $\xsgSp$ that - see Equation~\eqref{eq:purenoiseS}
$$\abs{\frac{\rgSp_k}{2} \norm{\NSk_{\tau_k,+\rgSp_k} - \NSk_{\tau_k,-\rgSp_k} }^2 - \sigma^2 s} \leq   4\Cxi L^2 \left(\Log{ \frac{n}{\rgSp_k\delta}} + 2s_k\Log{\frac{2ep}{s_k} } \right) \leq  \sigma^2\tDmsg_{\rgSp_k} \enspace,$$
for $\thresh \geq 4\Cxi$ - note that $\Cxi>0$ is a universal constant. And so on $\xsgSp$, combining the three previous displayed equations implies
\begin{align*}
    \psiDmsg_{\tau_k,\rgSp_k}&\geq  \frac{\rgD_k}{4 \sigma^2} \norm{ \tSk_{\tau_k,+\rgSp_k} - \tSk_{\tau_k,-\rgSp_k}}^2 -  \tDmsg_{\rgSp_k}\\
    &\geq \left(\frac{\kappa}{16} - \thresh\right) \frac{L^2}{\sigma^2}  \left(\Log{ \frac{n}{\rgSp_k\delta}} + s_k\Log{\frac{2ep}{s_k} } \right) > \tDmsg_{\rgSp_k}\enspace,
\end{align*}
since $\kappa > 32\thresh$. And so on $\xsgSp$
$$\TDmsg_{\tau_k,\rgSp_k} = 1.$$
This concludes the proof of the second part of the proposition.

\subsubsection{Proof of \Cref{cor:subgaus}}
Let $\xDense$ and $\xSparse$ be two events such that \Cref{prop:densesg} and \Cref{prop:sparsesg} both hold with probability $1-3\delta$, and write $\xi = \xDense \cap \xsgSp$. From now on, we work on the event $\xi$, which holds with probability $1 - 6\delta$. Define here simply $\tg_k = \tau_k$. Note that by definition of $\rg_k$ in the sub-Gaussian regime:
	\begin{equation*}\label{eq:rgrillesg}
	\rg_k = \left\lbrace
	\begin{aligned}
		\rgD_k \text{ if } s_k \log\left( \frac{ep}{s_k}\right) &> \sqrt{p \Log{\frac{n}{r_k\delta}}}\\
		\rgSp_k \text{ if } s_k\log\left( \frac{ep}{s_k}\right) &\leq \sqrt{p \Log{\frac{n}	{r_k\delta}}}
	\end{aligned}
	\right.
	\end{equation*}
	According to \Cref{th:g}, it is sufficient to prove that $\cA \left( \Theta, \Tg, \cK^*,(\tg_k, \rg_k)_{k\in \cK^*}\right)$ holds.
\begin{enumerate}
	\item {\bf (No false positive):} $\Tg_{l,r} = \TDmsg_{l,r}\lor \TDsg_{l,r} = 0$ for any $(l,r) \in \cG_F \cap \cH_0$.
	by \Cref{prop:densesg} and \Cref{prop:sparsesg}.
	\item {\bf (Significant change-point detection):} for every $k \in  \cK^*$ (see \Cref{eq:Energysg}), we have by definition of $\rg_k$ :
	$$4(\rg_k - 1) \leq r_k.$$
	Now if $s_k \log\left( \frac{ep}{s_k}\right) \geq \sqrt{p\Log{\frac{n}{r_k\delta}}}$, we have $\TD_{\tg_k,\rg_k} = 1$ by \Cref{prop:densesg}, by definition of $\Csto$, and for $\thresh^{(\mathrm{d})}$ as in \Cref{prop:densesg}.

	If $s_k \log\left( \frac{ep}{s_k}\right) \leq \sqrt{p\Log{\frac{n}{r_k\delta}}}$, we have $\TDmsg_{\tg_k,\rg_k} = 1$ by \Cref{prop:sparsesg}, by definition of $\Csto$, and for $\thresh^{(\mathrm{p})}$ as in \Cref{prop:sparsesg}.
\end{enumerate}
\Cref{th:g} ensures that for all $k \in \cK^*$, there exists $k' \in [\hat K]$ such that
	$$|\hat\tau_{k'} - \tau_{k}| \leq \rg_{k} - 1.$$
	This concludes the proof since $4(\rg_{k} - 1) \leq r_k$ for $k \in \cK^*$.


\subsection{ Proof of \Cref{th:borneinf}}
    Let us fix $(r,s) \in [1,n/4] \times [1,p]$. 
    Let $\Delta$ be such that $$r\Delta^2 = \frac{1}{2}\sigma^2\left[s\Log{1+u\frac{\sqrt{p}}{s} \sqrt{\Log{ \frac{n}{r }}}} +  u\Log{\frac{n}{r}}\right],$$
    for some $u \leq \frac{1}{8}$.

    In what follows, we consider any change-point detection method that outputs an estimator $\hat \tau$ of the change-points, associated to a number $\hat K$ of detected change-points, i.e.~the length of $\hat \tau$. We also write $\Prob_\Theta$ for the distribution of the data when the mean parameter or the time series is fixed to a $n\times p$ matrix $\Theta$, i.e.~of $\Theta+\Noise$ where the noise entries $(\Noise_t)_j$ are i.i.d.~and follow $\cN(0,\sigma^2)$ as in \Cref{sec:multi_scale_change_point_estimation_in_gaussian_noise}. Also abusing slightly notations, we write $\Prob_0$ for the distribution of the data when the parameter is constant and equal to $0$. \\
    Consider also any prior $\pi$ over the set of $n\times p$ matrices $\Theta$ such that the number of true change-points over the support of the prior is larger than $1$ - i.e.~the prior puts mass only on problems where more than one change-point occurs. Let $\bar \Prob_{\pi}$ be the corresponding distribution of the data, namely the distribution of the matrix of data when the mean parameter of the time series is the random matrix $\Thetarv \sim \pi$. Otherwise said,  $\bar \Prob_{\pi}$ is the distribution of $\Thetarv + \Noise$ where $\Thetarv \sim \pi$.

    We remind that in our setting $K$ is the number of true change-points in a given problem - which would be either $0$ under $\Prob_0$, or more than $1$ under $\bar \Prob_{\pi}$. If the support of $\pis$ is included in $\cP(r,s)$, then
    \begin{align}
    \underset{\Theta \in \cP(r,s)}{\sup}\Prob_\Theta( \hat K \neq K) &\geq \frac{1}{2}\left(\bar \Prob_{\pi}( \hat K = 0) + \Prob_{0}(\hat K \neq 0)\right) \nonumber\\
    &\geq \frac{1}{2} (1- d_{TV} (\bar \Prob_{\pi}, \Prob_0 )), \label{eq:TVLB}
    \end{align}
   where $d_{TV}$ is the total variation distance. From the Cauchy-Schwarz inequality, we have
    \begin{align}\label{eq:TVLBchi}
    d_{TV} (\bar \Prob_{\pi}, \Prob_0) \leq \frac{1}{2}\sqrt{\chi^2 ( \bar \Prob_{\pi}, \Prob_0 )},
	\end{align}
	where $\chi^2$ is the divergence between probability distributions:
	$$ \chi^2 ( \bar \Prob_{\pi}, \Prob_0 ) = \Esp_{\Prob_0}\left[ \left( \frac{\mathrm{d}\bar \Prob_{\pi}}{\mathrm{d}\Prob_0} - 1\right)^2\right] \enspace .$$
    By a simple computation that can be found for example in \cite{wu2017lecture}
    \begin{align}\label{eq:TVLBchi2}
    \chi^2 ( \bar\Prob_{\pi}, \Prob_0 ) = \Esp_{\Thetarv,\Thetarv'} \left[ e^{\frac{1}{\sigma^2}\proscal<\Thetarv,\Thetarv'>} \right] - 1,
    \end{align}
    where $\Thetarv$ and $\Thetarv'$ are i.i.d. and distributed according to $\pi$, $\proscal<\Thetarv,\Thetarv'> = \Tr (\Thetarv'\Thetarv^T)$ is the standard scalar product, and $\Esp_{\Thetarv,\Thetarv'}$ is the expectation according to $\Thetarv$ and $\Thetarv'$.

    Let us consider the three following cases for the couple $(r,s)$:
    \begin{align*}
    &\mbox{ {\bf Case 1}}: u\Log{\frac{n}{r}} \leq s \Log{1 + u\frac{\sqrt p}{s}\sqrt{\Log{\frac{n}{r}}}}  \quad \mbox{and} \quad s \leq u\sqrt{p \Log{\frac{n}{r}}}, \\
    &\mbox{ {\bf Case 2}}: u\Log{\frac{n}{r}} \leq s \Log{1 + u\frac{\sqrt p}{s}\sqrt{\Log{\frac{n}{r}}}}  \quad\mbox{and} \quad s > u\sqrt{p \Log{\frac{n}{r}}}, \\
    &\mbox{ {\bf Case 3}}: u\Log{\frac{n}{r}} > s \Log{1 + u\frac{\sqrt p}{s}\sqrt{\Log{\frac{n}{r}}}}.
    \end{align*}
    Each case corresponds to the regime of detection of one of the three statistics. The first one corresponds to the Berk-Jones statistic, the second one to the dense statistic and the last one to the partial norm statistic.

    \medskip

    {\bf Case 1} : In that case, $r\Delta^2 \leq \sigma^2 s\Log{4u\frac{p}{s^2}\Log{\frac{n}{r}}}$.
    Let us define a probability distribution on the parameter $\Theta \in \cP(r,s)$.
    For $\zeta = \floor{\frac{n}{r}} - 1$ and $l \in \tilde{\cD}_r= \{1,r+1, 2r+1, \ldots \zeta r +1 \}, $
    define the column vector $v_l =  \sum_{j=l}^{l+r-1}e_j$, where $e_j$ is the $j^{th}$ element of the canonical basis of $\bbR^n$.
    Let  $\Crvs$ be a random variable uniformly distributed in $\{x \in \{0,1\}^p,  |x|_0 = s\}$ and $\Lrv$
    be a random variable independent from $\Crvs$ and uniformly distributed on $\{v_l~:~ l \in \tilde{\cD}_r\}$.
    Let
    $$\Thetarvs = \frac{\Delta}{\sqrt s}\Crvs\Lrv^T \in \bbR^{p\times n},$$
     and $\pis$ be the distribution of the random variable $\Thetarvs$, and $\bar \Prob_{\pis}$ be the corresponding distribution of the data.

Consider two independent copies $\Thetarvs$ and $\Thetarvs'$ that are distributed like $\pis$. The probability that $\Thetarvs$ and $\Thetarvs'$ have the same support is exactly $\frac{1}{\zeta+1}$. Hence, from Equation~\eqref{eq:TVLBchi2}
    \begin{align}
    \chi^2 ( \bar \Prob_{\pis}, \Prob_0 ) = \frac{1}{\zeta+1}\left(\Esp_{\Crvs, \Crvs'} \left[ e^{\frac{r\Delta^2}{s\sigma^2}\proscal<\Crvs, \Crvs'>} -1 \right]\right) \enspace, \label{eq:LBpos}
    \end{align}
    where $\Crvs'$ is an independent copy of $\Crvs$, and $\Esp_{\Crvs, \Crvs'}$ is the expectation according to $\Crvs, \Crvs'$.
    Remark by symmetry that $\proscal<\Crvs, \Crvs'>$  has the same law as $\sum_{i=1}^s \Crvs_i$. Hence
    \begin{align*}
    \Esp_{\Crvs, \Crvs'} \left[ e^{\frac{r\Delta^2}{s\sigma^2}\proscal<\Crvs,\Crvs'>} \right] &= \Esp_{\Crvs} \left[ e^{\frac{r\Delta^2}{s\sigma^2} \sum\limits_{i = 1}^s \Crvs_i }\right]\enspace,
    \end{align*}
    where $\Esp_{\Crvs}$ is the expectation according to $\Crvs$.

    Remark that $(\Crvs_1, \ldots, \Crvs_p)$ has the same distribution as a random sampling without replacement of the list of length $p$ containing $(1, \ldots, 1,0, \ldots, 0)$ - the list containing exactly $s$ times the quantity $1$ and otherwise only $0$. The following lemma allows us to replace the variables $\Crvs_i$ by independent Bernoulli random variables $Z_i \sim \cB(s/p)$.
    \begin{lemma}
    Let $c = (c_1, \ldots, c_p) \in \bbR^p$. We associate to the list $c$ two random sampling processes: (i) the sampling process without replacement $(X_i)_{i= 1\dots s}$ of $s$ elements uniformly on the list $c$ and (ii) the sampling process with replacement $(Z_i)_{i= 1\dots s}$ of $s$ elements uniformly in the list. Then for any convex function $f$,
    $$\Esper{f\left(\sum\limits_{i=1}^s X_i\right)} \leq \Esper{f\left(\sum\limits_{i=1}^s Z_i \right)}\enspace.$$
	\end{lemma}
	The proof of this lemma can be found in \cite{hoeffding1994probability}. Thus, if $(Z_i)_{i = 1 \dots s}$ is an i.i.d sequence of Bernoulli variables with parameter $\frac{s}{p}$ as described above, we obtain
    \begin{align}
    \chi^2 ( \bar \Prob_{\pis}, \Prob_0 ) &\leq \frac{1}{\zeta+1}\left(\Esp_Z \left[ e^{\frac{r\Delta^2}{s\sigma^2} \sum\limits_{i=1}^s Z_i} \right] - 1\right)\label{eq:LBpos2}\\
    &= \frac{1}{\zeta+1}\left[\left( \frac{s}{p} e^{\frac{r\Delta^2}{s\sigma^2}} + 1 - \frac{s}{p} \right)^s - 1\right]
    \leq  \frac{1}{\zeta+1}\left[e^{\frac{s^2}{p} \left( e^{\frac{r\Delta^2}{s\sigma^2}} - 1\right)} - 1\right] \nonumber\\
    &\leq 2\frac{r}{n} e^{\frac{s^2}{p} \left( e^{\Log{4u^2\frac{p}{s^2}\Log{\frac{n}{r}}}}\right)}
    \leq 2\left(\frac{r}{n}\right)^{1-4u^2} \leq 1 \enspace , \label{eq:LB1}
    \end{align}
    where $\Esp_Z$ is the expectation according to the $(Z_i)_i$ and where in the last inequality we used $u \leq 1/3$ and $n \geq 4r$.
    \medskip

    {\bf Case 2} : In that case, $r\Delta^2 \leq \sigma^2 u\sqrt{p\Log{\frac{n}{r}}}$. 
    Let $s_0 = \ceil{u\sqrt{p \Log{\frac{n}{r}}}}$
	and $\Crvd$ be a random variable uniformly distributed in $\{x \in \{0,1\}^p,  |x|_0 = s_0\}$ and $\nu$ be defined as in {\bf Case 1}. Let
	$$\Thetarvd = \frac{\Delta}{\sqrt p} \Crvd \Lrv^T,$$
 let $\pid$ be the distribution of $\Thetarvd$ and $\bar \Prob_{\pid}$ be the associated probability distribution of the data. Doing the same reasoning and similar computations as for {\bf Case 1}, see in particular the steps of Equations~\eqref{eq:LBpos} and~\eqref{eq:LBpos2} - replacing $s$ by $s_0$ and $\Crvs$ by $\Crvd$ - we have
    \begin{align}
        \chi^2(\bar \Prob_{\pid}, \Prob_0) &= \bbE_{\Thetarvd, \Thetarvd'} \left[e^{\frac{1}{\sigma^2}\proscal<\Thetarvd, \Thetarvd'>}\right] - 1
        = \frac{1}{\zeta+1} \bbE_{\Crvd, \Crvd'} \left[e^{\frac{r\Delta^2}{p \sigma^2}\proscal<\Crvd, \Crvd'>}-1\right]
        \leq  \frac{1}{\zeta+1} \left[e^{\frac{s_0^2}{p} \left( e^{\frac{r\Delta^2}{s_0\sigma^2}}-1 \right)} - 1\right] \nonumber
        \\
        &\leq \frac{1}{\zeta+1}e^{2\frac{s_0r\Delta^2}{p\sigma^2}}
        \leq 2\frac{r}{n}e^{4u\log{\frac{n}{r}}} = 2\left(\frac{r}{n}\right)^{1- 4u } \leq 1 \enspace ,\label{eq:LB2}
        \end{align}
where $\bbE_{\Thetarvd, \Thetarvd'}$ is the expectation according to $\Thetarvd, \Thetarvd'$ (where $\Thetarvd'$ is an independent copy of $\Thetarvd$) and where $\bbE_{\Crvd, \Crvd'} $ is the expectation according to $\Crvd, \Crvd'$ (where $\Crvd'$ is an independent copy of $\Crvd$), and where in the last step we used $u \leq 1/8$ and $n \geq 4r$.
        \medskip

    {\bf Case 3} : In that case, $r\Delta^2 \leq u\Log{\frac{n}{r}}$.
    Let $\Crvm = (1,0,0,\dots,0)$ be the vector with $0$ entries except the first one. Let $\nu$ be the random vector defined as in {\bf Case 1}. Let $$\Thetarvm = \Delta \Crvm \Lrv^T,$$
    and $\pim$ be the distribution of the random variable $\Thetarvm$ - and $\bar \Prob_{\pim}$ be the associated probability distribution of the data. Doing the same reasoning as in {\bf Case 1} - see in particular the step of Equation~\eqref{eq:LBpos} - replacing $\Crvs$ by $\Crvm$ and $s$ by $1$ - for the prior $\pim$, we obtain
    \begin{align}\label{eq:LB3}
        \chi^2(\bar \Prob_{\pim}, \Prob_0) &= \bbE_{\Thetarvm, \Thetarvm'} \left[e^{\frac{1}{\sigma^2}\proscal<\Thetarvm, \Thetarvm'>}\right] - 1
        = \frac{1}{\zeta+1}e^{\frac{r\Delta^2}{\sigma^2}}
        \leq 2\frac{r}{n}e^{u\Log{\frac{n}{r}}}
        \leq 2\left(\frac{r}{n}\right)^{1-u}
        \leq 1 \enspace,
        \end{align}
    where $\bbE_{\Thetarvm, \Thetarvm'}$ is the expectation according to $\Thetarvm, \Thetarvm'$ (where $\Thetarvm'$ is an independent copy of $\Thetarvm$) and where in the last step we used $n \geq 4r$ and $u \leq 1/2$.

    \medskip

    Thus, in all cases - combining Equations~\eqref{eq:TVLB} and~\eqref{eq:TVLBchi} with Equations~\eqref{eq:LB1}, \eqref{eq:LB2} and~\eqref{eq:LB3} - we obtain in all three cases
    \begin{align*}
        \underset{\Theta \in \cP(r,s)}{\sup}\Prob_\Theta( \hat K \neq K) \geq \frac{1}{4} \enspace .
    \end{align*}
        and this concludes the proof.

\subsection{Proofs for covariance and nonparametric change-point detection}

\begin{proof}[Proof of \Cref{prop:covariance}]
	Consider an $r$-sample $(z_1,\ldots z_r)$ with mean zero and covariance matrix $\Sigma$ and Orlicz norm $B$. Koltchinskii and Lounici~\cite{koltchinskii2017concentration} have proved that, for any $x>0$, the empirical covariance matrix $\widehat{\Sigma}= r^{-1}(\sum_{i=1}^r z_iz_i^T)$ satistifies
	\[
	\|\widehat{\Sigma}-\Sigma\|_{op}\leq c'B^2 \left[ \sqrt{\frac{p}{r}}+ \frac{p}{r}+ \sqrt{\frac{x}{r}}+  \frac{x}{r}\right] \ ,
	\]
	with probability higher than $1-\exp(-x)$. Here $c'$ is a suitable positive constant. Considering a union bound over all $(l,r)\in\cG_D$ such that $\Sigma_t$ is constant over $[l-r,l+r)$, we have, with probability higher than $1-\delta/2$, that simultaneously on all such $r \in \cR$ and $l\in \cD_r$,
	\[
	\|\widehat{\Sigma}_{l,r}-\widehat{\Sigma}_{l,-r}\|_{op}\leq
	\|\widehat{\Sigma}_{l,r}-\Sigma_{l}\|_{op}+ \|\widehat{\Sigma}_{l,-r}-\Sigma_{l}\|_{op}
	\leq 8c'B^2 \left[ \sqrt{\frac{p}{r}}+ \frac{p}{r}+ \sqrt{\frac{\log(2n/(r\delta))}{r}}+  \frac{\log(2n/(r\delta))}{r}\right] \ ,
	\]
 where the constant $8$ comes from the union bound on all elements of the grid.
As a consequence, the FWER of the multiple testing collection is at most $\delta/2$ provided that we choose $c_0\leq 8c'$. 

Conversely, consider any high-energy change-point $\tau_k$. Let $\overline{r}_k$ be the smallest radius $r\in\cR$ such that
	\beq\label{eq:high_energy_covariance_2}
	 r \|\Sigma_{\tau_{k}}- \Sigma_{\tau_{k-1}}\|^2_{op}\geq 0.25 c_1 B^4 \left[\left(p+ \log\left(\frac{2n}{r\delta}\right)\right) \wedge  r \right]\ .
	\eeq
	and consider  the closest location $l\in \cD_r$ of $\tau_k$ so that $|l-\tau_k|\leq r/2$. To ease the notation, we still write $r$ for $\overline{r}_k$.
	 Without loss of generality, we assume that $l\geq \tau_k$.
	Let us decompose the statistic $\widehat{\Sigma}_{l,-r}= \frac{r-l+\tau_k}{r}\widehat{\Sigma}_{\tau_k,-(r-l+\tau_k)}+ \frac{l-\tau_k}{r}\widehat{\Sigma}_{l,-(l-\tau_k)}$. Since $r\leq r_k/2$, $\Sigma_{t}$ is constant over $[l-r,\tau_k)$ and over $[\tau_k,l+r)$. Then, we apply  three times the deviation inequality of Koltchinskii and Lounici~\cite{koltchinskii2017concentration} to get
	\begin{eqnarray*}
		\|\widehat{\Sigma}_{l,r}-\widehat{\Sigma}_{l,-r}\|_{op}
		&\geq &\frac{r-l+\tau_k}{r}\|\Sigma_{\tau_k}-\Sigma_{\tau_{k-1}}\|_{op}
		- \|\widehat{\Sigma}_{l,r} -\Sigma_{\tau_k}\|_{op}  \\ &&
		- \frac{l-\tau_k}{r}\|\widehat{\Sigma}_{l,-(l-\tau_k)} -\Sigma_{\tau_{k}}\|_{op}
		-  \frac{r-l+\tau_k}{r} \|\widehat{\Sigma}_{\tau_{k},-(r-l+\tau_k)} -\Sigma_{\tau_{k-1}}\|_{op} \\
		&\geq & \frac{1}{2}\|\Sigma_{\tau_k}-\Sigma_{\tau_{k-1}}\|_{op} -c''B^2 \left[ \sqrt{\frac{p}{r}}+ \frac{p}{r}+ \sqrt{\frac{\log(2n/(r\delta))}{r}}+  \frac{\log(2n/(r\delta))}{r}\right] \ ,
	\end{eqnarray*}
	with probability higher than $1-0.5\delta[r/(2n)]^2$.  As a consequence, we have $T_{l,r}=1$ provided that
	\[
		\|\Sigma_{\tau_k}-\Sigma_{\tau_{k-1}}\|_{op}\geq 2(c''+c_0)B^2 \left[ \sqrt{\frac{p}{r}}+ \frac{p}{r}+ \sqrt{\frac{\log(2n/(r\delta))}{r}}+  \frac{\log(2n/(r\delta))}{r}\right] \enspace .
	\]
	Since $\|\Sigma_{\tau_{k}}- \Sigma_{\tau_{k-1}}\|_{op}\leq 2B^2$ and if we choose $c_1\geq 17\vee 32(c''+c_0)$, the bound~\eqref{eq:high_energy_covariance_2} is achievable only if $r\geq p+ \log(2n/(r\delta))$ and we deduce from~\eqref{eq:high_energy_covariance_2} that $T_{l,r}=1$.

	Taking a union bound over all high-energy change-points, we deduce from Theorem~\ref{th:g} that, with probability higher than $1-\delta$, $\widehat{\tau}$ achieves ({\bf NoSp}) and {\bf detects} all high-energy change-points. Besides, the localization error~\eqref{eq:error_localization} is a consequence of the definition~\eqref{eq:high_energy_covariance_2} together with Theorem~\ref{th:g}.
	\end{proof}

\begin{proof}[Proof of Proposition~\ref{prp:lower:covariance}]
As in the proof of Theorem~\ref{th:borneinf}, we only consider a specific setting where one aims at testing $K=0$ with $\Sigma_1=I_p$ versus  $K=2$ with  $\tau_1\in (n/4;3n/4)$, $\tau_2=\tau_1+r$,
$\Sigma_{1}=\Sigma_{\tau_2}=I_p$ and $\Sigma_{\tau_1}=I_p + \zeta u u^{T}$ for some unknown  unit vector $u$ in $\mathbb{R}^p$. Obviously, we have $r_1=r_2=r$ and $\|\Sigma_{\tau_1}-\Sigma_{\tau_0}\|_{op}= \|\Sigma_{\tau_2}-\Sigma_{\tau_1}\|_{op}= \zeta$ so that it suffices to prove that the sum of the type I and type II error probabilities of any test of these hypotheses is bounded away from zero.
We consider two subcases:

\noindent 
{\bf Case 1}: $\zeta \leq c'\sqrt{p/r}\wedge \frac{1}{\sqrt{2}}$. Then, we focus on the specific alternative hypothesis where $\tau_1= \lfloor n/2 \rfloor$ and $\tau_2=\tau_1+r$, so that the problem reduces exactly to testing whether the covariance matrix $\Sigma$ of a $r$-sample satisfies $\Sigma=I_p$ or whether $\Sigma=I_p+ \zeta u u^T$. This hypothesis testing problem for covariance matrices is well understood. In particular, one can deduce from Theorem 5.1 in~\cite{berthet2013optimal} that, as soon as $\zeta\leq c'[\sqrt{p/r}\wedge  1]$, for some $c'$ sufficiently small, one has
\[
	\underset{\hat \tau}{\inf}\underset{\Theta \in \bar{\cP}(r,\zeta)}{\sup}\Prob_\Theta( \hat K \neq K) \geq \frac{1}{4}\enspace .
\]
{\bf Case 2}: $\zeta \leq c'\sqrt{\log(n/r)/r}\wedge 1/\sqrt{2}$. Here, we consider another specific class of alternative hypotheses where we fix $u=(1,0,\ldots, 0)$ but $\tau_1$ can take different values, i.e. $\tau_1\in \{\lfloor n/4\rfloor,\lfloor n/4\rfloor+ r,\ldots, \lfloor n/4\rfloor + r\lfloor n/2r \rfloor\}$. It turns out that this is equivalent to a univariate variance testing problem where one observes $q=\lfloor n/(2r)\rfloor$ samples of size $r$ with distributions $\cN(0,\sigma_1^2)$, \ldots, $\cN(0,\sigma_q^2)$. Under the null, we have $\sigma_1=\sigma_2 =\ldots = \sigma_q=1$. Under the alternative, for some $j\in [q]$, we have $\sigma_j=\sqrt{1+\zeta}$ and  $\sigma_l= 1$ for $l\neq j$.
		For $j=1,\ldots, q$, write $\mathbb{P}_j$ for the distributions of the $j$-th sample of size $r$ when $\sigma^2_j=1+\zeta$ and $\sigma_l=1$ for $l\neq j$. Besides, we write $L_j$ for the corresponding likelihood ratio with the null distribution $\mathbb{P}_0$. Then, the mixture distribution is defined as $\overline{\mathbb{P}}=\frac{1}{q}\sum_{j=1}^q \mathbb{P}_j$ whereas $\overline{L}$ stands for the mean likelihood ratio. Following the classical path of Le Cam's method we obtain that, for any test $T$,
		\[
		\P_0[T=1]+ \sup_{j=1,\ldots, q}\P_j[T=0]\geq \P_0[T=1]+ \overline{\P}[T=0]\geq 1- \|\P_0-\overline{\P}\|_{TV}\ ,
		\]
		where  $\|.\|_{TV}$ is the total variation norm.  Using Cauchy-Schwarz inequality, we bound this total variation distance between the covariates
		\[
			\|\P_0- \overline{\P}\|_{TV} \leq  \mathbb{E}_0\left[\overline{L}^2\right]-1 = \frac{1}{q}\left(\mathbb{E}_0\left[L_i^2\right]- 1\right) = \frac{1}{q}\left[(1-\zeta^2)^{-r/2}-1 \right]\leq \frac{1}{q}\left[e^{r\zeta^2}-1 \right] \ ,
		\]
		since $\zeta\in (0,1/2)$. As a consequence, we derive that
		$\|\P_0- \overline{\P}\|_{TV} \leq 1/4$  as long as $r \zeta^2 \leq c' \log(q)\wedge 1$. The result follows.
	\end{proof}

\begin{proof}[Proof of Proposition~\ref{prop:non_parametric}]
	The proof is based on an application of Dvoretzky–Kiefer–Wolfowitz (DKW) inequality~\cite{book_concentration} together with an union bound. For a $q$ sample of a univariate distribution with empirical distribution function $\widehat{F}$ and true distribution function $F$, DKW inequality ensures that
	\[
	\P\left[\|\widehat{F}-F\|_{\infty} \geq \sqrt{\frac{x}{2q}}\right]\leq 2 e^{-x}.
	\]
	Applying two-times DKW inequality to each statistic $T_{l,r}$ such that no-change-point occurs on $(l-r,l+r)$, we deduce that, setting $c_1$ sufficiently larger, the FWER of $(T_{l,r})$ is at most $\delta/2$ by summing the probabilities over all scales $r\in \cR$ and by a union bound on all $l \in \cD_r$.

	Turning to the high-energy change points, we consider $\tau_k$ satisfying~\eqref{eq:high_energy_non_param}.	
	Let $\overline{r}_k$ be the smallest radius $r\in\cR$ such that
	\beq\label{eq:high_energy_non_param2}
	 r  \|F_{\tau_{k}}- F_{\tau_{k-1}}\|^2_{\infty}\geq 0.25c_1
	 \frac{\log\left(\frac{n}{r\delta}\right)}{m}\ ,
	\eeq
	and consider  the closest location $l\in \cD_r$ of $\tau_k$ so that $|l-\tau_k|\leq r/2$ and $2r\leq r_k$.   To ease the notation, we still write $r$ for $\overline{r}_k$. As in the proof of Proposition~\ref{prop:covariance}, we decompose  the statistic
	\[
	\sum_{t=l}^{l+r-1}\widehat{F}_{t}- \sum_{t=l-r}^{l-1}\widehat{F}_{t}= \sum_{t=l}^{l+r-1}\widehat{F}_{t}- \sum_{t=l-r}^{\tau_k-1}\widehat{F}_{t}-\sum_{t=\tau_k}^{l-1}\widehat{F}_{t},
	\]
	and apply DKW inequality to each  of three sums. Taking the union bound over all possible $T_{l,r}$ we deduce that, with probability higher than $1-\delta/2$
	\[
	r^{-1}	\|\sum_{t=l}^{l+r-1}\widehat{F}_{t}- \sum_{t=l-r}^{l-1}\widehat{F}_{t}\|_{\infty}\geq \frac{1}{2}\|F_{\tau_k}-F_{\tau_{k-1}}\|_{\infty}	- c''\sqrt{\frac{\log(4n/r\delta)}{mr}}\ ,
	\]
	so that in view of Condition~\eqref{eq:high_energy_non_param2} implies that $T_{l,r}=1$. Applying Theorem~\ref{th:g} allows us to conclude.
	\end{proof}

\begin{proof}[Proof of Proposition~\ref{prop:lower_non_parametric}]
As in the proof of Proposition~\ref{prp:lower:covariance}, we focus on a simpler testing problem. Write $U$ for the cumulative distribution function of the uniform distribution on $[0,1]$, i.e.~$U(x)=x$ for any $x\in [0,1]$. Given $\zeta\in (0,1/4)$, we define the cumulative distribution function $U_{\zeta}$ by $U_{\zeta}(x)= (1+2\zeta)x$ for $x\in [0,1/2]$ and  $U_{\zeta}(x)= (1/2+\zeta)+ (1-2\zeta)(x-1/2)$ for $x\in[1/2,1]$. Note that $\|U_{\zeta}-U\|_{\infty}=\zeta$.

We focus on a testing problem where, under the null, $F_t=U$ for all $t=1,\ldots, n$, whereas under the alternative there exists $\tau_1\in \{\lfloor n/4\rfloor , \lfloor n/4\rfloor+ r ,\ldots, \lfloor n/4\rfloor+ (r-1) \lfloor n/(2r)\rfloor\}$ such that $F_t=U_\zeta$ for $t=\tau_1,\ldots, \tau_{1}+r-1$ and $F_t=U$ otherwise. Defining $q= \lfloor n/(2r)\rfloor$, we observe that this amounts to testing whether $q$ samples of size $rm$ are distributed according the null distribution or whether exactly one of them is distributed according to $U_\zeta$. Arguing again in the proof of Proposition~\ref{prp:lower:covariance}, we only need to bound the total variation distance between the distribution $\mathbb{P}_0$ under the null and the mixture distribution $q^{-1}\sum_{j=1}^q\mathbb{P}_j$ of the $q$ possible alternatives - here $\mathbb{P}_0 = \otimes_{k=1}^q U^{\otimes (rm)}$ is the distribution of the samples when $F_t = U$ and $\mathcal P_j= \Big[\otimes_{k=1}^{j-1} U^{\otimes (rm)}\Big] \otimes U_\zeta^{\otimes (rm)}\otimes \Big[\otimes_{k=j+1}^{q} U^{\otimes (rm)}\Big]$, is for $j \geq 1$ the distribution of the samples when $F_t = U$ except for $t \in [jr, (j+1)r)$, in which case $F_t = U_\zeta$.

Let $z$ be a uniform random variable over $[0,1]$ and $w$ be an independent Bernoulli random variable with parameter $1/2$. Then, one easily checks that $z/2+ w/2$ is uniformly distributed on $[0,1]$. If $w$ is a Bernoulli random variable with parameter $1/2- 2\zeta$, then one easily checks that the cumulative distribution function of  $z/2+ w/2$ is $F_{\zeta}$. As a consequence, by a standard data-processing inequality~\cite{wu2017lecture}, one derives that 
\[
\|\mathbb{P}_0- q^{-1}\sum_{j=1}^q\mathbb{P}_j\|_{TV}\leq \|\widetilde{\mathbb{P}}_0- q^{-1}\sum_{j=1}^q\widetilde{\mathbb{P}}_j\|_{TV}\ ,
\]
where under $\widetilde{\mathbb{P}}_0$ one observes $q$ independent Binomial random variables with parameter $(mr, 1/2)$, whereas under $\widetilde{\mathbb{P}}_j$, the $j$-th observation follows a Binomial distribution with parameter $(mr, 1/2-2\zeta)$. Using Cauchy-Schwarz inequality, we upper bound the square of the  total variation distance by the $\chi^2$ distance and then compute it explicitly. This leads us to
\[
	\|\widetilde{\mathbb{P}}_0- q^{-1}\sum_{j=1}^q\widetilde{\mathbb{P}}_j\|^2_{TV}\leq \frac{1}{q}\left[ (1+16\zeta^2)^{rm} - 1\right]\ ,
\]
which is smaller than $1/4$ provided that $16rm \zeta^2\leq \log(q/4+1)$. If we choose $c'$ small enough in the statement of the proposition, this last condition holds and the result follows.
\end{proof}

\bibliographystyle{plain}

\bibliography{bibliographie}

\begin{thebibliography}{10}

\bibitem{Arlot2019}
Sylvain Arlot, Alain Celisse, and Zaid Harchaoui.
\newblock A kernel multiple change-point algorithm via model selection.
\newblock {\em J. Mach. Learn. Res.}, 20, 2019.

\bibitem{baranowski2019narrowest}
Rafal Baranowski, Yining Chen, and Piotr Fryzlewicz.
\newblock Narrowest-over-threshold detection of multiple change points and
  change-point-like features.
\newblock {\em Journal of the Royal Statistical Society: Series B (Statistical
  Methodology)}, 81(3):649--672, 2019.

\bibitem{berthet2013optimal}
Quentin Berthet and Philippe Rigollet.
\newblock Optimal detection of sparse principal components in high dimension.
\newblock {\em Annals of Statistics}, 41(4):1780--1815, 2013.

\bibitem{book_concentration}
St{\'e}phane Boucheron, G{\'a}bor Lugosi, and Pascal Massart.
\newblock {\em {Concentration inequalities}}.
\newblock Oxford University Press, Oxford, 2013.
\newblock A nonasymptotic theory of independence, With a foreword by Michel
  Ledoux.

\bibitem{chan2017multi}
Hock-Peng Chan and Hao Chen.
\newblock Multi-sequence segmentation via score and higher-criticism tests.
\newblock {\em arXiv preprint arXiv:1706.07586}, 2017.

\bibitem{chan2015optimal}
Hock~Peng Chan, Guenther Walther, et~al.
\newblock Optimal detection of multi-sample aligned sparse signals.
\newblock {\em Annals of Statistics}, 43(5):1865--1895, 2015.

\bibitem{cho2021data}
Haeran Cho and Claudia Kirch.
\newblock Data segmentation algorithms: Univariate mean change and beyond.
\newblock {\em Econometrics and Statistics}, 2021.

\bibitem{chu2019asymptotic}
Lynna Chu and Hao Chen.
\newblock Asymptotic distribution-free change-point detection for multivariate
  and non-euclidean data.
\newblock {\em The Annals of Statistics}, 47(1):382--414, 2019.

\bibitem{dette2018relevant}
Holger Dette, Josua G{\"o}smann, et~al.
\newblock Relevant change points in high dimensional time series.
\newblock {\em Electronic Journal of Statistics}, 12(2):2578--2636, 2018.

\bibitem{dette2020estimating}
Holger Dette, Guangming Pan, and Qing Yang.
\newblock Estimating a change point in a sequence of very high-dimensional
  covariance matrices.
\newblock {\em Journal of the American Statistical Association}, pages 1--11,
  2020.

\bibitem{jin2004}
David Donoho and Jiashun Jin.
\newblock {Higher criticism for detecting sparse heterogeneous mixtures}.
\newblock {\em Ann. Statist.}, 32(3):962--994, 2004.

\bibitem{EnikeevaHarchaoui2019}
Farida Enikeeva and Zaid Harchaoui.
\newblock High-dimensional change-point detection under sparse alternatives.
\newblock {\em Ann. Statist.}, 47(4):2051--2079, 2019.

\bibitem{frick2014multiscale}
Klaus Frick, Axel Munk, and Hannes Sieling.
\newblock Multiscale change point inference.
\newblock {\em Journal of the Royal Statistical Society: Series B (Statistical
  Methodology)}, 76(3):495--580, 2014.

\bibitem{fryzlewicz2014wild}
Piotr Fryzlewicz.
\newblock Wild binary segmentation for multiple change-point detection.
\newblock {\em The Annals of Statistics}, 42(6):2243--2281, 2014.

\bibitem{fryzlewicz2018tail}
Piotr Fryzlewicz.
\newblock Tail-greedy bottom-up data decompositions and fast multiple
  change-point detection.
\newblock {\em The Annals of Statistics}, 46(6B):3390--3421, 2018.

\bibitem{Garreau2018}
Damien Garreau and Sylvain Arlot.
\newblock Consistent change-point detection with kernels.
\newblock {\em Electron. J. Stat.}, 12(2):4440--4486, 2018.

\bibitem{gibberd2017multiple}
Alex~J Gibberd and Sandipan Roy.
\newblock Multiple changepoint estimation in high-dimensional gaussian
  graphical models.
\newblock {\em arXiv preprint arXiv:1712.05786}, 2017.

\bibitem{gretton2012kernel}
Arthur Gretton, Karsten~M Borgwardt, Malte~J Rasch, Bernhard Sch{\"o}lkopf, and
  Alexander Smola.
\newblock A kernel two-sample test.
\newblock {\em The Journal of Machine Learning Research}, 13(1):723--773, 2012.

\bibitem{hoeffding1994probability}
Wassily Hoeffding.
\newblock Probability inequalities for sums of bounded random variables.
\newblock In {\em The Collected Works of Wassily Hoeffding}, pages 409--426.
  Springer, 1994.

\bibitem{hu2021sparsity}
Shouri Hu, Jingyan Huang, Hao Chen, and Hock~Peng Chan.
\newblock Sparsity likelihood for sparse signal and change-point detection.
\newblock {\em arXiv preprint arXiv:2105.07137}, 2021.

\bibitem{jirak2015uniform}
Moritz Jirak.
\newblock Uniform change point tests in high dimension.
\newblock {\em The Annals of Statistics}, 43(6):2451--2483, 2015.

\bibitem{jula2021multiscale}
Laura Jula~Vanegas, Merle Behr, and Axel Munk.
\newblock Multiscale quantile segmentation.
\newblock {\em Journal of the American Statistical Association}, pages 1--14,
  2021.

\bibitem{koltchinskii2017concentration}
Vladimir Koltchinskii and Karim Lounici.
\newblock Concentration inequalities and moment bounds for sample covariance
  operators.
\newblock {\em Bernoulli}, 23(1):110--133, 2017.

\bibitem{kovacs2020seeded}
Solt Kov{\'a}cs, Housen Li, Peter B{\"u}hlmann, and Axel Munk.
\newblock Seeded binary segmentation: A general methodology for fast and
  optimal change point detection.
\newblock {\em arXiv preprint arXiv:2002.06633}, 2020.

\bibitem{kovacs2020optimistic}
Solt Kov{\'a}cs, Housen Li, Lorenz Haubner, Axel Munk, and Peter B{\"u}hlmann.
\newblock Optimistic search strategy: Change point detection for large-scale
  data via adaptive logarithmic queries.
\newblock {\em arXiv preprint arXiv:2010.10194}, 2020.

\bibitem{Laurent00}
B.~Laurent and P.~Massart.
\newblock {Adaptive estimation of a quadratic functional by model selection}.
\newblock {\em Annals of Statistics}, 28(5):1302--1338, 2000.

\bibitem{li2019multiscale}
Housen Li, Qinghai Guo, Axel Munk, et~al.
\newblock Multiscale change-point segmentation: Beyond step functions.
\newblock {\em Electronic Journal of Statistics}, 13(2):3254--3296, 2019.

\bibitem{liu2019minimax}
Haoyang Liu, Chao Gao, and Richard~J Samworth.
\newblock Minimax rates in sparse, high-dimensional changepoint detection.
\newblock {\em arXiv preprint arXiv:1907.10012}, 2019.

\bibitem{moscovich2016exact}
Amit Moscovich, Boaz Nadler, and Clifford Spiegelman.
\newblock On the exact berk-jones statistics and their $ p $-value calculation.
\newblock {\em Electronic Journal of Statistics}, 10(2):2329--2354, 2016.

\bibitem{nemirovskiy1985nonparametric}
AS~Nemirovskiy.
\newblock Nonparametric estimation of smooth regression function.
\newblock {\em Soviet Journal of Computer and Systems Sciences}, 23(6):1--11,
  1985.

\bibitem{Niu2016}
Yue~S Niu, Ning Hao, and Heping Zhang.
\newblock Multiple change-point detection: A selective overview.
\newblock {\em Statistical Science}, 31(4):611--623, 2016.

\bibitem{olshen2004circular}
Adam~B Olshen, ES~Venkatraman, Robert Lucito, and Michael Wigler.
\newblock Circular binary segmentation for the analysis of array-based dna copy
  number data.
\newblock {\em Biostatistics}, 5(4):557--572, 2004.

\bibitem{padilla2019optimal1}
Oscar Hernan~Madrid Padilla, Yi~Yu, Daren Wang, and Alessandro Rinaldo.
\newblock Optimal nonparametric change point detection and localization.
\newblock {\em arXiv preprint arXiv:1905.10019}, 2019.

\bibitem{rinaldo2021localizing}
Alessandro Rinaldo, Daren Wang, Qin Wen, Rebecca Willett, and Yi~Yu.
\newblock Localizing changes in high-dimensional regression models.
\newblock In {\em International Conference on Artificial Intelligence and
  Statistics}, pages 2089--2097. PMLR, 2021.

\bibitem{rudelson2013hanson}
Mark Rudelson and Roman Vershynin.
\newblock Hanson-wright inequality and sub-gaussian concentration.
\newblock {\em Electronic Communications in Probability}, 18, 2013.

\bibitem{truong2020selective}
Charles Truong, Laurent Oudre, and Nicolas Vayatis.
\newblock Selective review of offline change point detection methods.
\newblock {\em Signal Processing}, 167:107299, 2020.

\bibitem{vershynin}
Roman Vershynin.
\newblock {\em High-dimensional probability: An introduction with applications
  in data science}, volume~47.
\newblock Cambridge university press, 2018.

\bibitem{verzelenoptimal_change_point}
Nicolas Verzelen, Magalie Fromont, Matthieu Lerasle, and Patricia
  Reynaud-Bouret.
\newblock {Optimal Change-Point Detection and Localization}.
\newblock {\em arXiv preprint arXiv:2010.11470}, 2020.

\bibitem{Wald1945}
Abraham Wald.
\newblock Sequential tests of statistical hypotheses.
\newblock {\em The annals of mathematical statistics}, 16(2):117--186, 1945.

\bibitem{wang2017optimal}
Daren Wang, Yi~Yu, and Alessandro Rinaldo.
\newblock Optimal covariance change point localization in high dimension.
\newblock {\em arXiv preprint arXiv:1712.09912}, 2017.

\bibitem{wang2018optimal}
Daren Wang, Yi~Yu, and Alessandro Rinaldo.
\newblock Optimal change point detection and localization in sparse dynamic
  networks.
\newblock {\em arXiv preprint arXiv:1809.09602}, 2018.

\bibitem{wang2020univariate}
Daren Wang, Yi~Yu, and Alessandro Rinaldo.
\newblock Univariate mean change point detection: {P}enalization, {CUSUM} and
  optimality.
\newblock {\em Electron. J. Stat.}, 14(1):1917--1961, 2020.

\bibitem{wang2019localizing}
Daren Wang, Yi~Yu, Alessandro Rinaldo, and Rebecca Willett.
\newblock Localizing changes in high-dimensional vector autoregressive
  processes.
\newblock {\em arXiv preprint arXiv:1909.06359}, 2019.

\bibitem{wang2020dating}
Runmin Wang and Xiaofeng Shao.
\newblock Dating the break in high-dimensional data.
\newblock {\em arXiv preprint arXiv:2002.04115}, 2020.

\bibitem{wang2019inference}
Runmin Wang, Stanislav Volgushev, and Xiaofeng Shao.
\newblock Inference for change points in high dimensional data.
\newblock {\em arXiv preprint arXiv:1905.08446}, 2019.

\bibitem{Wang2018}
Tengyao Wang and Richard~J. Samworth.
\newblock High dimensional change point estimation via sparse projection.
\newblock {\em J. R. Stat. Soc. Ser. B. Stat. Methodol.}, 80(1):57--83, 2018.

\bibitem{wu2017lecture}
Yihong Wu.
\newblock Lecture notes for ece598yw: Information-theoretic methods for
  high-dimensional statistics, 2017.

\bibitem{yu2017finite}
Mengjia Yu and Xiaohui Chen.
\newblock Finite sample change point inference and identification for
  high-dimensional mean vectors.
\newblock {\em Journal of the Royal Statistical Society: Series B (Statistical
  Methodology)}, 83(2):247--270, 2021.

\end{thebibliography}
\end{document}